\documentclass[11pt,oneside,english]{amsart}
\RequirePackage{amsthm,amsmath,mathtools}

\usepackage[T1]{fontenc}
\usepackage{geometry}
\usepackage{caption,subcaption,multirow,enumitem}
\usepackage{scalefnt}
\usepackage[english]{babel}
\usepackage{float,graphicx,verbatim}
\usepackage{amstext}
\usepackage{amssymb}
\usepackage{amsmath}
\usepackage[numbers]{natbib}
\RequirePackage[colorlinks,citecolor=blue,urlcolor=blue]{hyperref}

\RequirePackage{amsmath,mathtools}%,amsthm

\numberwithin{equation}{section}
\numberwithin{figure}{section}
\theoremstyle{plain}
\newtheorem{lyxalgorithm}{\protect\algorithmname}
\theoremstyle{plain}
\newtheorem*{lyxalgorithm*}{\protect\algorithmname}
\theoremstyle{plain}
\newtheorem{thm}{\protect\theoremname}
\theoremstyle{remark}
\newtheorem*{rem*}{\protect\remarkname}
\theoremstyle{remark}
\newtheorem{rem}[]{\protect\remarkname}
\theoremstyle{plain}
\newtheorem*{asm*}{\protect\assumptionname}
\theoremstyle{plain}

\theoremstyle{plain}
\newtheorem{lem}[thm]{\protect\lemmaname}
\theoremstyle{plain}
\newtheorem{cor}[thm]{\protect\corollaryname}
\theoremstyle{plain}
\newtheorem{prop}[thm]{\protect\propositionname}
\theoremstyle{definition}
\newtheorem*{example*}{\protect\examplename}

\usepackage{tikz}
\usetikzlibrary{calc}

\usepackage{algorithm,algpseudocode,bbm}%,multirow}
\newcommand{\1}{\mathbbm{1}}
\newcommand{\eqd}{\overset{d}{=}}

\newcommand{\C}{\mathcal{C}}
\newcommand{\D}{\mathcal{D}}

\newcommand{\W}{\mathcal{W}}

\newcommand{\E}{\mathbb{E}}
\newcommand{\V}{\mathbb{V}}
\newcommand{\p}{\mathbb{P}}
\newcommand{\Z}{\mathbb{Z}}
\newcommand{\R}{\mathbb{R}}
\newcommand{\N}{\mathbb{N}}

\newcommand{\ov}[1]{\overline{#1}}
\newcommand{\un}[1]{\underline{#1}}

\newcommand{\cl}[1]{\left\lceil{#1}\right\rceil}
\newcommand{\nf}[1]{{\normalfont#1}}
\newcommand{\tsqrt}[1]{{\textstyle\sqrt{#1}}}

\newcommand{\cid}{\overset{d}{\to}}

\newcommand{\sgn}{\text{\normalfont sgn}}
\newcommand{\Lip}{\text{\normalfont Lip}}
\newcommand{\locLip}{\text{\normalfont locLip}}
\newcommand{\BT}{\text{\normalfont BT}}
\newcommand{\BE}{\mathrm{BE}}

\newcommand{\U}{\text{\normalfont U}}
\newcommand{\Poi}{\text{\normalfont Poi}}

\newcommand{\MC}{\text{\normalfont MC}}
\newcommand{\ML}{\text{\normalfont ML}}

\usepackage{pgfplots} % Graphs

\makeatother
  \providecommand{\algorithmname}{Algorithm}
  \providecommand{\assumptionname}{Assumption}
  \providecommand{\examplename}{Example}
  \providecommand{\lemmaname}{Lemma}
  \providecommand{\propositionname}{Proposition}
  \providecommand{\remarkname}{Remark}
\providecommand{\corollaryname}{Corollary}
\providecommand{\theoremname}{Theorem}

\makeatletter
\@namedef{subjclassname@2020}{%
	\textup{2020} Mathematics Subject Classification}
\makeatother

\begin{document}

\title[Stick-breaking simulation with small-jump Gaussian approximation]{Simulation of the drawdown and its duraton in L\'evy models via 
	stick-breaking Gaussian approximation}

\author[Gonz\'alez C\'azares and Mijatovi\'c]{Jorge Gonz\'{a}lez C\'{a}zares \and Aleksandar Mijatovi\'{c}}

\address{Department of Statistics, University of Warwick \and 
	The Alan Turing Institute, UK}

\email{jorge.gonzalez-cazares@warwick.ac.uk}
\email{a.mijatovic@warwick.ac.uk}

\begin{abstract}
We develop a computational method for expected functionals of  the 
drawdown and its duration in exponential L\'evy models. 
It is based on a novel simulation algorithm for the joint law 
of the state, supremum and time the supremum is attained of 
the Gaussian approximation of a general L\'evy process. 
We bound the bias for various locally Lipschitz and 
discontinuous payoffs arising in applications and analyse 
the computational complexities of the corresponding 
Monte Carlo and multilevel Monte Carlo estimators. 
%based on our algorithm. 
Monte Carlo methods for L\'evy processes (using Gaussian approximation) have been analysed 
for Lipschitz payoffs, in which case 
%Our algorithm outperforming existing Monte Carlo methodologies for all L\'evy models, with 
the computational complexity 
of our algorithm is up to two orders of magnitude smaller 
when the jump activity is high. 
At the core of our approach are bounds 
on certain Wasserstein %and Kolmogorov 
distances, obtained via the novel SBG coupling 
between a L\'evy process and its Gaussian 
approximation. Numerical performance, based on the 
implementation in~\cite{Jorge_GitHub3}, exhibits a good 
agreement with our theoretical bounds. 
\end{abstract}


\maketitle
\section{Introduction}

\subsection{Setting and Motivation}	

%Applications introducing the topic

L\'evy processes are increasingly popular for the modeling of the 
market prices of risky assets. 
They naturally address the shortcoming of the diffusion models 
by 
allowing large (often heavy-tailed) 
sudden movements of the asset price observed in the 
markets~\cite{schoutens2003levy,KouLevy,tankov2015}. 
%Thus L\'evy models are capable of capturing important aspects 
%of both path and distributional properties of the observed prices 
%of risky assets. 
For risk management, it is therefore crucial to quantify the 
probabilities of rare and/or extreme events in L\'evy models. 
Of particular interest in this context are the 
distributions of the drawdown (the current decline from a 
historical peak) and its duration (the elapsed time since the 
historical peak), 
see e.g.~\cite{DrawdownSornette,DrawdownVecer,CarrDrawdown,
	FutureDrawdowns, MR3556778}. 
%a central topic of the present paper. 
Together with the hedges for barrier 
options~\cite{MR1919609,MR2202995,MR2519843,MR3723380} 
and ruin probabilities in 
insurance~\cite{MR2013414,MR2099651,MR3338431}, 
the expected drawdown and its duration constitute risk measures dependent 
on the following random vector, which is a statistic of the path of a L\'evy process $X$:
a historic maximum $\ov X_T$ at a time $T$, the time $\ov \tau_T(X)$ at which this 
maximum was attained and the value $X_T$ of the process at $T$. 
%In fact, such quantities can typically be recovered from the corresponding statistics of the 
%underlying L\'evy process. 
Since neither the distribution of the drawdown 
$1-\exp(X_T-\ov X_T)$ nor of its duration $T-\ov\tau_T(X)$ is 
analytically tractable for a general $X$,  %~\cite{MR3098676},
simulation provides a natural alternative. The main objective of the 
present paper is to develop and analyse a novel practical 
simulation algorithm for the joint law of $(X_T,\ov X_T,\ov\tau_T(X))$, applicable to a general L\'evy process $X$. 

%LevySupSim, estimation of characteristics and not knowing how to simulate increments

Exact simulation of the drawdown of a L\'evy process is currently 
out of reach \emph{except} for the stable~\cite{MR4032169} and 
jump diffusion cases. However, even in the stable case it is not known 
how to jointly simulate any two 
components of 
$(X_T,\ov X_T,\ov\tau_T(X))$.
%of the following: 
%the state of the process, the drawdown and the drawdown duration. 
Among the approximate simulation algorithms, the recently 
developed stick-breaking approximation~\cite{LevySupSim} 
is the fastest in terms of its computational complexity, as it 
samples from 
the law of
$(X_T,\ov X_T,\ov\tau_T(X))$
with a geometrically decaying bias. 
%geometrically convergent simulation algorithm for the state, 
%drawdown and drawdown duration. 
However, like most 
approximate simulation algorithms for a statistic of the entire 
trajectory, it is only valid for L\'evy process whose increments can 
be sampled. Such a requirement does not hold for large classes of 
widely used L\'evy processes, including the general CGMY 
(aka KoBoL) model~\cite{MR1995283}. Moreover, nonparametric 
estimation of L\'evy processes typically yields L\'evy measures 
whose transitions cannot be 
sampled~\cite{MR2546805,MR2661599,MR2816339,cai_guo_you_2018,MR3909959}, again making a direct application of 
the algorithm in~\cite{LevySupSim} infeasible. 

If the increments of $X$ cannot be sampled, 
a general approach is to use the Gaussian approximation~\cite{MR1834755}, which
substitutes the small-jump component of the L\'evy process by a 
Brownian motion. Thus, the Gaussian approximation process is a jump diffusion 
%For any given small-jump cutoff level, the 
%corresponding Gaussian approximation is in fact a jump diffusion. 
and the exact sample of the random vector (consisting of 
the state of the process, the supremum and the time the supremum 
is attained)  can be obtained by 
applying~\cite[Alg.~{\small{MAXLOCATION}}]{MR2730908} 
between the consecutive jumps. However, little is known about how 
close these quantities are to the vector 
$(X_T,\ov X_T,\ov\tau_T(X))$
that is being approximated in either
Wasserstein or Kolmogorov distances. Indeed, bounds on the 
distances between the marginal of the Gaussian approximation 
and $X_T$ have been considered 
in~\cite{MR3077542} and recently improved 
in~\cite{MR3833470,Mariucci2}. A Wasserstein bound on the 
supremum is given in~\cite{MR3077542} but so far no 
improvement analogous to the marginal case has been established. 
Moreover, to the best of our knowledge, there are no 
corresponding results either for the joint law of 
$(X_T,\ov X_T)$ 
or the time 
$\ov\tau_T(X)$.
Furthermore, as explained in Subsection~\ref{subsubsec:Error_term} 
below, the exact simulation algorithm for the supremum 
and the time of the supremum of a Gaussian approximation 
based on~\cite[Alg.~{\small{MAXLOCATION}}]{MR2730908} is 
unsuitable for the multilevel Monte Carlo  
estimation. 
%(see Subsections~\ref{subsec:literature} 
%and~\ref{subsubsec:Error_term} below for details). 

The main motivation for the present work is to provide 
an operational framework for L\'evy processes, 
which allows us to settle the issues raised in the previous paragraph, develop 
a general simulation algorithm for 
$(X_T,\ov X_T,\ov\tau_T(X))$
and analyse the computational complexity of its Monte Carlo (MC) and multilevel Monte Carlo (MLMC) estimators.

\subsection{Contributions}
\label{subsec:contributions}

The main contributions of this paper are twofold. \textbf{(I)} We establish 
bounds on the Wasserstein and Kolmogorov distances between 
the vector $\ov\chi_T=(X_T,\ov{X}_T,\ov\tau_T(X))$ and 
its Gaussian approximation
$\ov\chi_T^{(\kappa)} = (X_T^{(\kappa)},\ov{X}_T^{(\kappa)},\ov\tau_T(X^{(\kappa)}))$, 
where 
$X^{(\kappa)}$
is a jump diffusion 
equal to the L\'evy process $X$ 
with all the jumps smaller than $\kappa\in(0,1]$ substituted by a Brownian motion (see definition~\eqref{eq:ARA} below),
and
$\ov{X}^{(\kappa)}_T$ 
(resp. $\ov\tau_T(X^{(\kappa)})$) is the supremum 
of
$X^{(\kappa)}$
(resp. the time 
$X^{(\kappa)}$
attains the supremum)
over the time interval $[0,T]$.
\textbf{(II)} We introduce a simple and fast algorithm, \nameref{alg:SBG}, which 
samples exactly the vector of interest for the Gaussian 
approximation of any L\'evy process $X$, develop an MLMC 
estimator based on \nameref{alg:SBG} 
(see~\cite{Jorge_GitHub3} for an implementation in Julia) 
and analyse its 
complexity for discontinuous and locally Lipschitz payoffs arising 
in applications. 
We now briefly 
discuss each of the two contributions. 

%What we bring to the table
\textbf{(I)}~In Theorem~\ref{thm:Wd-triplet} (see also 
Corollary~\ref{cor:Wd-chi}) we establish bounds on the 
Wasserstein distance between $\ov\chi_T$ and 
$\ov\chi_T^{(\kappa)}$
(as $\kappa$ tends to $0$)
under weak assumptions, typically satisfied by the 
models used in applications.
%second order mean distances between the extremal vector consisting of 
%the state, the drawdown and the drawdown duration of the Gaussian 
%approximation and the corresponding extremal vector of the true process 
%under a novel coupling. 
The proof of Theorem~\ref{thm:Wd-triplet} has two main ingredients.
First, in Subsection~\ref{subsec:proofThm1} below, we construct a  
novel \emph{SBG coupling} between $\ov \chi_T$ and 
$\ov\chi_T^{(\kappa)}$,
based on the stick-breaking (SB) representation of 
$\ov\chi_T$ in~\eqref{eq:chi_infty} and the minimal transport 
coupling between the increments of $X$ and its approximation  
$X^{(\kappa)}$.
The second ingredient consists of new bounds on the Wasserstein 
and Kolmogorov distances, given  in Theorems~\ref{thm:W-marginal}
and~\ref{thm:K-marginal} respectively, between the laws of 
$X_t$  and 
$X^{(\kappa)}_t$ for any $t>0$.

Theorem~\ref{thm:Wd-triplet} is our main tool for controlling the 
distance between $\ov \chi_T$ and $\ov\chi_T^{(\kappa)}$.
The SBG coupling underlying it cannot be simulated, but it provides
a bound on the bias of %our simulation 
\nameref{alg:SBG}. Dominating the bias of the time
$\ov\tau_T(X)$,
%the 
%Gaussian approximation of $\ov \tau_T(X)$, 
which is a non-Lipschitz 
functional of the path of $X$, requires 
(by SB representations~\eqref{eq:chi_infty}) the bound in 
Theorem~\ref{thm:K-marginal} on the Kolmogorov distance 
between the marginals. Applications related to the duration of 
drawdown and the risk-management of barrier options require
bounding the bias of certain discontinuous functions of 
$\ov \chi_T$. In Subsection~\ref{subsec:bias} we develop such 
bounds. Their proofs are based on Theorem~\ref{thm:Wd-triplet} 
and Lemma~\ref{lem:Lp-to-barrier} of 
Subsection~\ref{subsec:proofProps}, which essentially 
converts Wasserstein distance into Kolmogorov distance 
for sufficiently regular distributions. We give explicit general 
sufficient conditions on the characteristic triplet of the L\'evy 
process $X$ (see Proposition~\ref{prop:SimpAsmH} below),
which guarantee the applicability of the results of 
Subsection~\ref{subsec:bias} to models typically used in practice. 
%see Figure~\ref{fig:SBG_func_comparison} for the resulting 
%computational complexities of the MLMC estimators. 
Moreover, we obtain bounds on the Kolmogorov distance between the components of
$(\ov X_T,\ov\tau_T(X))$ 
and
$(\ov X^{(\kappa)}_T,\ov\tau_T(X^{(\kappa)}))$
(see Corollary~\ref{cor:K-sup-tau} below), which we hope 
are of independent interest.

\textbf{(II)}
Our main simulation algorithm, \nameref{alg:SBG}, samples jointly 
coupled Gaussian approximations of $\ov\chi_T$ at distinct 
approximation levels. The coupling in \nameref{alg:SBG} exploits 
the following simple observations: 
\begin{itemize}\label{itemize:SBG_alg}
	\item Any Gaussian approximation $\ov \chi^{(\kappa)}_T$ 
	has an SB representation in~\eqref{eq:chi}, where 
		the law of $Y$ in~\eqref{eq:chi} must equal that of $X^{(\kappa)}$. %given in~\eqref{eq:ARA}. 
	\item For any two Gaussian approximations, 
	the stick-breaking process in~\eqref{eq:chi} can be shared. 
	\item The increments in~\eqref{eq:chi} over the 
		shared sticks can be coupled using definition~\eqref{eq:ARA} of the Gaussian approximation $X^{(\kappa)}$.
\end{itemize}

We analyse the computational complexity of the MLMC estimator 
based on \nameref{alg:SBG} for a variety of payoff functions arising 
in applications. Figure~\ref{fig:SBG_func_comparison} shows the 
leading power of the resulting MC and MLMC complexities, 
summarised in Tables~\ref{tab:MC} and~\ref{tab:MLMC} below (see Theorem~\ref{thm:SBG_MLMC} for full details), 
for locally Lipschitz and discontinuous payoffs used in practice. 
To the best of 
our knowledge, neither locally Lipschitz nor discontinuous payoffs 
had been previously considered in the context of MLMC estimation 
under Gaussian approximation.

A key component of the analysis of the complexity of an MLMC estimator 
is the rate of decay of level variances 
(see Appendix~\ref{subsec:MLMC} for details). 
In the case of \nameref{alg:SBG}, the rate of decay is given in 
Theorem~\ref{thm:summary} below for locally Lipschitz and discontinuous payoffs of interest.  
Moreover, the proof of Theorem~\ref{thm:summary} shows that the decay of the level 
variances for Lipschitz payoffs under \nameref{alg:SBG} is 
asymptotically equal to that of Algorithm~\ref{alg:ARA}, which 
samples jointly the increments at two distinct levels only. 
Furthermore, an improved  coupling 
in Algorithm~\ref{alg:ARA}
for the increments of the Gaussian 
approximations (cf. the last bullet on the list above) would reduce 
the computational complexity the MLMC estimator for all payoffs 
considered in this paper (including the discontinuous ones). 
To the best of our knowledge, \nameref{alg:SBG} is the first exact 
simulation algorithm for coupled Gaussian approximations  
of $\ov\chi_T$ with vanishing level variances when $X$ has a 
Gaussian component, see also 
Subsection~\ref{subsubsec:Error_term} below.

%We remark that, when a Brownian component is present, there is 
%currently no other algorithm capable of simulating exactly the 
%triplets of two Gaussian approximations in such a way that the 
%distance between them vanishes as the cutoff levels tend to zero. 
%We analyse the complexities of the corresponding MC and MLMC 
%estimations for all the functions considered in the previous 
%paragraph (see Corollary~\ref{cor:SBG_MC} and 
%Theorem~\ref{thm:SBG_MLMC}). 
%We show that, up to logarithmic factors, the estimations developed here for 
%the spot price and its drawdown are as fast as the corresponding 
%estimations based on~\cite{MR3833470} and developed only for the spot 
%price. 

%Consider the ``straightforward'' simulation of the extremal vector 
%of a jump diffusion. Such an algorithm requires the simulation of 
%the jumps and jump times of the compound Poisson process, 
%simulates the extremal vector of the Brownian component in 
%between every pair of consecutive jumps and finds the overall 
%maximum and its location. The sequential nature of this 
%procedure typically results in long runtimes and numerical 
%problems when the jump intensity is large. 

In Section~\ref{sec:numerics},
using the code in repository~\cite{Jorge_GitHub3}, we test our 
theoretical findings against numerical results. 
We run \nameref{alg:SBG} for models 
in the tempered stable and Watanabe classes. The former is a 
widely used class of processes whose increments cannot be 
sampled for all parameter values and the latter is a well-known class of processes with 
infinite activity but singular continuous increments. In both cases 
we find a reasonable agreement between the theoretical prediction 
and the estimated decays of the bias and level variance, 
see Figures~\ref{fig:TS} \&~\ref{fig:Watanabe} below. 

In the context of MC estimation, a direct simulation algorithm 
based on~\cite[Alg.~{\small{MAXLOCATION}}]{MR2730908} (Algorithm~\ref{alg:ARA_2} below) can 
be used instead of \nameref{alg:SBG}. In Subsection~\ref{subsec:CP_example} we compare 
numerically its cost with that of \nameref{alg:SBG}. In the examples 
we considered, the speedup of \nameref{alg:SBG} over Algorithm~\ref{alg:ARA_2} is about 50, 
see Figure~\ref{fig:ARA-Speedup2}, remaining significant 
even for processes with small jump activity, see 
Figure~\ref{fig:ARA-Speedup}.

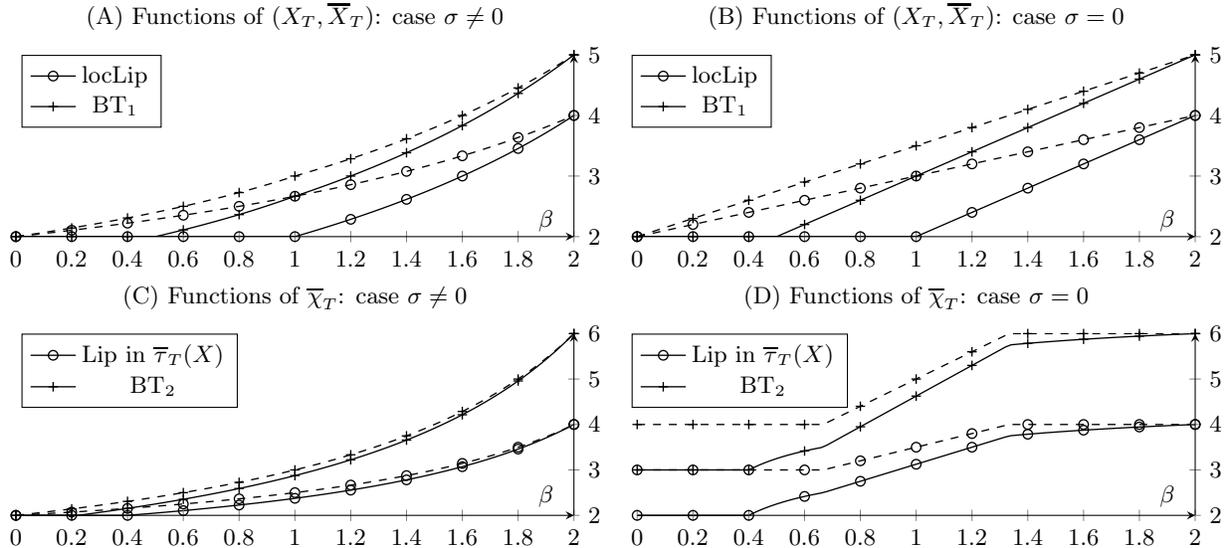
\begin{figure}[ht]
	\begin{center}
		\begin{subfigure}[l]{.49\linewidth}
			{\scalefont{.8}
				\begin{tikzpicture} 
					\begin{axis} 
						[
						title ={(A) Functions of $(X_T,\ov{X}_T)$: 
							case $\sigma\ne0$},
						ymin=2,
						ymax=5,
						xmin=0,
						xmax=2,
						width=9cm,
						height=4cm,
						axis on top=true,
						axis x line=bottom, 
						axis y line=right,
						legend style={at={(.01,1)},anchor=north west}]
						
						\addplot [dashed, mark=o, samples = 101, mark options={scale=.9, solid}, 
						mark repeat=10, color=black, line width = .5, domain=0:2]
						{2+x/(2-x/2)};		
						
						\addplot [solid, mark=o, samples = 101, mark options={scale=.9, solid}, 
						mark repeat=10, color=black, line width = .5, domain=0:2]
						{2 + max(0, 2 * (x-1)/(2-x/2))};
						
						\addplot [dashed, mark=+, samples = 101, mark options={scale=.9, solid}, 
						mark repeat=10, color=black, line width = .5, domain=0:2]
						{2 + x*min(3/(4-x),2/(3-x))};		
						
						\addplot [solid, mark=+, samples = 101, mark options={scale=.9, solid}, 
						mark repeat=10, color=black, line width = .5, domain=0:2]
						{2 + max(0,2*x-1)*min(2/(4-x),4/(9-3*x))};
						
						%\addplot [loosely dotted, samples = 2, line width = .5, domain=0:2]{4};
						
						\node[label=above:{$\beta$}] at (190,-20) {}{};
						% Legend 
						\legend {,
							\footnotesize $\locLip$,,
							\footnotesize $\BT_1$,
						};
					\end{axis}
			\end{tikzpicture}}%\subcaption{}
		\end{subfigure}
		\begin{subfigure}[r]{.49\linewidth}
			{\scalefont{.8}
				\begin{tikzpicture} 
					\begin{axis} 
						[
						title = {(B) Functions of $(X_T,\ov{X}_T)$: case $\sigma=0$},
						ymin=2,
						ymax=5,
						xmin=0,
						xmax=2,
						width=9cm,
						height=4cm,
						axis on top=true,
						axis x line=bottom, 
						axis y line=right,
						legend style={at={(-.01,1)},anchor=north west}]
						
						\addplot [dashed, mark=o, samples = 101, mark options={scale=.9, solid}, 
						mark repeat=10, color=black, line width = .5, domain=0:2]
						{x+2};		
						
						\addplot [solid, mark=o, samples = 101, mark options={scale=.9, solid}, 
						mark repeat=10, color=black, line width = .5, domain=0:2]
						{max(2, 2 * x };
						
						\addplot [dashed, mark=+, samples = 101, mark options={scale=.9, solid}, 
						mark repeat=10, color=black, line width = .5, domain=0:2]
						{2 + 3 * x/2};		
						
						\addplot [solid, mark=+, samples = 101, mark options={scale=.9, solid}, 
						mark repeat=10, color=black, line width = .5, domain=0:2]
						{max(2, 1 + 2 * x};
						
						%\addplot [loosely dotted, samples = 2, line width = .5, domain=0:2]{4};
						
						\node[label=above:{$\beta$}] at (190,-20) {}{};
						% Legend 
						\legend {,
							\footnotesize $\locLip$,,
							\footnotesize $\BT_1$,
						};
					\end{axis}
			\end{tikzpicture}}%\subcaption{}
		\end{subfigure}
		\begin{subfigure}[l]{.49\linewidth}
			{\scalefont{.8}
				\begin{tikzpicture} 
					\begin{axis} 
						[
						title = {(C) Functions of $\ov\chi_T$: 
							case $\sigma\neq0$},
						ymin=2,
						ymax=6,
						xmin=0,
						xmax=2,
						width=9cm,
						height=4cm,
						axis on top=true,
						axis x line=bottom, 
						axis y line=right,
						legend style={at={(.01,1)},anchor=north west}]
						
						\addplot [dashed, mark=o, samples = 101, mark options={scale=.9, solid}, 
						mark repeat=10, color=black, line width = .5, domain=0:2]
						{2 + max(0, x/(3-x))};		
						
						\addplot [solid, mark=o, samples = 101, mark options={scale=.9, solid}, 
						mark repeat=10, color=black, line width = .5, domain=0:2]
						{2 + max(0, (-1 / 2 + 5 * x / 4)/(3-x))};
						
						\addplot [dashed, mark=+, samples = 101, mark options={scale=.9, solid}, 
						mark repeat=10, color=black, line width = .5, domain=0:2]
						{2 + max(0, 2 * x/(3-x))};		
						
						\addplot [solid, mark=+, samples = 101, mark options={scale=.9, solid}, 
						mark repeat=10, color=black, line width = .5, domain=0:2]
						{2 + max(0, (9 * x - 2)/(4 * (3-x)))};
						
						%\addplot [loosely dotted, samples = 2, line width = .5, domain=0:2]{4};
						
						\node[label=above:{$\beta$}] at (190,-20) {}{};
						% Legend 
						\legend {,
							\footnotesize $\Lip$ in $\ov\tau_T(X)$,,
							\footnotesize $\BT_2$,
						};
					\end{axis}
			\end{tikzpicture}}%\subcaption{}
		\end{subfigure}
		\begin{subfigure}[r]{.49\linewidth}
			{\scalefont{.8}
				\begin{tikzpicture} 
					\begin{axis} 
						[
						title = {(D) Functions of $\ov\chi_T$: 
							case $\sigma=0$},
						ymin=2,
						ymax=6,
						xmin=0,
						xmax=2,
						width=9cm,
						height=4cm,
						axis on top=true,
						axis x line=bottom, 
						axis y line=right,
						legend style={at={(-.01,1)},anchor=north west}]
						
						\addplot [dashed, mark=o, samples = 101, mark options={scale=.9, solid}, 
						mark repeat=10, color=black, line width = .5, domain=0:2]
						{2 + min(2, max((3/2) * x, 1))};		
						
						\addplot [solid, mark=o, samples = 101, mark options={scale=.9, solid}, 
						mark repeat=10, color=black, line width = .5, domain=0:2]
						{2 + ifthenelse(x < 2 / 5, 0, ifthenelse(x < 2/3,5 / 4 - 1 / (2 * x), 
							ifthenelse(x < 4/3, 15 * x / 8  - 3 / 4, 5 / 2  - 1 / x)))};
						
						\addplot [dashed, mark=+, samples = 101, mark options={scale=.9, solid}, 
						mark repeat=10, color=black, line width = .5, domain=0:2]
						{2 + min(4, max(3 * x, 2))};		
						
						\addplot [solid, mark=+, samples = 101, mark options={scale=.9, solid}, 
						mark repeat=10, color=black, line width = .5, domain=0:2]
						{2 + ifthenelse(x < 2 / 5, 1, (x * 9 / 8 -  1 / 4) * min(4 / x, max(3, 2 / x)))};
						
						\node[label=above:{$\beta$}] at (190,-20) {}{};
						% Legend 
						\legend {,
							\footnotesize $\Lip$ in $\ov\tau_T(X)$,,
							\footnotesize $\BT_2$,
						};
					\end{axis}
			\end{tikzpicture}}%\subcaption{}
		\end{subfigure}
		\caption{\footnotesize Dashed (resp. solid) line plots the power of 
			$\epsilon^{-1}$ in the computational 
			complexity of an MC (resp. MLMC) estimator, 
			as a function of the BG index $\beta$ defined in~\eqref{eq:I0_beta}, 
			for discontinuous functions in $\BT_1$~\eqref{def:BT1} and 
			$\BT_2$~\eqref{def:BT2}, locally Lipschitz payoffs as well as 
			Lipschitz functions of $\ov \tau_T(X)$. 
			The cases are split according to whether 
			$X$ is with ($\sigma\ne 0$) or without ($\sigma=0$) a 
			Gaussian component. 
			The pictures are based on Tables~\ref{tab:MC} and~\ref{tab:MLMC} 
			under assumptions typically satisfied in applications, 
			see Subsection~\ref{subsec:complexities} below for details. 
			%We assume~(\nameref{asm:(H)}) holds with $\gamma=1$ in 
			%graphs (A) and (B) for the plots relating to barrier-type functions in 
			%$\BT_1$. In graph (D), it is assumed that $\delta=\beta$ 
			%satisfies~(\nameref{asm:(O)}).
			%	These pictures are based 
			%	on Theorem~\ref{thm:SBG_MLMC} and Corollary~\ref{cor:SBG_MC}.
		}\label{fig:SBG_func_comparison}
	\end{center}
\end{figure}

\subsection{Comparison with the literature}
\label{subsec:literature}

Approximations of 
the pair 
$(X_T,\ov X_T)$
%the pair consisting of the state and the drawdown 
%of a L\'evy process 
abound. They include the random walk approximation, a Wiener-Hopf based 
approximation~\cite{MR2895413,MR3138603}, the jump-adapted 
Gaussian (JAG) approximation~\cite{MR2802466,MR2759203} and, 
more recently, the SB approximation~\cite{LevySupSim}. The SB 
approximation converges the fastest as its bias decays 
geometrically in its computational cost. However, the JAG 
approximation is the only method known to us that does not require the ability 
to simulate the increments of the L\'evy process $X$. Indeed, the JAG 
approximation simulates all jumps above a cutoff level, together with their jump 
times, and then samples the transitions of the Brownian motion from the 
Gaussian approximation on a random grid containing all the jump times. 
In contrast, in the present paper we approximate 
$\ov\chi_T=(X_T,\ov{X}_T,\ov\tau_T(X))$ 
with an \emph{exact} sample from the law of 
the Gaussian approximation
$\ov\chi_T^{(\kappa)} = (X_T^{(\kappa)},\ov{X}_T^{(\kappa)},\ov\tau_T(X^{(\kappa)}))$. 
%the 
%corresponding vector of the Gaussian approximation. 
%This elementary method could have already been implemented; 
%however, no sharp theoretical bounds existed and neither did 
%appropriate multilevel simulation methods.

The JAG approximation has been analysed for Lipschitz payoffs of 
the pair $(X_T,\ov{X}_T)$ in~\cite{MR2802466,MR2759203}.
The discontinuous and locally Lipschitz payoffs arising in applications, 
considered in this paper (see Figure~\ref{fig:SBG_func_comparison}),
have to the best of our knowledge not been analysed for the JAG approximation.
Nor have the payoffs involving the time $\ov\tau_T(X)$ the supremum is attained.  
Within the class of Lipschitz payoffs of 
$(X_T,\ov{X}_T)$, 
the computational complexities of the MC and MLMC estimators 
based on \nameref{alg:SBG} are asymptotically dominated by those based on 
the JAG approximation, 
see Figure~\ref{fig:JAG_SBG}. 
In fact, 
\nameref{alg:SBG} applied to discontinuous payoffs 
outperforms the JAG approximation applied to Lipschitz payoffs by 
up to an order of magnitude in computational complexity, 
cf. Figure~\ref{fig:SBG_func_comparison}(A) \&~(B) and Figure~\ref{fig:JAG_SBG}.

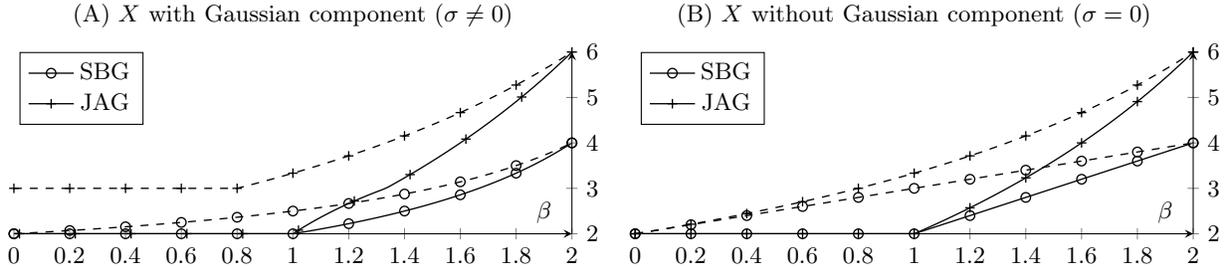
\begin{figure}[ht]
	%\title{Dominant power of $\epsilon^{-1}$ in the computational complexity}\\
	\begin{subfigure}[r]{.49\linewidth}
		{\scalefont{.8}
			\begin{tikzpicture} 
				\begin{axis} 
					[
					title = {(A) $X$ with Gaussian component ($\sigma\neq0$)},
					ymin=2,
					ymax=6,
					xmin=0,
					xmax=2,
					width=9cm,
					height=4cm,
					axis on top=true,
					axis x line=bottom, 
					axis y line=right,
					legend style={at={(.01,1)},anchor=north west}]
					
					\addplot [dashed, mark=o, samples = 101, mark options={scale=.9, solid}, 
					mark repeat=10, color=black, line width = .5, domain=0:2]
					{x/(3-x)+2};		
					
					\addplot [solid, mark=o, samples = 101, mark options={scale=.9, solid}, 
					mark repeat=10, color=black, line width = .5, domain=0:2]
					{2 + max(0, 2 * (x-1) / (3-x)) };
					
					\addplot [dashed, mark=+, samples = 101, mark options={scale=.9, solid}, 
					mark repeat=10, 
					color=black, line width = .5, domain=0:2]
					{max(4*x/(4-x)+2,3)};%4-4/x
					
					\addplot [solid, mark=+, samples = 101, mark options={scale=.9, solid}, 
					mark repeat=10, 
					color=black, line width = .5, domain=0:2]
					{max(2,6-4/x,6*x/(4-x)};%
					
					\node[label=above:{$\beta$}] at (190,-15) {}{};
					% Legend 
					\legend {
						,\footnotesize SBG,
						,\footnotesize JAG,
					};
				\end{axis}
		\end{tikzpicture}}%\subcaption{}
	\end{subfigure}
	\begin{subfigure}[l]{.49\linewidth}
		{\scalefont{.8}
			\begin{tikzpicture} 
				\begin{axis} 
					[
					title = {(B) $X$ without Gaussian component ($\sigma=0$)},
					ymin=2,
					ymax=6,
					xmin=0,
					xmax=2,
					width=9cm,
					height=4cm,
					axis on top=true,
					axis x line=bottom, 
					axis y line=right,
					legend style={at={(.01,1)},anchor=north west}]
					
					\addplot [dashed, mark=o, samples = 101, mark options={scale=.9, solid}, 
					mark repeat=10, color=black, line width = .5, domain=0:2]
					{x+2};		
					
					\addplot [solid, mark=o, samples = 101, mark options={scale=.9, solid}, 
					mark repeat=10, color=black, line width = .5, domain=0:2]
					{max(2, 2 * x};
					
					\addplot [dashed, mark=+, samples = 101, mark options={scale=.9, solid}, 
					mark repeat=10, 
					color=black, line width = .5, domain=0:2]
					{4*x/(4-x)+2};%4-4/x
					
					\addplot [solid, mark=+, samples = 101, mark options={scale=.9, solid}, 
					mark repeat=10, 
					color=black, line width = .5, domain=0:2]
					{max(2,6*x/(4-x)};%4-4/x
					
					\node[label=above:{$\beta$}] at (190,-15) {}{};
					% Legend 
					\legend {
						,\footnotesize SBG,
						,\footnotesize JAG,
					};
				\end{axis}
		\end{tikzpicture}}%\subcaption{}
	\end{subfigure}
	\caption{\footnotesize Dashed (resp. solid) lines represent 
the power of $\epsilon^{-1}$ in the computational 
complexity of the 
MC (resp. MLMC) estimator
for the expectation of a Lipschitz functional $f(X_T,\ov{X}_T)$,
plotted as a function of the BG index $\beta$ defined in~\eqref{eq:I0_beta}. 
The SBG plots are based on %require $\E[f(X_T,\un{X}_T)^2]<\infty$ 
Tables~\ref{tab:MC} and~\ref{tab:MLMC} below. 
%	Theorem~\ref{thm:SBG_MLMC} and Corollary~\ref{cor:SBG_MC}; 
The JAG plots are based on~\cite[Cor.~3.2]{MR2759203} for the MC 
cost, and~\cite[Cor.~1.2]{MR2759203} if $\beta\geq 1$ 
(resp.~\cite[Cor.~1]{MR2802466} if $\beta<1$) for the MLMC cost.
	}\label{fig:JAG_SBG}
\end{figure}

In order to understand where the differences in 
Figure~\ref{fig:JAG_SBG} come from, in Table~\ref{tab:JAG_SBG} 
we summarise the bias and level variance for \nameref{alg:SBG} and 
the JAG approximation as a function of the cutoff level 
$\kappa\in(0,1]$ in the Gaussian approximation 
(cf.~\eqref{eq:ARA} below). 

\begin{table}[ht]
{\scalefont{.9}
	\begin{tabular}{|c|c|c|c|}
		\hline
		% ROW 1
		Gaussian component & Approximation& Bias & Level variance \\
		\hline
		% ROW 2
		\multirow{2}{*}{With ($\sigma\ne0$)}
		& JAG
		& $\max\{\kappa^{1-\beta/4},\kappa^{\beta/2}\}
		\log^{1/2}(1/\kappa)$ 
		& $\max\{\kappa^{2-\beta},\kappa^\beta\log(1/\kappa)\}$\\
		%\hline
		% ROW 3
		& SBG
		& $\kappa^{3-\beta}\log(1/\kappa)$ 
		& $\kappa^{2-\beta}$\\
		\hline
		% ROW 4
		\multirow{2}{*}{Without ($\sigma=0$)}
		& JAG
		& $\max\{\kappa^{1-\beta/4}\log^{1/2}(1/\kappa),
		\kappa^\beta\}$ 
		& $\max\{\kappa^{2-\beta},\kappa^{2\beta}\}$\\
		%\hline
		% ROW 5
		& SBG
		& $\kappa\log(1/\kappa)$
		& $\kappa^{2-\beta}$\\
		\hline
	\end{tabular}
}\caption{\footnotesize 
The rates (as $\kappa\to 0$) of decay of bias and level variance for Lipschitz payoffs 
	of $(X_T,\ov X_T)$ under the JAG approximation are based on~\cite[Cor.~3.2]{MR2759203} 
and~\cite[Thm~2]{MR2802466}, respectively. The rates on the bias 
and level variance for the \nameref{alg:SBG} are given in 
Theorems~\ref{thm:Wd-triplet} \&~\ref{thm:summary} below.
}\label{tab:JAG_SBG}
\end{table}

Table~\ref{tab:JAG_SBG} shows that both bias and level variance 
decay at least as fast (and typically faster) for \nameref{alg:SBG} 
than for the JAG approximation. The large improvement in the 
computational complexity of the MC estimator in 
Figure~\ref{fig:JAG_SBG} is due to the faster decay of the bias under 
\nameref{alg:SBG}. Put differently, the SBG coupling constructed in 
this paper controls the Wasserstein distance much better than the 
KMT-based coupling in~\cite{MR2759203}. For the BG index $\beta>1$, 
the improvement in the computational complexity of the MLMC 
estimator is mostly due to an faster bias decay. For $\beta<1$, 
Figure~\ref{fig:JAG_SBG}(A) suggests that the computational 
complexity of the MLMC estimator under both algorithms is optimal.
However, in this case, Table~\ref{tab:JAG_SBG} and the equality 
in~\eqref{eq:ML_cost} imply that the MLMC estimator based on the 
JAG approximation has a computational complexity proportional to 
$\epsilon^{-2}\log^3(1/\epsilon)$ while that of 
\nameref{alg:SBG} is proportional to $\epsilon^{-2}$. This 
improvement is due solely to the faster decay of level variance 
under \nameref{alg:SBG}. 
%Both improvements are starker when $X$ has a Gaussian 
%component. 
The numerical experiments in Subsection~\ref{subsec:TSW_example} 
suggest that our bounds for Lipschitz and locally Lipschitz functions 
are sharp, see graphs (A) \& (C) in 
Figures~\ref{fig:TS} \&~\ref{fig:Watanabe}. 

To the best of our knowledge, in the literature there are no directly comparable results to 
either Theorem~\ref{thm:Wd-triplet} or 
Proposition~\ref{prop:barrier-tau}. 
Partial results in the direction of Theorem~\ref{thm:Wd-triplet} 
are given in~\cite{MR3077542,MR3833470,Mariucci2}. We will now briefly 
comment on these results.\\
\textit{Distance between the marginals $X_t$ and $X^{(\kappa)}_t$}: 
 Theorem~\ref{thm:W-marginal} below, a key step 
in the proof of Theorem~\ref{thm:Wd-triplet}, improves the 
bounds in \cite[Thm~9]{MR3833470} on the Wasserstein 
distance. Theorem~\ref{thm:K-marginal} below, a further 
key ingredient in the proof of Theorem~\ref{thm:Wd-triplet}, 
bounds the Kolmogorov distance with better rates than those 
of~\cite[Prop.~10 (part 1)]{MR3077542} (as $\kappa\to 0$). 
Papers~\cite{MR3833470,Mariucci2} obtain bounds on the total 
variation distance between $X_t$ and $X^{(\kappa)}_t$, dominating the Kolmogorov 
distance. However, Theorem~\ref{thm:K-marginal} again yields faster decay. 
For  more details about these comparisons see Subsection~\ref{subsec:wasserstein} below.\\
\textit{Distance between the suprema $\ov{X}_t$ and $\ov{X}^{(\kappa)}_t$}:
 the rate of the bound 
in~\cite[Thm~2]{MR3077542} on the Wasserstein distance is worse 
than that implied by the bound in Corollary~\ref{cor:Wd-chi} below 
on the Wasserstein distance between $(X_t,\ov{X}_t)$ and 
$(X^{(\kappa)}_t,\ov{X}^{(\kappa)}_t)$.
Proposition~\ref{prop:logLip} below 
bounds the bias of locally Lipschitz 
functions, generalising~\cite[Prop.~9]{MR3077542} and providing 
a faster decay rate. Proposition~\ref{prop:barrier} and 
Corollary~\ref{cor:K-sup-tau}(a) below cover a class of 
discontinuous payoffs, including the up-and-in digital option 
%$\p(\ov{X}_T>x)$ 
considered in~\cite[Prop.~10 (part 3)]{MR3077542}, and 
provide a faster rate of decay as $\kappa\to0$ if either $X$ has a 
Gaussian component or the BG index $\beta>2/3$. 

\subsection{Organisation}
\label{subsec:organisation}

The remainder of the paper is organised as follows. 
Section~\ref{sec:approximations} recalls SB 
representation~\eqref{eq:chi_infty}--\eqref{eq:chi}  and the 
Gaussian approximation~\eqref{eq:ARA} developed in~\cite{LevySupSim} and~\cite{MR1834755}, respectively. 
Section~\ref{sec:main-theory} presents bounds on Wasserstein and 
Kolmogorov distances between $\ov\chi_T$ and its Gaussian 
approximation 
$\ov\chi^{(\kappa)}_T$
and the biases of certain payoffs arising in 
applications. Section~\ref{sec:main-theory} also provides simple 
sufficient conditions, in terms of the L\'evy triplet, under which 
these bounds hold. Section~\ref{sec:main-apps} constructs our 
main algorithm, \nameref{alg:SBG}, and presents the computational 
complexity of the corresponding MC and MLMC estimators for all 
payoffs considered in this paper. 
In Section~\ref{sec:numerics} we illustrate numerically our 
results for a widely used class of L\'evy models. 
The proofs and the technical results are found in 
Section~\ref{sec:proofs}. Appendix~\ref{app:MonteCarlo} gives a 
brief account of the complexity analysis of MC and MLMC (introduced in~\cite{Heinrich_MLMC,MR2436856}) 
estimators. 

\section{The stick-breaking representation and the Gaussian approximation}
\label{sec:approximations}

Let $f:[0,\infty)\to\R$ be a  right-continuous function with left 
limits. For any $t\in(0,\infty)$, define $\un{f}_t:=\inf_{s\in[0,t]}f(s)$, 
$\ov{f}_t:=\sup_{s\in[0,t]}f(s)$ and let $\un\tau_t(f)$ 
(resp. $\ov\tau_t(f)$) be the last time before $t$ that 
the infimum $\un{f}_t$ (resp. supremum $\ov{f}_t$) is attained. 
Throughout $X=(X_t)_{t\geq0}$ denotes a L\'evy process, 
i.e. a  stochastic process started at the origin with independent, 
stationary increments and right-continuous paths with left limits, 
see~\cite{MR1406564,MR2250061,MR3185174} for background 
on L\'evy processes. In mathematical finance, the risky asset price 
$S=(S_t)_{t\geq0}$ under an exponential L\'evy model is given by 
$S_t:=S_0e^{X_t}$. The price $S_t$, its \textit{drawdown} 
$1-S_t/\ov{S}_t$ (resp. \textit{drawup} $1-\un{S}_t/S_t$) and 
\textit{duration} $t-\un\tau_t(S)$ (resp. $t-\ov\tau_t(S)$) at time $t$ 
can be recovered from the vector 
$\ov\chi_t:=(X_t,\ov{X}_t,\ov\tau_t(X))$ 
(resp. $\un\chi_t:=(X_t,\un{X}_t,\un\tau_t(X))$). Since $Z:=-X$ 
is a L\'evy process and $\ov\chi_t=(-Z_t,-\un{Z}_t,\un\tau_t(Z))$, 
%by virtue of the 
%distributional identity $\ov\chi_t\eqd(-X'_t,-\un{X}'_t,\un\tau_t(X'))$ for the 
it is sufficient to analyse the vector $\un\chi_t$. 

\subsection{The stick-breaking (SB) representation}
\label{subsec:SB-rep}
We begin by recalling~\cite[Thm~1]{LevySupSim}, which is at the core of the 
bounds and algorithms developed in this paper. Given a L\'evy process $X$  
and a time horizon $t>0$, there exists a 
coupling 
$(X,Y)$, where $Y\eqd X$ (throughout the paper $\eqd$ denotes 
equality in law), and a stick-breaking process 
$\ell=(\ell_n)_{n\in\N}$ on $[0,t]$ based on the uniform law $\U(0,1)$ 
(i.e. $L_0:=t$, $L_n:=L_{n-1}U_n$, 
$\ell_n:=L_n-L_{n-1}$  for $n\in\N$,
where $(U_n)_{n\in\N}$ is an iid sequence following $U_n\sim\U(0,1)$),
such that a.s. 
\begin{equation}
\label{eq:chi_infty}
\un\chi_t
=\sum_{k=1}^\infty \big(Y_{L_{k-1}}-Y_{L_k},
	\min\{Y_{L_{k-1}}-Y_{L_k},0\},
	\ell_k\cdot\1_{\{Y_{L_{k-1}}-Y_{L_k}\le 0\}}\big).
\end{equation}
Since, given $L_n$, $(\ell_k)_{k>n}$ is a stick-breaking process on $[0,L_n]$, for any $n\in\N$, \eqref{eq:chi_infty} implies  
\begin{equation}\label{eq:chi}
\un\chi_t
\eqd(Y_{L_n},\un{Y}_{L_n},\un\tau_{L_n}(Y))+
	\sum_{k=1}^n \big(Y_{L_{k-1}}-Y_{L_k},
	\min\{Y_{L_{k-1}}-Y_{L_k},0\},
	\ell_k\cdot\1_{\{Y_{L_{k-1}}-Y_{L_k}\le 0\}}\big).
\end{equation}

Observe that the vector $(Y_{L_n},\un{Y}_{L_n},\un\tau_{L_n}(Y))$
and the sum on the right-hand side of the identity in~\eqref{eq:chi}
are conditionally independent given $L_n$: the former (resp. latter) is a function of $(Y_s)_{s\in[0,L_n]}$ 
(resp. $(Y_s-Y_{L_n})_{s\in[L_n,t]}$), cf. Figure~\ref{fig:stick}.
The vector of interest
$\un\chi_t$
is thus represented by the corresponding vector 
$(Y_{L_n},\un{Y}_{L_n},\un\tau_{L_n}(Y))$
over an exponentially small interval 
(since $\E[L_n]=2^{-n}t$) and $n$ independent increments of the L\'evy process 
over random intervals independent of $Y$.
In~\eqref{eq:chi} 
and throughout $\1_A$ is an indicator of a set $A$:
$\1_A(x) =\1_{\{x\in A\}}$ %=\1\{x\in A\}$ 
is 1 (resp. 0) if $x\in A$ (resp. $x\notin A$).

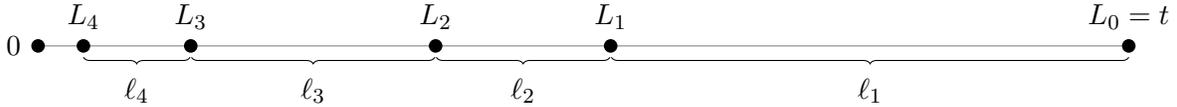
\begin{figure}[ht]
\begin{tikzpicture}
% Draw the axis
\draw [thin, draw=gray, ->] (0,0) -- (14.5,0);
\node[circle, fill=black, scale=.5, label=left:{$0$}] at (0,0){}{};
\node[circle, fill=black, scale=.5, label=above:{$L_0=t$}] at (14.5,0){}{};
\node[circle, fill=black, scale=.5, label=above:{$L_1$}] at (7.612,0){}{};
\node[circle, fill=black, scale=.5, label=above:{$L_2$}] at (5.285,0){}{};
\node[circle, fill=black, scale=.5, label=above:{$L_3$}] at (2.030,0){}{};
\node[circle, fill=black, scale=.5, label=above:{$L_4$}] at (0.603,0){}{};
%\node[circle, fill=black, scale=.5, label=above:{$L_5$}] at (0.203,0){}{};

\draw [
%thick,
decoration={
	brace,
	mirror,
	raise=0.15cm
}, decorate] (7.622,0) -- (14.5,0) node 
	[pos=0.5,anchor=north,yshift=-.3cm] {$\ell_1$}; 

\draw [
%thick,
decoration={
	brace,
	mirror,
	raise=0.15cm
}, decorate] (5.295,0) -- (7.602,0) node 
	[pos=0.5,anchor=north,yshift=-.3cm] {$\ell_2$}; 

\draw [
%thick,
decoration={
	brace,
	mirror,
	raise=0.15cm
},
decorate] (2.040,0) -- (5.275,0) node 
[pos=0.5,anchor=north,yshift=-.3cm] {$\ell_3$};

\draw [
%thick,
decoration={
	brace,
	mirror,
	raise=0.15cm
}, decorate] (0.603,0) -- (2.020,0) node 
[pos=0.5,anchor=north,yshift=-.3cm] {$\ell_4$};
\end{tikzpicture}
\caption{\footnotesize 
	The figure illustrates the first $n=4$ sticks of a 
	stick-breaking process. The increments of $Y$ in~\eqref{eq:chi} are taken 
	over the intervals $[L_k,L_{k-1}]$ of length $\ell_k$. Crucially, the time  
	$L_n$ featuring in the vector $(Y_{L_n},\un{Y}_{L_n},\un\tau_{L_n}(Y))$
	%on the right-hand side of
	in~\eqref{eq:chi} is exponentially
	small in $n$.  \label{fig:stick} }
\end{figure}

We stress that~\eqref{eq:chi_infty} and~\eqref{eq:chi} reduce the analysis 
of the path-functional $\un\chi_t$ to that of the increments of $X$, 
since the ``error term'' 
$(Y_{L_n},\un{Y}_{L_n},\un\tau_{L_n}(Y))$
in~\eqref{eq:chi} is typically exponentially 
small in $n$. More generally, for another L\'evy process $X'$, the vectors
$\un\chi_t$ and $(X'_t,\un{X}'_t,\un\tau_t(X'))$ will be close if the 
increments of $Y$ and $Y'$ over the intervals $[L_k,L_{k-1}]$ are close: apply~\eqref{eq:chi} with a single 
stick-breaking process $\ell$, independent of both L\'evy processes 
$Y\eqd X$ and $Y'\eqd X'$, 
%in the stick-breaking 
%representations in~\eqref{eq:chi_infty}, %for $X$ and $X'$, 
respectively. 
This observation constitutes a key step in the construction of the coupling used in the proof of 
Theorem~\ref{thm:Wd-triplet} below, which in turn plays a crucial role in controlling the bias 
(see the subsequent results 
of Section~\ref{sec:main-theory})
of our main simulation algorithm \nameref{alg:SBG}
%Moreover, the simulation algorithms 
described in Section~\ref{sec:main-apps} below. 
%are based on the identity in~\eqref{eq:chi} for general L\'evy processes and thus  
%requires the ability to sample (approximately) their increments. 
\nameref{alg:SBG} is 
based on~\eqref{eq:chi} with $X'$ being 
the Gaussian approximation of a general L\'evy process $X$ 
introduced  in~\cite{MR1834755} and recalled
briefly in the next subsection.

\subsection{The Gaussian approximation}
\label{subsec:ARA}

The law of a L\'evy process $X=(X_t)_{t\geq0}$ is 
uniquely determined by the law of its marginal $X_t$ (for any $t>0$),
which is in turn given by the L\'evy-Khintchine formula~\cite[Thm~8.1]{MR3185174} 
\begin{equation}
	\label{eq:Levy_Khinchin}
\frac{1}{t}\log\E\big[e^{iuX_t}\big]
= iub-\frac{1}{2}u^2\sigma^2
+ \int_{\R\setminus\{0\}}
	\big(e^{iux}-1-iux\cdot\1_{(-1,1)}(x)\big)\nu(dx),
\qquad u\in\R.
\end{equation}
The \textit{L\'evy measure} $\nu$ is required to satisfy 
$\int_{\R\setminus\{0\}} \min\{x^2,1\}\nu(dx)<\infty$, while $\sigma\geq0$ specifies the 
volatility of the Brownian component of $X$. Note that the drift $b\in\R$ 
depends on the cutoff function $x\mapsto\1_{(-1,1)}(x)$. Thus the L\'evy 
triplet $(\sigma^2,\nu,b)$, with respect to the cutoff function 
$x\mapsto\1_{(-1,1)}(x)$, determines the law of $X$.
All the L\'evy triplets in the present paper use this cutoff function. 

The \textit{L\'evy-It\^o decomposition} 
at level $\kappa\in(0,1]$ (see~\cite[Thms~19.2~\&~19.3]{MR3185174}) 
\label{page:levy_ito_decomp} is given by 
\begin{equation}\label{eq:levy-ito}
X_t=b_\kappa t+\sigma B_t+J^{1,\kappa}_t+J^{2,\kappa}_t,
\qquad t\geq0,
\end{equation}
where $b_\kappa := b-\int_{(-1,1)\setminus(-\kappa,\kappa)} x\nu(dx)$, 
$B=(B_t)_{t\geq0}$ is a standard Brownian motion and the processes 
$J^{1,\kappa}=(J^{1,\kappa}_t)_{t\geq0}$ and 
$J^{2,\kappa}=(J^{2,\kappa}_t)_{t\geq0}$ are L\'evy with triplets 
$(0,\nu|_{(-\kappa,\kappa)},0)$ and 
$(0,\nu|_{\R\setminus(-\kappa,\kappa)},b-b_\kappa)$, respectively.  
The processes $B$, $J^{1,\kappa}$, $J^{2,\kappa}$ in~\eqref{eq:levy-ito} are 
independent, $J^{1,\kappa}$ is an $L^2$-bounded martingale with jumps of 
magnitude less than $\kappa$ and $J^{2,\kappa}$ is a driftless
 (i.e. piecewise constant) compound Poisson process with intensity 
$\ov\nu(\kappa):=\nu(\R\setminus(-\kappa,\kappa))$
and jump distribution 
$\nu|_{\R\setminus(-\kappa,\kappa)}/\ov\nu(\kappa)$.

In applications, the main problem lies in the user's inability to simulate the 
increments of $J^{1,\kappa}$ in~\eqref{eq:levy-ito}, i.e. the small jumps 
of the L\'evy process $X$. Instead of ignoring this component for a small 
value of $\kappa$, the Gaussian approximation~\cite{MR1834755}
\begin{equation}
\label{eq:ARA}
X^{(\kappa)}_t 
	:=b_\kappa t + \tsqrt{\ov\sigma_\kappa^2+\sigma^2} W_t
		+ J^{2,\kappa}_t,
\quad\text{where}\quad 
\ov\sigma^2_{\kappa}
	:=\int_{(-\kappa,\kappa)\setminus\{0\}}x^2\nu(dx)
\quad\text{and}\quad\kappa\in(0,1],
\end{equation}
substitutes the martingale $\sigma B + J^{1,\kappa}$ in~\eqref{eq:levy-ito} with a 
Brownian motion with variance $\ov\sigma^2_{\kappa} + \sigma^2$. 
In~\eqref{eq:ARA}, the standard Brownian motion $W=(W_t)_{t\geq0}$ 
is independent of $J^{2,\kappa}$. Let $\ov\sigma_\kappa$ denote the 
non-negative square root of $\ov\sigma^2_\kappa$. The 
\textit{Gaussian approximation} of $X$ at level $\kappa$, given by 
the L\'evy process $X^{(\kappa)}=(X^{(\kappa)}_t)_{t\ge0}$, is natural in 
the following sense:  the weak convergence 
$\ov\sigma_\kappa^{-1}J^{1,\kappa}\cid W$ (in the Skorokhod space 
$\D[0,\infty)$) as $\kappa\to0$ holds if and only if 
$\ov\sigma_{\min\{K\ov\sigma_\kappa,\kappa\}}/
\ov\sigma_\kappa\to1$ for every $K>0$ (see~\cite{MR1834755}). 
This condition holds if $\ov\sigma_\kappa/\kappa\to\infty$ and the 
two conditions are equivalent if $\nu$ has no atoms in a 
neighbourhood of zero~\cite[Prop.~2.2]{MR1834755}. 
%In particular, it is implied by Orey's condition (see Assumption~(\nameref{asm:(O)}) below). 

Since $J^{2,\kappa}$ has an average of $\ov\nu(\kappa)t$ jumps on 
$[0,t]$, the expected complexity of simulating the increment 
$X_t^{(\kappa)}$ is a constant multiple of $1+\ov\nu(\kappa)t$ 
(see Algorithm~\ref{alg:ARA} below). Moreover, the user need only 
be able to sample from the normalised tails of $\nu$, which can 
typically be achieved in multiple ways (see e.g.~\cite{MR1833707}). 
The behaviour of $\ov\nu(\kappa)$ and $\ov\sigma_\kappa$ as 
$\kappa\downarrow0$, key in the analysis of the MC/MLMC complexity, 
can be described in terms of the 
\textit{Blumenthal-Getoor} (BG) index~\cite{MR0123362} $\beta$, 
defined as 
\begin{equation}\label{eq:I0_beta}
\beta:=\inf\{p>0:I_0^p<\infty\},\quad\text{where}\quad 
	I_0^p:=\int_{(-1,1)\setminus\{0\}}|x|^p\nu(dx)\quad\text{for any }p\geq 0. 
\end{equation}
Note that $\beta\in[0,2]$, since $I_0^2<\infty$ by the definition of the L\'evy measure $\nu$. Furthermore, $I_0^1<\infty$ if and 
only if the paths of $J^{1,\kappa}$ have finite variation. Moreover, $I_0^p<\infty$ for any 
$p>\beta$, but $I_0^\beta$ can be either finite or infinite. 
%We have the 
%following elementary result (see~\cite[Lem.~9]{LevySupSim}). 
If $q\in[0,2]$ satisfies $I_0^q<\infty$, the following inequalities hold for all 
$\kappa\in(0,1]$ (see e.g.~\cite[Lem.~9]{LevySupSim}):
\begin{equation}
\label{eq:BG_bounds}
\ov\sigma^2_{\kappa}\leq I_0^q \kappa^{2-q}\qquad\text{and}\qquad
\ov\nu(\kappa)\leq\ov\nu(1)+I_0^q \kappa^{-q}.
\end{equation}

Finally we stress that the dependence between $W$ in~\eqref{eq:ARA} 
and $\sigma B + J^{1,\kappa}$ in~\eqref{eq:levy-ito} has not been 
specified.  This coupling will vary greatly, depending on the 
circumstance (e.g. the analysis of the Wasserstein distance between 
functionals of $X$ and $X^{(\kappa)}$ (Section~\ref{sec:main-theory}) 
or the minimisation of level variances in MLMC 
(Section~\ref{sec:main-apps})). Thus, unless otherwise stated, no 
explicit dependence between $\sigma B + J^{1,\kappa}$ and $W$ is 
assumed.

\section{Distance between the laws of $\un \chi_t$ and its Gaussian approximation $\un\chi_t^{(\kappa)}$} 
\label{sec:main-theory}

%Denote 
%$\un\chi_t^{(\kappa)}= (X_t^{(\kappa)},\un{X}_t^{(\kappa)},\un\tau_t(X^{(\kappa)}))$
In this section we present bounds on the distance between the laws 
of the vectors $\un\chi_t$, defined in 
Section~\ref{sec:approximations} above, and its Gaussian 
approximation 
$\un\chi_t^{(\kappa)}
:= (X_t^{(\kappa)},\un{X}_t^{(\kappa)},\un\tau_t(X^{(\kappa)}))$, 
based on the L\'evy process  $X^{(\kappa)}$ in~\eqref{eq:ARA}. 
Our bounds on the Wasserstein distance (see 
Theorem~\ref{thm:Wd-triplet} and Corollary~\ref{cor:Wd-chi} in 
Subsection~\ref{subsec:wasserstein}) are based on a coupling 
constructed in Subsection~\ref{subsec:proofThm1} below, which in 
turn draws on the coupling in~\eqref{eq:chi_infty}. 
Theorem~\ref{thm:Wd-triplet} is then applied to control the bias of 
certain discontinuous and non-Lipschitz functions of $\un\chi_t$ 
arising in applications (Subsection~\ref{subsec:bias} below) as well 
as the Kolmogorov distances between the components of 
$(\un{X}_t,\un\tau_t(X))$ and 
$(\un{X}_t^{(\kappa)},\un\tau_t(X^{(\kappa)}))$ 
(see Subsection~\ref{subsec:Kolmogorov} below).

\subsection{Bounds on the Wasserstein and Kolmogorov distances}
\label{subsec:wasserstein}

In order to study the Wasserstein distance between $\un\chi_t$ and 
$\un\chi_t^{(\kappa)}$ via~\eqref{eq:chi_infty}--\eqref{eq:chi}, we 
have to quantify the Wasserstein and Kolmogorov distances between 
the increments $X_s$ and $X^{(\kappa)}_s$ for any time $s>0$. 
With this in mind, we start with Theorems~\ref{thm:W-marginal} 
and~\ref{thm:K-marginal}, which play a key role in the proofs of the 
main results of the subsection, Theorem~\ref{thm:Wd-triplet} and 
Corollary~\ref{cor:Wd-chi} below, and are of independent interest. 

\begin{thm}
\label{thm:W-marginal}
There exist universal constants $K_1:=1/2$ and $K_p>0$, 
$p\in(1,2]$, independent of $(\sigma^2,\nu,b)$, such that for any 
$t>0$ and $\kappa\in(0,1]$ there exists a coupling 
$(X_t,X_t^{(\kappa)})$ satisfying 
\begin{equation}\label{eq:Lp-marginal}
\E\big[\big|X_t-X_t^{(\kappa)}\big|^p\big]^{1/p}
%=\E\Big[\Big|M^{(\kappa)}_t
%	-\sqrt{\ov\sigma_\kappa^2+\sigma^2} W_t\Big|^p\Big]^{1/p} 
\leq\min\big\{\sqrt{2t}\ov\sigma_\kappa,
	K_p\kappa\varphi_\kappa^{2/p}\big\},
\quad\text{where $\varphi_\kappa:=\ov\sigma_\kappa/
	\tsqrt{\ov\sigma_\kappa^2+\sigma^2}$, 
	for all $p\in[1,2]$.}
\end{equation}
\end{thm}

Theorem~\ref{thm:W-marginal} bounds the $L^p$-Wasserstein 
distance (see~\eqref{eq:Wasserstein} below for definition) between 
$X_t$ and $X_t^{(\kappa)}$. The inequality in~\eqref{eq:Lp-marginal} 
sharpens the bound 
$
\E[|X_t-X_t^{(\kappa)}|^p]^{1/p}
%=\E\Big[\Big|M^{(\kappa)}_t
%	-\sqrt{\ov\sigma_\kappa^2+\sigma^2} W_t\Big|^p\Big]^{1/p} 
\leq\min\{\sqrt{2t}\ov\sigma_\kappa,
	K_p\kappa\}$
in~\cite[Thm~9]{MR3833470}: the factor 
$\varphi_\kappa^{2/p}\in[0,1]$ tends to zero (with $\kappa\to0$) as a constant multiple of $\ov \sigma_\kappa^{2/p}$ if the 
Brownian component is present (i.e. $\sigma>0$) and is equal to $1$ 
when $\sigma=0$. The bound in~\eqref{eq:Lp-marginal} cannot be 
improved in general in the sense that there exists a L\'evy processes 
for which, up to constants, the reverse inequality holds 
(see~\cite[Rem.~3]{MR3833470} and~\cite[Sec.~4]{MR2870514}). 

The proof of Theorem~\ref{thm:W-marginal}, given in 
Subsection~\ref{subsec:marginals} below, decomposes the increment  
$M_t^{(\kappa)}$ of the L\'evy martingale 
$M^{(\kappa)}:=\sigma B + J^{1,\kappa}$ into a sum of $m$ iid copies 
of $M_{t/m}^{(\kappa)}$ and applies a Berry-Essen-type bound for the 
Wasserstein distance~\cite{MR2548505} in the context of a central limit 
theorem (CLT) as $m\to\infty$. The small-time moment asymptotics of 
$M_{t/m}^{(\kappa)}$ in~\cite{MR2479503} imply that
$M^{(\kappa)}_t$ is much closer to the 
Gaussian limit in the CLT 
if the Brownian component is present 
than if $\sigma=0$. This explains a vastly 
superior rate in~\eqref{eq:Lp-marginal} in the case $\sigma^2>0$.

Bounds on the Kolmogorov distance may require the 
following generalisation of  Orey's condition, 
which makes the distribution of 
$X_t$ sufficiently regular (see~\cite[Prop.~28.3]{MR3185174}).

\begin{asm*}[O-$\delta$]\label{asm:(O)}
The inequality 
$\inf_{u\in(0,1]}u^{\delta-2}(\ov\sigma^2_u+\sigma^2)>0$
holds for some $\delta\in(0,2]$. 
%	such that the L\'evy measure $\nu$ and the Gaussian component $\sigma^2$ of $X$ satisfy
%	(see~\eqref{eq:ARA} for the definition of $\ov\sigma^2_u$).
\end{asm*}

\begin{thm}
\label{thm:K-marginal}
\nf{(a)} 
There exists a constant $C_{\BE}\in(0,\tfrac{1}{2})$, such that 
% be the Berry-Esseen constant. For any 
for any $\kappa\in(0,1]$, $t>0$ we have:
\begin{equation}\label{eq:K-marginal-1}
\sup_{x\in\R}\big|\p\big(X_t\le x\big)
	-\p\big(X_t^{(\kappa)}\le x\big)\big|
\le C_{\BE}(\kappa/\ov\sigma_\kappa)\varphi_\kappa^3/\sqrt{t}.
\end{equation}
	\nf{(b)} Let Assumption~(\nameref{asm:(O)}) hold.
Then for every $T>0$ there exists a constant $C>0$, depending only on 
$(T,\delta,\sigma,\nu)$, such that for any $\kappa\in(0,1]$ and 
$t\in(0,T]$ we have: 
\begin{equation}\label{eq:K-marginal-3}
\sup_{x\in\R}\big|\p\big(X_t\le x\big)
	-\p\big(X_t^{(\kappa)}\le x\big)\big|
\le\big(Ct^{-1/\delta}
	\min\big\{\sqrt{t}\ov\sigma_\kappa,
		\kappa\varphi_\kappa\big\}\big)^{2/3}.
\end{equation}
\end{thm}

The proof of Theorem~\ref{thm:K-marginal} is in 
Subsection~\ref{subsec:marginals} below. Part~(a) follows the same 
strategy as the proof of Theorem~\ref{thm:W-marginal}, applying 
the Berry-Esseen theorem (instead of~\cite[Thm 4.1]{MR2548505}) 
to bound the Kolmogorov distance. For the same reason as 
in~\eqref{eq:Lp-marginal}, the rate in~\eqref{eq:K-marginal-1} is far 
better if $\sigma^2>0$. Proof of Theorem~\ref{thm:K-marginal}(b) 
bounds the density of $X_t$ using results in~\cite{MR1449834} 
and applies~\eqref{eq:Lp-marginal}. 

Note that  no assumption is made on the L\'evy process $X$ in 
Theorem~\ref{thm:K-marginal}(a). 
%To the best of our 
%knowledge,~\eqref{eq:K-marginal-1} is the first such general bound 
%in the literature. 
In particular, Assumption~(\nameref{asm:(O)}) is not 
required in part (a); however, if~(\nameref{asm:(O)}) is not satisfied, 
%if no $\delta\in(0,2]$ satisfies this 
%assumption 
implying in particular that $\sigma=0$, it is possible for the bound 
in~\eqref{eq:K-marginal-1} not to vanish as $\kappa\to 0$  even if the 
L\'evy process has infinite activity, i.e. $\nu(\R\setminus\{0\})=\infty$. In fact, if $\sigma=0$, the bound 
in~\eqref{eq:K-marginal-1} vanishes (as $\kappa\to0$) if and only if $\ov\sigma_\kappa/\kappa\to \infty$, which is 
also a necessary and sufficient condition for the weak limit 
$\ov\sigma_\kappa^{-1}J^{1,\kappa}\cid W$ to hold  whenever $\nu$ has 
no atoms in a neighbourhood of $0$ (see~\cite[Prop.~2.2]{MR1834755}). 
%For instance, 
%the bound does not vanish if $\nu(dx)\sim |x|^{-1}dx$ (e.g. a variance 
%gamma process) as $x\to 0$, but it does vanish if we instead have 
%$\nu(dx)\sim |x|^{-1}\log(1/|x|)dx$ and in both cases 
%Assumption~(\nameref{asm:(O)}) fails for every $\delta\in(0,2]$. 

If $X$ has a Brownian component (i.e. $\sigma\ne 0$), the bound on 
the total variation distance between the laws of $X_t$ and $X^{(\kappa)}_t$ established 
in~\cite[Prop.~8]{MR3833470} implies the following upper bound on 
the Kolmogorov distance:
$\sup_{x\in\R}|\p(X_t\le x)-\p(X_t^{(\kappa)}\le x)|
\le \min\{\sqrt{8t}\ov\sigma_\kappa,\kappa\}/\sqrt{2\pi\sigma^2t}$.
This inequality is both generalised and sharpened (as $\kappa\to0$) by the 
bound in~\eqref{eq:K-marginal-1}. Further improvements to the bound 
on the total variation were made in~\cite{Mariucci2}, but the implied 
rates for the Kolmogorov distance are worse than the ones 
in Theorem~\ref{thm:K-marginal}
and require model restrictions when $\sigma=0$ (beyond those of 
Theorem~\ref{thm:K-marginal}(b)) that can be hard to verify 
(see~\cite[Subsec.~2.1.1]{Mariucci2}). 

We stress that the dependence in $t$ in the bounds of 
Theorem~\ref{thm:K-marginal} is explicit. This is crucial in the proof of 
Theorem~\ref{thm:Wd-triplet}
as we need to apply~\eqref{eq:K-marginal-1}--\eqref{eq:K-marginal-3}
over intervals of small random lengths. 
%To the best of our knowledge 
%such a result has not been stated in the literature. 
A related result~\cite[Prop.~10]{MR3077542} contains similar 
bounds, which are  non-explicit in $t$ and suboptimal in $\kappa$. 
%The proof of 
%Theorem~\ref{thm:K-marginal} is based on the results 
%in~\cite{MR1449834,MR2479503} and the Berry-Esseen theorem. 
%Recall that the constant in the Berry-Esseen inequality satisfies $C_{\BE}<1/2$. 

If  Assumption~(\nameref{asm:(O)}) is satisfied,
the parameter $\delta$ in part (b) of Theorem~\ref{thm:K-marginal} 
should be taken as large as possible to get the sharpest inequality 
in~\eqref{eq:K-marginal-3}. If $\sigma\ne 0$ (equivalently $\delta=2$), 
the bound in part (a) has a faster decay in $\kappa$ than the bound in 
part (b). If $\sigma=0$ (equivalently $0<\delta<2$), it is possible for 
the bound in part~(a) to be sharper than the one in part~(b) or vice 
versa. Indeed, it is easy to construct a L\'evy measure $\nu$ such that 
$\delta\in(0,2)$ in Theorem~\ref{thm:K-marginal}(b) satisfies 
$\lim_{u\downarrow0}u^{\delta-2}\ov\sigma_u^2
=\inf_{u\in(0,1]}u^{\delta-2}\ov\sigma_u^2=1$. Then
%$\sup_{(0,1]}u^{q-2}\ov\sigma_u^2<\infty$. Then 
the bound in~\eqref{eq:K-marginal-1} is a multiple of 
$t^{-1/2}\kappa^{\delta/2}$ as $t,\kappa\to0$, while the one 
in~\eqref{eq:K-marginal-3} behaves as 
$t^{-2/(3\delta)}\kappa^{2/3}\min\{1,t^{1/3}\kappa^{-\delta/3}\}$. 
Hence one bound may be sharper than the other depending on the 
value of $\delta$, as $t$ and/or $\kappa$ tend to zero. In fact, 
we will use the bound in part (b) only when the maximal $\delta$ 
satisfying the assumption of Theorem~\ref{thm:K-marginal}(b)
is smaller than $4/3$, bounding the activity of the L\'evy measure 
around $0$ away from maximal possible activity.

%Let $K_1=1/2$ and $K_p$, $p\in(1,2]$, be the universal constants in the 
%inequality~\eqref{eq:Lp-marginal} of Theorem~\ref{thm:W-marginal}. 
Denote $x^+:=\max\{x,0\}$ for $x\in\R$. The next result quantifies the 
Wasserstein distance between the laws of the vectors 
$\un\chi_t$ and $\un\chi_t^{(\kappa)}$. 

\begin{thm}\label{thm:Wd-triplet}
For any $\kappa\in(0,1]$ and $t>0$, there exists a coupling between 
%$\un\chi_t$ and $\un\chi_t^{(\kappa)}$ 
$X$ and $X^{(\kappa)}$
on the interval $[0,t]$
	such that the following inequalities hold for $p\in\{1,2\}$: 
%\begin{align}\label{eq:Lp-chi}
%&
%\E\big[\max\big\{\big|X_t-X_t^{(\kappa)}\big|,
%	\big|\un{X}_t-\un{X}_t^{(\kappa)}\big|\big\}^p\big]^{1/p}\leq \mu_p(\kappa,t),
%	\qquad\text{where}\\
%\label{def:mu}
%&\mu_p(\kappa,t): =
%c_{\cl{p}}\min\bigg\{\sqrt{2t}\ov\sigma_\kappa,
%	\frac{K_{\cl{p}}\kappa\ov\sigma_\kappa^{2/\cl{p}}}
%		{(\ov\sigma_\kappa^2+\sigma^2)^{1/\cl{p}}}
%	\bigg\}
%	\bigg(1+\log^+\bigg(\frac{\sqrt{2t}
%			(\ov\sigma_\kappa^2+\sigma^2)^{1/\cl{p}}}
%		{K_{\cl{p}}\kappa\ov\sigma_\kappa^{2/\cl{p}-1}}
%	\bigg)\bigg)
%\end{align} 
\begin{align}\label{eq:Lp-chi}
&
\E\big[\max\big\{\big|X_t-X_t^{(\kappa)}\big|,
\big|\un{X}_t-\un{X}_t^{(\kappa)}\big|\big\}^p\big]^{1/p}
	\le \mu_p(\kappa,t),
	\qquad\text{where}\\
\label{def:mu}
	\begin{split}
		& \mu_1(\kappa,t): =
	\min\big\{2\sqrt{2t}\ov\sigma_\kappa,
	\kappa\varphi_\kappa^2\big\}
	\big(1+\log^+\big(2\sqrt{2t}(\ov\sigma_\kappa/\kappa)
		\varphi_\kappa^{-2}\big)\big), \\ %\quad\text{and, for $p\in(1,2]$,}\\
%\nonumber
& %\qquad\qquad\quad 
		\mu_2(\kappa,t):=\sqrt{2} \mu_1(\kappa,t)+
	\min\big\{\sqrt{2t}\ov\sigma_\kappa,
		K_2\kappa\varphi_\kappa\big\}
	\sqrt{1+2\log^+\big(K_2^{-1}\sqrt{2t}(\ov\sigma_\kappa/
		\kappa)\varphi_\kappa^{-1}\big)}, %\quad p\in(1,2],
%\nonumber
	\end{split}
\end{align}
%\label{def:mu}
%&\mu_p(\kappa,t): =
%	\min\big\{2\sqrt{2t}\ov\sigma_\kappa,
%	\kappa\varphi_\kappa^2\big\}
%	\big(1+\log^+\big(2\sqrt{2t}(\ov\sigma_\kappa/\kappa)
%	\varphi_\kappa^{-2}\big)\big) \\
%\nonumber
%&\qquad\qquad\quad 
%+\1_{(1,2]}(p)\cdot 
%	\min\big\{\sqrt{2t}\ov\sigma_\kappa,
%		K_2\kappa\varphi_\kappa\big\}
%	\sqrt{1+2\log^+\big(K_2^{-1}\sqrt{2t}(\ov\sigma_\kappa/
%	\kappa)\varphi_\kappa^{-1}\big)}.
%\end{align}
	with $\varphi_\kappa=\ov\sigma_\kappa/
\sqrt{\ov\sigma_\kappa^2+\sigma^2}$ and the universal 
constant $K_2$ from Theorem~\ref{thm:W-marginal}. Furthermore we have 
%The times $\un\tau_t=\un\tau_t(X)$ and 
%$\un\tau_t^{(\kappa)}=\un\tau_t(X^{(\kappa)})$ satisfy 
\begin{equation}\label{eq:mu_0}
	\E\big|\un\tau_t(X)-\un\tau_t(X^{(\kappa)})\big|
\le \mu_0^\tau(\kappa,t):=\sqrt{t}(\kappa/\ov\sigma_\kappa)
	\varphi_\kappa^{3}.
\end{equation}
Moreover, if Assumption~(\nameref{asm:(O)}) holds, then for every 
$T>0$ there exists a constant $C>0$, dependent only on 
	$(T,\delta,\sigma,\nu)$, such that for all $t\in[0,T]$ and $\kappa\in(0,1]$ we have 
\begin{align}
\label{eq:mu_tau}
&
\E|\un\tau_t(X)-\un\tau_t(X^{(\kappa)})|\leq \mu^\tau_\delta(\kappa,t),
\quad\text{where}\quad 
\psi_\kappa:= C\kappa\varphi_\kappa
\quad\text{and}\\
\label{def:mu_tau_delta}
&\mu^\tau_\delta(\kappa,t)
\!:=\begin{cases}
\min\{t,\psi_\kappa^\delta\}
	+%\frac{2\delta}{2\delta-1}
		t^{1-2/(3\delta)}\psi_\kappa^{2/3}
	\big(1-\min\big\{1,t^{-1/\delta}
		\psi_\kappa\big\}^{\delta-2/3}\big),
&\! \delta\in(0,2]\setminus\{\tfrac{2}{3}\},\\
	\min\{t,\psi_\kappa^{2/3}\}
	(1+\log^+(t\psi_\kappa^{-2/3})),
&\! \delta=\tfrac{2}{3}.
\end{cases}
\end{align}
\end{thm}%\cite[Thm~3.1]{MR1449834}

The proof of Theorem~\ref{thm:Wd-triplet}, given in 
Subsection~\ref{subsec:proofThm1} below, constructs the 
\emph{SBG coupling}
$(X,X^{(\kappa)})$, 
satisfying the above inequalities, 
in terms of the distribution functions of the marginals
$X_s$ and $X^{(\kappa)}_s$ (for $s>0$) %from~\eqref{eq:ARA} 
and the coupling used in~\eqref{eq:chi_infty}, 
see~\cite{LevySupSim} for the latter.
%More precisely, 
%the coupling of the vectors $(\un \chi_t,\un\chi_t^{(\kappa)})$ 
%is explicit in the distribution functions 
%of the variables $J^{1,\kappa}_s$ for any $s>0$ (cf.~\eqref{eq:comonotonic} below) and does not require the coupling
%$(X,Y)$ underlying~\eqref{eq:chi_infty}.
The key idea is to 
%Theorem~\ref{thm:Wd-triplet} is proved in 
%Subsection~\ref{subsec:proofThm1} and its proof rests on two main 
%tools. First, we 
couple $\un\chi_t$ and $\un\chi_t^{(\kappa)}$ so that  
they share the stick-breaking process in their respective SB representations~\eqref{eq:chi_infty}, 
%\textit{and} that the $L^2$-distance 
%between
while the increments of the associated L\'evy processes over each 
interval $[L_n,L_{n-1}]$ 
are coupled so that they minimise appropriate Wasserstein distances.
%is minimal. 
This coupling produces a bound 
on the distance between $\un\chi_t$ and $\un\chi_t^{(\kappa)}$ 
that depends only on the distances between the marginals of $X_s$ 
and $X_s^{(\kappa)}$, $s>0$, 
so that 
%Thus the second tool is the strong control on 
%the distance between both marginals provided by 
Theorems~\ref{thm:W-marginal} and~\ref{thm:K-marginal} above 
can be applied. We stress that the bound in~\eqref{eq:Lp-chi} 
cannot be obtained from Doob's $L^2$-maximal inequality 
(see, e.g.~\cite[Prop.~7.16]{MR1876169}) and 
Theorem~\ref{thm:W-marginal}: if the processes $X$ and 
$X^{(\kappa)}$ are coupled in such a way that 
$X_t-X_t^{(\kappa)}$ satisfies the inequality 
in~\eqref{eq:Lp-marginal}, the difference process 
$(X_s-X_s^{(\kappa)})_{s\in[0,t]}$ need not be a martingale. 

Inequality~\eqref{eq:Lp-chi} holds without assumptions on $X$
%generalising~\eqref{eq:Lp-marginal}, 
and is at most a logarithmic factor worse than the marginal inequality~\eqref{eq:Lp-marginal} for $p\in\{1,2\}$,
with 
the upper bound satisfying $\mu_p(\kappa,t)\leq  2\kappa\log(1/\kappa)$ for all sufficiently small $\kappa$.
%In particular, 
 %and any $p\in[1,2]$.
Moreover, by Jensen's inequality, for all $1<p<2$ the SBA coupling satisfies the inequality
$\E[\max\{|X_t-X_t^{(\kappa)}|,
|\un{X}_t-\un{X}_t^{(\kappa)}|\}^p]^{1/p}
	\le \mu_2(\kappa,t)$.
%with the right-hand side equal to $\mu_2(\kappa,t)$.
%on the increment 
In the absence of a Brownian component (i.e. $\sigma=0$)
we have $\varphi_\kappa=1$, making the upper bound $\mu_2(\kappa,t)$ 
proportional to $\mu_1(\kappa,t)$ as $\kappa\to0$.
If $\sigma>0$, 
then $\mu_1(\kappa,t)\leq 2\kappa\ov\sigma_\kappa^2\log(1/(\kappa \ov\sigma_\kappa))/\sigma^2$
for all small $\kappa$
and, 
typically, $\mu_2(\kappa,t)$ is proportional 
to $\kappa \ov\sigma_\kappa \tsqrt{\log(1/(\kappa\ov\sigma_\kappa))}$ 
as $\kappa\to0$, %for $p\in(1,2]$, 
which dominates $\mu_1(\kappa,t)$. 
%in~\eqref{eq:Lp-marginal} when $p\in\{1,2\}$. 
%The same is also 
%true for all $p\in(1,2)$ when $\sigma=0$ (and hence 
%$\varphi_\kappa=1$). In fact, the condition $\sigma=0$ implies 
%that $\mu_1$ is the dominant component of $\mu_p$. 
%Otherwise (i.e. if $\sigma>0$), 

The bound in~\eqref{eq:mu_0} holds without assumptions on the L\'evy process $X$, 
while~\eqref{eq:mu_tau} requires Assumption~(\nameref{asm:(O)}) and is sharper 
%is sharper than~\eqref{eq:mu_tau}:
%Note that the bound in~\eqref{eq:mu_tau} is sharper 
the larger the value of 
$\delta\in(0,2]$, 
satisfying~(\nameref{asm:(O)}), is. 
Note that, if $\sigma\ne 0$, (\nameref{asm:(O)}) holds with 
$\delta=2$. If $\sigma= 0$
and $\delta$ satisfies~(\nameref{asm:(O)}),
we must have 
$\beta\ge\delta$, where
$\beta$ is 
the \emph{Blumenthal-Getoor (BG) index}
defined in~\eqref{eq:I0_beta} above. 
In fact, models typically used in applications either  have
$\sigma\neq0$ or~(\nameref{asm:(O)}) 
holds with $\delta=\beta$ (however, it is possible 
for~(\nameref{asm:(O)}) to hold for some $\delta<\beta$ but not $\delta=\beta$,
cf.~\cite[p.~362]{MR3185174}).

If $\sigma>0$, 
%we may take $\delta$ 
%in~(\nameref{asm:(O)})
%equal to $2$,
%but nevertheless 
the inequality in~\eqref{eq:mu_0}  is sharper than~\eqref{eq:mu_tau}, i.e.
$\mu_0^\tau(t,\kappa)\le \mu_2^\tau(t,\kappa)$ 
for all small
$\kappa>0$. 
However,
if $\sigma=0$ and $\delta\in(0,2)$ satisfies~(\nameref{asm:(O)}),
then typically 
$ \mu_0^\tau(\kappa,t)$
is proportional to 
$\kappa^{\delta/2}$,
while 
$\mu_\delta^\tau(\kappa,t)$
behaves as 
$\kappa^{\min\{2/3,\delta\}}(1+\log(1/\kappa)\cdot\1_{\{2/3\}}(\delta))$
as $\kappa\to0$,
implying that~\eqref{eq:mu_tau} is sharper than~\eqref{eq:mu_0} 
for $\delta<4/3$.
%$\ov\sigma_\kappa^2\ge\varsigma^2\kappa^{2-\delta}$ 
%where
%$\varsigma^2:=\inf_{u\in(0,1]}u^{\delta-2}\ov\sigma^2_{u}>0$. 
%Moreover,
%for every $t>0$ there exists a constant $c_t>0$, such that the following inequalities hold:
%\[
%\mu_0^\tau(\kappa,t)
%\le \varsigma^{-1}\sqrt{t}\kappa^{\delta/2}
%\quad\&\quad
%\mu_\delta^\tau(\kappa,t)
%\le c_t\kappa^{\min\{2/3,\delta\}}(1+\log(1/\kappa)
%\cdot\1_{\{2/3\}}(\delta)) \quad\text{for all }\kappa\in(0,1],
%\]
%indicating that 
%$\mu_\delta^\tau(\kappa,t)$
%may decay faster as $\kappa\to0$ 
%than 
%$\mu_0^\tau(\kappa,t)$.
%for larger values of $\delta$. 
%If $\inf_{u\in(0,1]}u^{\delta-2}(\ov\sigma^2_u+\sigma^2)>0$ for 
%$\delta=4/3$, then $\sup_{\kappa\in(0,1]}\mu_0^\tau(\kappa,t)/
%	\mu_\delta^\tau(\kappa,t)<\infty$, otherwise, 
%$\sup_{\kappa\in(0,1]}\mu_\delta^\tau(\kappa,t)/
%	\mu_0^\tau(\kappa,t)<\infty$. 
The following quantity is the smallest of the 
upper bounds in~\eqref{eq:mu_0} and~\eqref{eq:mu_tau}:
\begin{equation*}%\label{def:mu_tau}
\mu^\tau_*(\kappa,t)
:=\min\big\{\mu_0^\tau(\kappa,t), 
	\inf\big\{\mu_\delta^\tau(\kappa,t)
		:\delta\in(0,2]\enskip
		\text{satisfies Assumption~(\nameref{asm:(O)})}\big\}\big\}.
\end{equation*}
Under Assumption~(\nameref{asm:(O)}),  
for some constant $c_t>0$ and all $\kappa\in(0,1]$,
we have 
\begin{equation}
	\label{eq:bound_mu_tau}
\mu^\tau_*(\kappa,t)\leq c_t
\kappa^{\max\{\delta/2, \min\{2/3,\delta\}\}}(1+\log(1/\kappa)\cdot\1_{\{2/3\}}(\delta)). 
\end{equation}

%Theorem~\ref{thm:Wd-triplet} can be applied to develop bounds on 
%the Wasserstein distance~\cite[Def.~6.1]{MR2459454} between 
%either $\un\chi_t$ and $\un\chi_t^{(\kappa)}$ or $(X_t,\un X_t)$ 
%and $(X_t^{(\kappa)},\un X_t^{(\kappa)})$. Indeed, 

For any $a\in\R^d$, let $|a|:=\sum_{i=1}^d|a_i|$ denote its $\ell^1$-norm. 
Recall that for 
$p\ge 1$, the $L^p$-Wasserstein distance~\cite[Def.~6.1]{MR2459454} 
between the laws of random vectors $\xi$ and 
$\zeta$ in $\R^d$ can be defined as 
\begin{equation}
\label{eq:Wasserstein}
\W_p(\xi,\zeta) 
	:= \inf\big\{\E\big[|\xi'-\zeta'|^p\big]^{1/p}: %\enskip
%	:\enskip
	\xi'\eqd\xi,\zeta'\eqd\zeta\big\}.
\end{equation}
%where the infimum is taken over all couplings $(\xi',\zeta')$ with 
%$\xi'\eqd\xi$ and $\zeta'\eqd\zeta$. 
%In the context of $\R^d$, we substitute the absolute value with the 
%$\ell^1$-norm in $\R^d$. With these conventions, we have the next 
Theorem~\ref{thm:Wd-triplet} implies a bound on the 
$L^p$-Wasserstein distance 
between the vectors $\un\chi_t$ and $\un\chi_t^{(\kappa)}$,
extending the bound on the distance between the laws of the marginals $X_t$ and $X_t^{(\kappa)}$ in~\cite[Thm~9]{MR3833470}.

\begin{cor}\label{cor:Wd-chi}
Fix $\kappa\in(0,1]$ and $t>0$. Then 
$\W_p((X_t,\un X_t),(X_t^{(\kappa)},\un X_t^{(\kappa)}))
	\le 2(\1_{\{p=1\}}\mu_{1}(\kappa,t)+\1_{\{1<p\leq2\}}\mu_{2}(\kappa,t))$ for 
$p\in[1,2]$ 
	and  
	$\W_p(\un\tau_t(X),\un\tau_t(X^{(\kappa)})) \le t^{1-1/p}\mu_*^\tau (\kappa,t)^{1/p}$ 
for
$p\ge 1$. 
Moreover,
$$\W_p(\un\chi_t,\un\chi_t^{(\kappa)}) \le 
	2^{2-1/p}(\1_{\{p=1\}}\mu_{1}(\kappa,t)+\1_{\{1<p\leq2\}}\mu_{2}(\kappa,t)) +
	(2t)^{1-1/p}\mu_*^\tau (\kappa,t)^{1/p}, \quad p\in[1,2].$$ 
\end{cor}

Given the bounds in Corollary~\ref{cor:Wd-chi} and 
Theorem~\ref{thm:K-marginal}, it is natural to inquire about the 
convergence in the Kolmogorov distance of the components of
$(\un{X}_t^{(\kappa)}, \un\tau_t(X^{(\kappa)}))$ to 
$(\un{X}_t,\un\tau_t(X))$ as $\kappa\to0$.
This question is addressed by Corollary~\ref{cor:K-sup-tau} of 
Subsection~\ref{subsec:Kolmogorov} below. 

The famous K\'omlos-Major-Tusn\'ady (KMT) coupling  
is used in~\cite[Thm~6.1]{MR2759203} to bound the 
$L^2$-Wasserstein distance between the paths of $X$ and 
$X^{(\kappa)}$ on $[0,t]$ in the supremum norm, implying a bound 
on $\W_2((X_t,\un X_t),(X_t^{(\kappa)},\un X_t^{(\kappa)}))$ 
proportional to $\kappa\log(1/\kappa)$ as $\kappa\to 0$, 
cf.~\cite[Cor.~6.2]{MR2759203}. If $\sigma>0$, $\mu_2(\kappa,t)$ 
in~\eqref{eq:Lp-chi} is bounded by a multiple of 
$\kappa \ov\sigma_\kappa \log(1/(\kappa\ov\sigma_\kappa))$
for small $\kappa$ and is thus smaller than a multiple of 
$\kappa^{2-q/2}$ for any $q\in(\beta,2)$ (where $\beta$ is the 
BG index defined in~\eqref{eq:I0_beta} above). As mentioned above,
$\mu_2(\kappa,t)$ is bounded by a multiple of 
$\kappa\log(1/\kappa)$ for small $\kappa$ in the case $\sigma=0$. 
Unlike the SBG coupling, which underpins 
Theorem~\ref{thm:Wd-triplet}, the KMT coupling does not imply a 
bound on the distance between the times of the infima $\un\tau_t(X)$ 
and $\un\tau_t(X^{(\kappa)})$ as these are not Lipschitz functionals 
of the trajectories with respect to the supremum norm.

\begin{rem}
\label{rem:natural-drift}
The bounds on $\E|\un\tau_t(X)-\un\tau_t(X^{(\kappa)})|$ 
in Theorem~\ref{thm:Wd-triplet} and Corollary~\ref{cor:Wd-chi},
based on the SB representation in~\eqref{eq:chi_infty}, require the 
control on the expected difference between the signs of the 
components of $(X_s, X_s^{(\kappa)})$ as either $s$ or $\kappa$ 
tend to zero. This is achieved via the minimal transport coupling 
(see~\eqref{eq:comonotonic} and Lemma~\ref{lem:comonotonic_ind} 
below) and a general bound in Theorem~\ref{thm:K-marginal} on the 
Kolmogorov distance. However, further improvements seem possible 
in the finite variation case if the \emph{natural drift} (i.e. the drift of 
$X$ when small jumps are not compensated) is nonzero. Intuitively, 
the sign of the natural drift determines the sign of both components 
of $(X_s, X_s^{(\kappa)})$ with overwhelming likelihood as $s\to0$. 
This suggestion is left for future research. 
\end{rem}

\subsection{Bounds on the bias of functions of $\un\chi_t$}
\label{subsec:bias}

The main tool for studying the bias of various Lipschitz, 
non-Lipschitz and discontinuous functions of $\un\chi_t$ 
is the SBG coupling underpinning Theorem~\ref{thm:Wd-triplet}. 
The Lipschitz case is a direct consequence: for any $d\in\N$, 
let $\Lip_K(\R^d)$ denote the space of real-valued Lipschitz 
functions on $\R^d$ (under $\ell^1$-norm given above 
display~\eqref{eq:Wasserstein}) with Lipschitz constant $K\geq0$ 
and note that the triangle inequality and 
Theorem~\ref{thm:Wd-triplet} imply the following bounds on the bias 
\begin{equation}
\label{eq:summary_1}
\big|\E f(X_T,\un{X}_T)
	-\E f\big(X^{(\kappa)}_T,\un{X}^{(\kappa)}_T\big)\big| 
	\leq 2K\mu_1(\kappa,T)\>\>\&\>\>
%	\quad\text{and}\quad
	\big|\E g(\un\tau_T) -\E g\big(\un\tau_T(X^{(\kappa)})\big)\big|	
	\leq K'\mu_*^\tau(\kappa,T)
\end{equation}
for any time horizon $T>0$ and $f\in\Lip_K(\R^2)$, such that $\E|f(X_T,\un{X}_T)|<\infty$, and $g\in\Lip_{K'}(\R)$. 
Since in applications, the process $X$ is often used to model log-returns of a risky asset $(S_0 e^{X_t})_{t\geq0}$,
it is natural to study the bias of a Monte Carlo estimator of a 
locally Lipschitz function 
%on the log-return process $X$. Moreover, it can be used 
%to obtain bounds on two classes of functions of interest in 
%applications: functions that are Lipschitz in $S$ and barrier-type 
%functions.
%
%We write 
$f\in\locLip_K(\R^2)$,
satisfying
$|f(x,y)-f(x',y')|\leq K\big(\big|e^x-e^{x'}\big|+\big|e^y-e^{y'}\big|\big)$
	for any $x,x',y,y'\in\R$
(equivalently 
%$f\in\locLip_K(\R^2)$ if and only if 
$(x,y)\mapsto f(\log x,\log y)$ is in 
$\Lip_K((0,\infty)^2)$).
Such payoffs arise in risk management 
(e.g. absolute drawdown) and in the pricing of hindsight call, 
perpetual American call and  lookback put options.
%	\quad\text{for any}\quad x,x',y,y'\in\R.
%\]

\begin{prop}\label{prop:logLip}
Let $f\in\locLip_K(\R^2)$ and assume 
$\int_{[1,\infty)}e^{2x}\nu(dx)<\infty$,
where $\nu$ is
the L\'evy measure of $X$. 
%	$\E[e^{2X_T}]<\infty$. 
	%for 
%some $q\ge 2$ and  define $p:=(1-1/q)^{-1}\in(1,2]$. 
	For any $T>0$ 
and $\kappa\in(0,1]$ and $\mu_2(\kappa,T)$ defined in~\eqref{def:mu}, the SBG coupling satisfies 
\begin{align*}
\E \big|f(X_T,\un{X}_T)
	- f\big(X^{(\kappa)}_T,\un{X}^{(\kappa)}_T\big)\big|
%&\leq 2K\W_2\big(\big(X_T,\un{X}_T\big),
%	\big(X^{(\kappa)}_T,\un{X}^{(\kappa)}_T\big)\big)
%	\sqrt{\big(1+e^{2T\ov\sigma^2_{\kappa}}\big)\E\big[e^{2X_T}\big]}\\
&\leq 4K\E[e^{2X_T}]^{1/2}(1 + e^{\ov\sigma^2_{\kappa}T})
	\mu_2(\kappa,T).
\end{align*}
\end{prop}

The assumption $\int_{[1,\infty)}e^{2x}\nu(dx)<\infty$ is equivalent 
to $\E[e^{2X_T}]<\infty$ (see~\cite[Thm~25.3]{MR3185174}), 
which is a natural requirement as  the asset price model 
$(S_0 e^{X_t})_{t\geq0}$ ought to have finite variance. Moreover, via 
the L\'evy-Khintchine formula, an explicit  bound on the expectation 
$\E[e^{2X_T}]$ (and hence the constant in the inequality of 
Proposition~\ref{prop:logLip}) in terms of the L\'evy triplet of $X$ 
can be obtained. If one instead considers $f(X_T,\ov X_T)$ 
(a function on the supremum $\ov X_T$), the proof of 
Proposition~\ref{prop:logLip} in Subsection~\ref{subsec:proofProps} 
below can be used to establish that $\E\big|f(X_T,\ov{X}_T)
	- f\big(X^{(\kappa)}_T,\ov{X}^{(\kappa)}_T\big)\big|$ is bounded 
by $4K(\E[e^{2\ov X_T}]+\E[e^{2\ov X^{(\kappa)}_T}])^{1/2}
	\mu_2(\kappa,T)$, where both expectations $\E[e^{2\ov X_T}]$ 
and $\E[e^{2\ov X^{(\kappa)}_T}]$ are finite under our assumption 
$\int_{[1,\infty)}e^{2x}\nu(dx)<\infty$ and bounded explicitly in 
terms of the L\'evy triplet of $X$, see the proof 
of~\cite[Prop.~2]{LevySupSim}. Thus, by 
Proposition~\ref{prop:logLip}, the bias for $f\in\locLip_K(\R^2)$ is 
at most a multiple of $\kappa\log(1/\kappa)$,  as is the case for 
$f\in\Lip_K(\R^2)$ by~\eqref{eq:summary_1}, cf. discussion after 
Theorem~\ref{thm:Wd-triplet}.

In financial markets, the class of barrier-type functions arises 
naturally: for $K,M\geq0$, $y<0$ let
\begin{equation}
\label{def:BT1}
\BT_1(y,K,M) := 
	\{f:\R^2\to\R \enskip:\enskip f(x,z)=h(x)\1_{[y,\infty)}(z),
		\enskip h\in\Lip_K(\R),
		\enskip 0\le h\leq M\}.
\end{equation} 
Note that the indicator function $\1_{[y,\infty)}$ lies in 
$\BT_1(y,0,1)$ and satisfies 
$\E[\1_{[y,\infty)}(\un{X}_T)]=\p(\un{X}_T\ge y)$. 
Moreover, a down-and-out put option payoff $x\mapsto \max\{e^k-e^x,0\}\1_{[y,\infty)}(x)$, for some constants
$y<0<k$, is in $\BT_1(y,e^k,e^k-e^y)$.
Bounding the 
bias of the estimators for 
functions in $\BT_1(y,K,M)$
requires the following regularity of the 
distribution of $\un{X}_T$ at $y$. 

\begin{asm*}[H]
\label{asm:(H)} 
Given $C,\gamma>0$ and $y<0$, the following inequality holds, 
\[
|\p(\un{X}_T\leq x+y)-\p(\un{X}_T\leq y)|\leq C|x|^\gamma
	\quad\text{for all }x\in\R.
\]
\end{asm*}

\begin{prop}\label{prop:barrier}
Let $f\in\BT_1(y,K,M)$ for some $K,M\geq0$ and $y<0$. If $y$ and 
some $C,\gamma>0$ satisfy Assumption~(\nameref{asm:(H)}), 
then for any $T>0$ and $\kappa\in(0,1]$, the SBG coupling satisfies 
\begin{equation}\label{eq:barrier}
\E \big|f(X_T,\un{X}_T)
	- f\big(X^{(\kappa)}_T,\un{X}^{(\kappa)}_T\big)\big|
\leq K\mu_1(\kappa,T)
	+ M'\min\{\mu_1(\kappa,T)^{\gamma/(1+\gamma)},
		\mu_2(\kappa,T)^{2\gamma/(2+\gamma)}\},
\end{equation}
where $M' = M\max\{(1+1/\gamma)(2C\gamma)^{1/(1+\gamma)},
	(1+2/\gamma)(C\gamma)^{2/(2+\gamma)}\}$.
\end{prop}

\begin{rem}\label{rem:barrier}
Since $\mu_1(\kappa,T)\to0$ and $\mu_2(\kappa,T)\to0$ as 
$\kappa\to0$ and $\gamma/(1+\gamma)<2\gamma/(2+\gamma)$ 
for all $\gamma>0$, the bound in~\eqref{eq:barrier} is typically 
dominated by a multiple of 
$\mu_1(\kappa,T)^{\gamma/(1+\gamma)}$, if $\sigma\ne 0$ 
and $\beta<2-\gamma$ (recall the definition of the BG index 
$\beta$ in~\eqref{eq:I0_beta}), or 
$\mu_2(\kappa,T)^{2\gamma/(1+\gamma)}$, otherwise. 
By H\"older's inequality, $f$ in~\eqref{eq:barrier} need not be 
bounded if appropriate moments of $X$ exist. 
\end{rem}

The proof of Proposition~\ref{prop:barrier} is in 
Subsection~\ref{subsec:proofProps} below. 
Assumption~(\nameref{asm:(H)}) with  $\gamma=1$ requires the 
distribution function of $\un{X}_T$ to be  locally Lipschitz at~$y$. By 
the Lebesgue differentiation theorem~\cite[Thm~6.3.3]{MR3098996}, 
any distribution function is differentiable Lebesgue-a.e., implying 
that Assumption~(\nameref{asm:(H)}) holds for $\gamma=1$ and 
a.e. $y<0$. However, there exist L\'evy processes satisfying  
Assumption~(\nameref{asm:(H)}) for countably many levels $y$ with 
$\gamma<1$, but not with $\gamma=1$, 
see~\cite[App.~B]{LevySupSim}. Proposition~\ref{prop:SimpAsmH} 
below provides simple sufficient conditions, in terms of the L\'evy 
triplet of $X$, for Assumption~(\nameref{asm:(H)}) to hold with 
$\gamma=1$ for all $y<0$. In particular, this is the case if 
$\sigma\ne 0$. 

The next class arises in the analysis of the duration of drawdown: 
for $K,M\ge 0$, $s\in(0,T)$ let:  %and set: 
\begin{equation}
\label{def:BT2}
\BT_2(s,K,M) := 
	\{f:\R^3\to\R \enskip:\enskip f(x,z,t)=h(x,z)\1_{(s,T]}(t),
		\enskip h\in\Lip_K(\R^2),
		\enskip 0\le h\leq M\}.
\end{equation}
The biases of these functions clearly include the error 
$|\p(\un\tau_T(X)>s)-\p(\un\tau_T(X^{(\kappa)})>s)|$. Analogous to 
Proposition~\ref{prop:barrier}, we require the following regularity from 
the distribution function of $\un\tau_T(X)$. 

\begin{asm*}[H$\tau$]
\label{asm:(Htau)} 
Given $C,\gamma>0$ and $s\in(0,T)$, the following inequality holds, 
\[
	|\p(\un\tau_T(X)\le s)-\p(\un\tau_T(X)\le s+t)|\le C|t|^\gamma,
\quad \text{for all }t\in\R.
\] 
\end{asm*}

\begin{prop}
\label{prop:barrier-tau}
Let Assumption~(\nameref{asm:(Htau)}) hold for some $s\in(0,T)$ 
and $C,\gamma>0$. Let $f\in\BT_2(s,K,M)$ for some $K,M\ge 0$. 
Then for all $\kappa\in(0,1]$ the SBG coupling satisfies 
\begin{equation}
\label{eq:barrier-tau}
\E \big|f(\un\chi_T)-f\big(\un\chi_T^{(\kappa)}\big)\big|
\le 2K\mu_1(\kappa,T) 
	+ M(2C\gamma)^{1/(1+\gamma)}(1+1/\gamma)
		\mu^\tau_*(\kappa,T)^{\gamma/(1+\gamma)}.
\end{equation}
\end{prop}
\begin{rem}\label{rem:barrier2}
As in Remark~\ref{rem:barrier}, the bound in~\eqref{eq:barrier-tau} 
is proportional to 
$\mu^\tau_*(\kappa,T)^{\gamma/(1+\gamma)}$ as $\kappa\to0$. 
Inequality~\eqref{eq:barrier-tau} can be generalised to unbounded 
function $f$ if appropriate moments of $X$ exist.
\end{rem}

If $X$ is not a compound Poisson process, 
then Assumption~(\nameref{asm:(Htau)}) holds with $\gamma=1$ 
for all $s\in(0,T)$, since, by Lemma~\ref{lem:density-tau} in 
Subsection~\ref{subsec:proofProps} below, $\un\tau_T(X)$ has a 
locally bounded density, making the distribution function of 
$\un\tau_T(X)$ locally Lipschitz on $(0,T)$. 
%Again, by the Lebesgue differentiation 
%theorem, the distribution function of $\un\tau_T(X)$ is 
%differentiable a.e. and hence Assumption~(\nameref{asm:(Htau)}) 
%holds for $\gamma=1$ and a.e.~$s\in(0,T)$. 
Assumption~(\nameref{asm:(Htau)}) is satisfied if either 
$\nu(\R\setminus\{0\})=\infty$ or $\sigma\ne 0$. In particular, 
Assumption~(\nameref{asm:(O)}) implies~(\nameref{asm:(Htau)}). 
The proof of Proposition~\ref{prop:barrier-tau} is in 
Subsection~\ref{subsec:proofProps} below. 

\subsection{Convergence of $\un X_T^{(\kappa)}$ and $\un\tau_T(X^{(\kappa)})$ in the Kolmogorov distance}
\label{subsec:Kolmogorov}

As a consequence of Proposition~\ref{prop:barrier} 
(resp.~\ref{prop:barrier-tau}), if Assumption~(\nameref{asm:(H)}) 
(resp.~(\nameref{asm:(Htau)})) holds uniformly, then 
$\un{X}_T^{(\kappa)}$ (resp. $\un\tau_T(X^{(\kappa)})$) converges 
to $\un{X}_T$ (resp. $\un\tau_T(X)$) in Kolmogorov distance as 
$\kappa\to0$. 

\begin{cor}
\label{cor:K-sup-tau}
\nf{(a)} Suppose $C,\gamma>0$ satisfy~(\nameref{asm:(H)}) for all 
$y<0$. Then for any $\kappa\in(0,1]$ we have 
\begin{equation}
\label{eq:K-sup}
\sup_{x\in\R}\big|\p(\un{X}_T\le x) 
	- \p\big(\un{X}_T^{(\kappa)}\le x\big)\big|
\le M'\min\{\mu_1(\kappa,T)^{\gamma/(1+\gamma)},
	\mu_2(\kappa,T)^{2\gamma/(2+\gamma)}\},
\end{equation} 
where $M' = \max\{(1+1/\gamma)(2C\gamma)^{1/(1+\gamma)},
(1+2/\gamma)(C\gamma)^{2/(2+\gamma)}\}$.
\nf{(b)} Suppose $C,\gamma>0$ satisfy~(\nameref{asm:(Htau)}) 
for all $s\in[0,T]$. Then for any $\kappa\in(0,1]$ we have 
\begin{equation}
\label{eq:K-tau}
\sup_{x\in\R}\big|\p(\un\tau_T(X)\le x) 
	- \p\big(\un\tau_T(X^{(\kappa)})\le x\big)\big|
\le (2C\gamma)^{1/(1+\gamma)}(1+1/\gamma)
	\mu_*^\tau(\kappa,T)^{\gamma/(1+\gamma)}.
\end{equation}
\end{cor}

%The proof of Corollary~\ref{cor:K-sup-tau} is in 
%Subsection~\ref{subsec:proofProps} below. 
Proposition~\ref{prop:SimpAsmH} gives sufficient conditions 
(in terms of the L\'evy triplet $(\sigma^2,\nu,b)$ of $X$) for 
Assumptions~(\nameref{asm:(H)}) and~(\nameref{asm:(Htau)}) 
to hold for all $y<0$ and $s\in[0,T]$, respectively.  
%i.e. with the same constants $C$ and $\gamma$ 
%on any compact interval in $(-\infty,0)$ and $(0,T)$, respectively.
Recall that a function $f(x)$ is said to be regularly varying 
with index $r$ as $x\to0$ if $\lim_{x\to0}f(\lambda x)/f(x)=\lambda^r$ 
for every $\lambda>0$ (see~\cite[p.~18]{MR1015093}).

\begin{prop}
\label{prop:SimpAsmH}
Let $\ov\nu_+(x):=\nu([x,\infty))$ and $\ov\nu_-(x):=\nu((-\infty,-x])$ 
for $x>0$ and let $\beta$ be the BG index of $X$ defined 
in~\eqref{eq:I0_beta} above. 
Suppose that either (I) $\sigma> 0$ or (II) the L\'evy measure $\nu$ 
satisfies the following conditions: $\ov\nu_+(x)$ is regularly varying 
with index~$-\beta$ as $x\to 0$ and \\
\nf{$\bullet$} $\beta=2$ and 
	$\liminf_{x\to 0}\ov\nu_+(x)/\ov\nu_-(x)>0$,\\
\nf{$\bullet$} $\beta\in(1,2)$ and 
	$\lim_{x\to 0}\ov\nu_+(x)/\ov\nu_-(x)\in(0,\infty]$ or \\
\nf{$\bullet$} $\beta\in(0,1)$, $b=\int_{(-1,1)}x\nu(dx)$ and 
	$\lim_{x\to 0}\ov\nu_+(x)/\ov\nu_-(x)\in(0,\infty)$.\\ 
Then there exists constants $\gamma>0$ and $C$ such that 
Assumption~(\nameref{asm:(Htau)}) holds with $\gamma,C$ for all 
$s\in[0,T]$. Either (I) or (II) with $\beta>1$ imply 
that~(\nameref{asm:(H)}) holds with $\gamma=1$ and some 
constant $C_I$ for all $y$ in a compact $I\subset (-\infty,0)$. 
\end{prop}

Note that Proposition~\ref{prop:SimpAsmH} holds if the roles of 
$\ov\nu_+$ and $\ov \nu_-$ are interchanged, i.e $\ov\nu_-(x)$ is 
regularly varying and the limit conditions are satisfied by the 
quotients $\ov\nu_-(x)/\ov\nu_+(x)$. The assumptions of 
Proposition~\ref{prop:SimpAsmH} are satisfied by most models used 
in practice, including tempered stable and most 
subordinated Brownian motion processes. Excluded are L\'evy 
processes without a Brownian component and with barely any 
jump activity (i.e. BG index $\beta=0$, which includes compound 
Poisson and variance gamma processes), where the Gaussian 
approximation $X^{(\kappa)}$ is not useful. 

Proposition~\ref{prop:SimpAsmH} is a consequence of a more 
general result, Proposition~\ref{prop:AsmH} below, stating that 
Assumptions~(\nameref{asm:(Htau)}) and~(\nameref{asm:(H)}) hold 
uniformly and locally uniformly, respectively, if over short time 
horizons, $X$ is ``attracted to'' an $\alpha$-stable process with 
non-monotone paths, see Subsection~\ref{subsec:proofProps} below 
for details. In this case $\rho:=\lim_{t\to 0}\p(X_t>0)$ exists in 
$(0,1)$ and $\gamma$ in the conclusion of 
Proposition~\ref{prop:SimpAsmH}, satisfying 
Assumption~(\nameref{asm:(Htau)}) on $[0,T]$, can be arbitrarily 
chosen in the interval $(0,\min\{\rho,1-\rho\})$. 
%Differently put, the assumptions of Propositions~\ref{prop:barrier} 
%and~\ref{prop:barrier-tau} hold for all levels with $\gamma=1$ 
%%and for all levels in compact sets (again with $\gamma=1$) 
%for large classes of L\'evy processes. 
In contrast to Assumption~(\nameref{asm:(Htau)}), a simple 
sufficient condition for the uniform version of 
Assumption~(\nameref{asm:(H)}), required in 
Corollary~\ref{cor:K-sup-tau}(a), remains elusive beyond special 
cases such as stable or tempered stable processes with $\gamma$ 
in the interval $(0,\alpha(1-\rho))$, where $\alpha$ is the stability 
parameter and $\rho$ is as above. 

%In Propositions~\ref{prop:barrier} and~\ref{prop:barrier-tau} we 
%required the function $f$ to be bounded by $M$. This restriction may 
%be lifted in exchange for some integrability condition and a slower 
%decay rate (as an application of H\"older inequality would be 
%required, see Remark~\ref{rem:Lp-to-barrier} below). However, to 
%prevent the presentation of these results becoming any more 
%convoluted, adopted a simpler assumption. Moreover, 
%the bounds on the bias of barrier-type functions rely on general 
%lemmas and assumptions and may therefore miss possible 
%improvements for commonly used classes of well behaved L\'evy 
%processes and specific functions 
%(see Figure~\ref{fig:TS} below). 
%For instance,~\cite{ZoomIn} found such improvements for the 
%functions in $\BT_1(y,0,1)$ when the infimum is approximated by 
%the exact skeleton of the process on a uniform grid and the L\'evy 
%%measure is regularly varying at $0$, 
%contributing to the results in~\cite{MR1482707,MR2867949} which 
%handle the purely Brownian and jump diffusion cases for certain 
%functions in the class $\BT_1(y,K,M)$. 

\section{Simulation and the computational complexity of MC and MLMC} 
\label{sec:main-apps}

This section describes an MC/MLMC  method for the simulation of 
$\un\chi_T^{(\kappa)}
= (X_T^{(\kappa)},\un{X}_T^{(\kappa)},\un\tau_T(X^{(\kappa)}))$
(\nameref{alg:SBG} in Subsection~\ref{subsec:algorithms}) 
and analyses the computational complexities for various locally 
Lipschitz and discontinuous functions of $\un\chi_T^{(\kappa)}$ 
(Subsection~\ref{subsec:complexities}). The numerical performance 
of \nameref{alg:SBG}, which is based on the SB representation 
in~\eqref{eq:chi_infty}-\eqref{eq:chi} of $\un\chi_T^{(\kappa)}$, 
is far superior to that of the ``obvious'' algorithm for jump diffusions 
(see Algorithm~\ref{alg:ARA_2} below), particularly when the jump 
intensity is large (cf. Subsections~\ref{subsubsec:Error_term} 
and~\ref{subsec:SBG_sampler}). Moreover, \nameref{alg:SBG}
is designed with MLMC in mind, which turns out not to be feasible 
in general for the ``obvious'' algorithm (see 
Subsections~\ref{subsubsec:Error_term}). 

\subsection{Simulation of $\un\chi_T^{(\kappa)}$}
\label{subsec:algorithms}

The main aim of the subsection is to develop a simulation algorithm for the pair of vectors
$(\un\chi_T^{(\kappa)},\un\chi_T^{(\kappa')})$
at levels $\kappa,\kappa'\in(0,1]$
over a time horizon $[0,T]$,
such that the $L^2$-distance between 
$\un\chi_T^{(\kappa)}$ and $\un\chi_T^{(\kappa')}$
tends to zero as $\kappa,\kappa'\to0$.
\nameref{alg:SBG} below, based on the SB representation in~\eqref{eq:chi}, achieves this aim: 
it applies Algorithm~\ref{alg:ARA} for the increments over the stick-breaking lengths that arise in~\eqref{eq:chi} 
and Algorithm~\ref{alg:ARA_2} for the ``error term''
over the time horizon $[0,L_n]$. 
By Theorem~\ref{thm:summary} below
the $L^2$-distance  for the coupling given in 
\nameref{alg:SBG} decays to zero,
ensuring the feasibility of MLMC (see Theorem~\ref{thm:SBG_MLMC} for the computational complexity of MLMC). 
%Additionally, \nameref{alg:SBG} 
%requires a sampler of the  

%Recall that, by definition there is no specific dependence structure between neither 
%the L\'evy process and its Gaussian approximation nor between multiple Gaussian 
%approximations at different levels. Indeed, this is a crucial distinction when 
%constructing MLMC estimators, for instance. For this reason, and to avoid confusion, 
%we will construct processes whose paths have the same distributions as those of 
%the Gaussian approximation, such that the dependence structure is explicit and with 
%as little $L^2$-distance as possible. The subsection is structured as follows. 
%We first present a coupling between the increments of two Gaussian approximations. 
%Next, we explain the problems arising in the simulation of increments, jointly with the 
%minima, of the respective Gaussian approximations over the given time interval. 
%We define the SBG approximation in order to circumvent these difficulties. 
%All corresponding algorithms 
%are presented in this section. Moreover, the algorithms requiring only a single level $\kappa_2$ 
%(i.e. when the variables with parameter $\kappa_1$ are not necessary) may be 
%recovered by disregarding all the operations associated to the parameter $\kappa_1$.

\subsubsection{Increments in the SB representation} 
A simulation algorithm for a coupling 
$\big(X^{(\kappa_1)}_t,X^{(\kappa_2)}_t\big)$ of Gaussian 
approximations (at levels $1\geq \kappa_1>\kappa_2>0$) of $X_t$ 
at an arbitrary time $t>0$ is based on the following observation: 
the compound Poisson processes $J^{2,\kappa_1}$ and 
$J^{2,\kappa_2}$ in the L\'evy-It\^o decomposition 
in~\eqref{eq:levy-ito} can be simulated jointly, as the jumps of 
$J^{2,\kappa_1}$ are precisely those of $J^{2,\kappa_2}$ with 
modulus of at least $\kappa_1$. By choosing the same Brownian 
motion $W$ in representation~\eqref{eq:ARA} of $X^{(\kappa_1)}_t$ 
and $X^{(\kappa_2)}_t$, we obtain the coupling 
$\big(X^{(\kappa_1)}_t,X^{(\kappa_2)}_t\big)$ with law 
$\Pi_t^{\kappa_1,\kappa_2}$ given in Algorithm~\ref{alg:ARA}.

\begin{lyxalgorithm} 
Simulation of the law $\Pi_t^{\kappa_1,\kappa_2}$
\label{alg:ARA}
\begin{algorithmic}[1]
	\Require{Cutoff levels $1\geq\kappa_1>\kappa_2>0$ and time horizon $t>0$.}
	\State{Compute $b_{\kappa_i}$ and $\ov\sigma^2_{\kappa_i}$ 
		for $i\in\{1,2\}$ and $\ov\nu(\kappa_2)$}
	\State{Sample $W_t\sim N(0,t)$, $N_t\sim\Poi(\ov{\nu}(\kappa_2)t)$ and 
		$\lambda_k\sim \nu(\cdot\setminus(-\kappa_2,\kappa_2))/\ov{\nu}(\kappa_2)$ 
		for $k\in\{1,\ldots,N_t\}$\label{alg:Poison_number_line2}}
	\State{Put $J^{2,\kappa_i}_t:=\sum_{k=1}^{N_t}\lambda_k\cdot \1{\{|\lambda_k|\geq\kappa_i\}}$ for $i\in\{1,2\}$\label{alg_line:sums_of_jumps}}
	%and 
	%$J^{2,\kappa_2}_t=\sum_{k=1}^{N_t}\lambda_k$\label{alg_line:sums_of_jumps}}
	\State{\Return $\big(Z^{(\kappa_1)}_t,Z^{(\kappa_2)}_t\big)$, where 
	$Z^{(\kappa_i)}_t:=b_{\kappa_i}t +\tsqrt{\sigma^2+\ov\sigma^2_{\kappa_i}}W_t +J^{2,\kappa_i}_t$ for $i\in\{1,2\}$}
\end{algorithmic}
\end{lyxalgorithm}

Since $Z^{(\kappa_i)}_t\eqd X^{(\kappa_i)}_t$, $i\in\{1,2\}$, 
Proposition~\ref{prop:SBG_app_coupling}(a) below implies that 
the coupling $\Pi_t^{\kappa_1,\kappa_2}$ provides a bound on the 
$L^2$-Wasserstein distance 
$\W_2\big(X^{(\kappa_1)}_t,X^{(\kappa_2)}_t\big)
\le (2t(\ov\sigma^2_{\kappa_1}-\ov\sigma^2_{\kappa_2}))^{1/2}$.
This bound is suboptimal as the variables 
$J^{2,\kappa_2}_t-J^{2,\kappa_1}_t$ and 
$(\ov\sigma^2_{\kappa_2}-\ov\sigma^2_{\kappa_1})^{1/2}W_t$ in 
Algorithm~\ref{alg:ARA} are independent. The minimal transport 
coupling, with $L^2$-distance equal to 
$\W_2\big(X^{(\kappa_1)}_t,X^{(\kappa_2)}_t\big)$, is not 
accessible from the perspective of simulation. Since the law 
$\Poi(\ov{\nu}(\kappa_2)t)$ of the variable $N_t$ in 
line~\ref{alg:Poison_number_line2} of Algorithm~\ref{alg:ARA} is 
Poisson with mean $\ov\nu(\kappa_2)t$, the expected number of 
steps of Algorithm~\ref{alg:ARA} is bounded by a constant multiple 
of $1+\ov{\nu}(\kappa_2)t$, which is in turn bounded by a negative 
power of $\kappa_2$ by~\eqref{eq:BG_bounds}. Since  the 
computational complexity of sampling the law of $X^{(\kappa_2)}_t$ 
is of the same order as that of the law $\Pi_t^{\kappa_1,\kappa_2}$, 
in the complexity analysis of \nameref{alg:SBG} below, we may 
apply Algorithm~\ref{alg:ARA} with $\Pi_t^{1,\kappa}$ to sample 
$X^{(\kappa)}_t$ for any $\kappa\in(0,1]$.

\subsubsection{``Error term'' in the SB representation}
\label{subsubsec:Error_term}
Algorithm~\ref{alg:ARA_2} samples from
the law 
$\un\Pi_t^{\kappa_1,\kappa_2}$
of a coupling 
$(\un\chi_t^{(\kappa_1)},\un\chi_t^{(\kappa_2)})$
for levels 
$0<\kappa_2<\kappa_1\leq1$
and any (typically very small) $t>0$. 
%applied at the random time $L_n$, see~\eqref{eq:chi}.
%The law 
%$\un\Pi_t^{\kappa_1,\kappa_2}$,
%defined by 
%Algorithm~\ref{alg:ARA_2},
In particular, it requires the sampler~\cite[Alg.~{\small{MAXLOCATION}}]{MR2730908} 
for the law  
$\Phi_t(v,\mu)$ of $(B^*_t,\un B^*_t,\tau_t(B^*))$ 
where
$(B^*_s)_{s\ge 0}=(v B_s+\mu s)_{s\ge 0}$
is a Brownian motion with drift $\mu\in\R$ and volatility $v>0$.

\begin{lyxalgorithm} 
Simulation of the law $\un\Pi_t^{\kappa_1,\kappa_2}$
\label{alg:ARA_2}
\begin{algorithmic}[1]
	\Require{Cutoff levels $1\geq\kappa_1>\kappa_2>0$ and time horizon $t>0$.}
	\State{Compute $b_{\kappa_i}$, $\ov\sigma^2_{\kappa_i}$ and 
		$\upsilon_{\kappa_i} :=\tsqrt{\sigma^2+\ov\sigma^2_{\kappa_i}}$ for 
		$i\in\{1,2\}$ and $\ov\nu(\kappa_2)$, see~\eqref{eq:levy-ito}--\eqref{eq:ARA}}
	\State{Sample $N_t\sim\Poi(\ov\nu(\kappa_2)t)$ and 
		$U_k\sim\U(0,t)$ for $k\in\{1,\ldots,N_t+1\}$}
	\State{Set $s := \sum_{i=1}^{N_t+1}\log U_k$ and let 
		$t_k:=s^{-1}\sum_{i=1}^k\log U_i$ for $k\in\{0,\ldots,N_t+1\}$}
	\label{alg_line:sorting_U}
	\State{Set 
		$(Z^{(\kappa_1)}_0,\un{Z}^{(\kappa_1)}_0,\un\tau_0^{(\kappa_1)},
		Z^{(\kappa_2)}_0,\un{Z}^{(\kappa_2)}_0,\un\tau_0^{(\kappa_2)}):=(0,0,0,0,0,0)$}
	\For{$k\in\{1,\ldots,N_t+1\}$\label{alg_ling:for_start}}
	\State{Sample $\lambda_k\sim
		\nu(\cdot\setminus(-\kappa_2,\kappa_2))/\ov\nu(\kappa_2)$ if $k\leq N_t$
		and otherwise put $\lambda_k=0$}
	\State{Let $\delta_k:=t_k-t_{k-1}$ and sample 
		$(\Delta_{k,i}^1,\Delta_{k,i}^2, \Delta_{k,i}^3)\sim
		\Phi_{\delta_k}(\upsilon_{\kappa_i},b_{\kappa_i})$ independently for $i\in\{1,2\}$ \label{alg_ling:indep_samples_BM_triplet}}
	\For{$i\in\{1,2\}$}
	\If{$\un{Z}^{(\kappa_i)}_{t_{k-1}}
			> Z^{(\kappa_i)}_{t_{k-1}}+\Delta_{k,i}^2$ \label{alg_line:condition>}}
	\State{Set $(Z_{t_k}^{(\kappa_i)},\un{Z}^{(\kappa_i)}_{t_k},\un\tau^{(\kappa_i)}_{t_k}):=(Z_{t_{k-1}}^{(\kappa_i)}+\Delta_{k,i}^1
			+\lambda_k\cdot\1{\{|\lambda_k|\geq\kappa_i\}},
		Z^{(\kappa_i)}_{t_{k-1}}+\Delta_{k,i}^2,t_{k-1}+\Delta_{k,i}^3)$}
	\Else
	\State{Set $(Z^{(\kappa_i)}_{t_k},\un{Z}^{(\kappa_i)}_{t_k},\un\tau^{(\kappa_i)}_{t_k}):=(Z_{t_{k-1}}^{(\kappa_i)}+\Delta_{k,i}^1
			+\lambda_k\cdot\1{\{|\lambda_k|\geq\kappa_i\}},
		\un{Z}^{(\kappa_i)}_{t_{k-1}},\un\tau^{(\kappa_i)}_{t_{k-1}})$}
	\EndIf
	\EndFor
	\EndFor \label{alg_ling:for_end}
	\State{\Return $(\un\zeta^{(\kappa_1)},\un\zeta^{(\kappa_2)})$, where $\un\zeta^{(\kappa_i)}:=(Z_t^{(\kappa_i)},\un{Z}_t^{(\kappa_i)}, \un\tau_t^{(\kappa_i)})$ 
	for $i\in\{1,2\}$
%	$ $(X_t^{(\kappa_1)},\un{X}_t^{(\kappa_1)},
%		\un\tau_t^{(\kappa_1)},X_t^{(\kappa_2)},\un{X}_t^{(\kappa_2)}, \un\tau_t^{(\kappa_2)})$
		\label{alg_line:funny_coupling}}
\end{algorithmic}
\end{lyxalgorithm}

Algorithm~\ref{alg:ARA_2} samples the jump times and sizes of the 
compound Poisson process $J^{2,\kappa_2}$ on the interval $(0,t)$ 
and prunes the jumps to get $J^{2,\kappa_1}$. Then it samples the 
increment,infimum and the time the infimum is attained for 
the Brownian motion with drift on each interval between the jumps 
of $J^{2,\kappa_2}$ and assembles the pair 
$(\un\zeta^{(\kappa_1)},\un\zeta^{(\kappa_2)})$, clearly satisfying 
$\un\zeta^{(\kappa_i)}\eqd\un\chi_t^{(\kappa_i)}$, $i\in\{1,2\}$.
Since~\cite[Alg.~{\small{MAXLOCATION}}]{MR2730908} samples 
the law $\Phi_t(v,\mu)$ 
%where 
%$(B^*_s)_{s\ge 0}=(\sigma B_s+\mu s)_{s\ge 0}$, 
with uniformly bounded 
expected runtime over the choice of parameters $\mu$, $v$ and $t$, 
the computational cost of sampling the pair of vectors 
$(\un\chi_t^{(\kappa_1)},\un\chi_t^{(\kappa_2)})$
using  Algorithm~\ref{alg:ARA_2}
is proportional to to the cost of sampling  $X^{(\kappa)}_t$
via Algorithm~\ref{alg:ARA}.

In principle, Algorithm~\ref{alg:ARA_2} is an exact algorithm for the 
simulation of a coupling 
$(\un\chi_t^{(\kappa_1)},\un\chi_t^{(\kappa_2)})$. However, it cannot be applied
within an MLMC simulation scheme 
for  a function of
$\un\chi_T^{(\kappa)}$ 
at a fixed time horizon $T$
(the next paragraph explains why).  
%there is 
%one significant issue with Algorithm~\ref{alg:ARA_2}, described in the 
%next paragraph, which makes it hard to apply this algorithm directly 
%in the context of $(\un\chi_T^{(\kappa_1)},\un\chi_T^{(\kappa_2)})$ 
\nameref{alg:SBG} below 
circumvents this issue via the SB representation in~\eqref{eq:chi},
which also makes \nameref{alg:SBG} \textit{paralellizable} and 
thus much faster in practice even in the context of MC simulation 
(see the discussion after Corollary~\ref{cor:SBG_MC} below). 

%\noindent \underline{(a) Level variance in MLMC.}
%Furthermore, we do not know how to simulate the tuple 
%$\big(\zeta_t^{(\kappa_1)},\un\zeta_t^{(\kappa_1)},
%	\un\tau_t(\zeta^{(\kappa_1)}),
%\zeta_t^{(\kappa_2)},\un\zeta_t^{(\kappa_2)},
%	\un\tau_t(\zeta^{(\kappa_2)})\big)$ 
To the best of our knowledge, there is no simulation algorithm for 
the increment, the infima and the times the infima are attained of a 
Brownian motion under different drifts, i.e. of the vector
\begin{equation*}%\label{eq:BM-two-min}
\big(B_t,\un{B}^{(c_1)}_t,\un\tau_t(B^{(c_1)}),
	\un{B}^{(c_2)}_t,\un\tau_t(B^{(c_2)})\big),
\quad\text{where}\quad 
	(B_s^{(c)})_{s\ge 0}=(B_s+cs)_{s\ge 0}
\quad\text{and}\quad 
	c_1\neq c_2.
\end{equation*}
Thus,
in line~\ref{alg_ling:indep_samples_BM_triplet} of 
Algorithm~\ref{alg:ARA_2}, we are forced to take independent 
samples from $\Phi_{\delta_k}(\upsilon_{\kappa_1},b_{\kappa_1})$
and $\Phi_{\delta_k}(\upsilon_{\kappa_2},b_{\kappa_2})$ at each 
step $k$. In particular, the coupling  of the marginals 
$X_t^{(\kappa_1)}$ and $X_t^{(\kappa_2)}$ of 
$\un\Pi_t^{\kappa_1,\kappa_2}$, given in 
line~\ref{alg_line:funny_coupling} of Algorithm~\ref{alg:ARA_2}, 
amounts to taking two independent Brownian motions in the 
respective representations in~\eqref{eq:ARA} of $X_t^{(\kappa_1)}$ 
and $X_t^{(\kappa_2)}$. Thus, unlike the coupling defined in 
Algorithm~\ref{alg:ARA}, here, 
by Proposition~\ref{prop:SBG_app_coupling}(b) below, 
the squared $L^2$-distance satisfies 
$\E[ (X_t^{(\kappa_1)}-X_t^{(\kappa_2)})^2]\ge 2t \sigma^2$ for all 
levels $1\geq \kappa_1>\kappa_2>0$, where $\sigma^2$ is the 
Gaussian component of $X$. Hence, for a fixed time horizon, 
the coupling $\un\Pi_t^{\kappa_1,\kappa_2}$ of 
$\un\chi_t^{(\kappa_1)}$ and $\un\chi_t^{(\kappa_2)}$ is not 
sufficiently strong for an MLMC scheme to be feasible if $X$ has a 
Gaussian component, because the level variances do not decay to 
zero. However, by Proposition~\ref{prop:SBG_app_coupling}(b), 
the $L^2$-distance between $\un\zeta^{(\kappa_1)}$ and 
$\un\zeta^{(\kappa_2)}$ constructed in Algorithm~\ref{alg:ARA_2}
does tend to zero with $t\to0$. Thus, \nameref{alg:SBG} below, 
which applies Algorithm~\ref{alg:ARA_2} over the time interval 
$[0,L_n]$ (recall $\E L_n=T/2^n$ from SB representation~\eqref{eq:chi}), circumvents this 
issue.

\subsubsection{The SBG sampler} 
\label{subsec:SBG_sampler}
For a time horizon $T$, we can now define the coupling 
$\un\Pi_{n,T}^{\kappa_1,\kappa_2}$ of the vectors 
$\un\chi^{(\kappa_1)}_T$ and $\un\chi^{(\kappa_2)}_T$ via the 
following algorithm.

\begin{lyxalgorithm*}[SBG-Alg]
Simulation of the coupling 
$(\un\chi^{(\kappa_1)}_T,\un\chi^{(\kappa_2)}_T)$ 
with law $\un\Pi_{n,T}^{\kappa_1,\kappa_2}$
\label{alg:SBG}
\begin{algorithmic}[1]
	\Require{Cutoff levels $1\geq\kappa_1>\kappa_2>0$, 
		number of sticks $n\in\N\cup\{0\}$ and time horizon  $T>0$.}
	\State{Set $L_0:=T$, sample $U_k\sim\U(0,1)$, put $\ell_k:=L_{k-1}U_k$ 
		and $L_k:=L_{k-1}-\ell_k$ for $k\in\{1,\ldots,n\}$ \label{alg_line:Sticks}}
	\State{Sample $\big(\xi_{k,1},\xi_{k,2}\big)\sim
		\Pi_{\ell_k}^{\kappa_1,\kappa_2}$ for $k\in\{1,\ldots,n\}$
		and $\big(\un\xi_1,\un\xi_2)
		\sim\un\Pi_{L_n}^{\kappa_1,\kappa_2}$ 
		\label{alg_line:samplingFrom_A1_A2}}
	\Comment{Algorithms~\ref{alg:ARA} \& \ref{alg:ARA_2}}
	\State{Put 
		%$\big(X_t^{(\kappa_i)},\un{X}_t^{(\kappa_i)},\un\tau_t^{(\kappa_i)}\big)
		$\un\chi_{n,T}^{(\kappa_i)}:=
		\un \xi_i + \sum_{k=1}^n
		\big(\xi_{k,i},\min\{\xi_{k,i},0\},\ell_k\cdot\1{\{\xi_{k,i}\le 0\}}\big)$ 
		for $i\in\{1,2\}$
		\label{alg_line:min_in_sum}}
	\State{\Return $\big(\un\chi_{n,T}^{(\kappa_1)},
		\un\chi_{n,T}^{(\kappa_2)}\big)$}
\end{algorithmic}
\end{lyxalgorithm*}

By SB representation~\eqref{eq:chi}, the law 
$\un\Pi_{n,T}^{\kappa_1,\kappa_2}$ is indeed a coupling of the 
vectors $\un\chi^{(\kappa_1)}_T$ and $\un\chi^{(\kappa_2)}_T$ 
for any $n\in\N\cup\{0\}$. Note that if $n$ equals zero, the set 
$\{1,\ldots,n\}$ in lines~\ref{alg_line:Sticks} 
and~\ref{alg_line:samplingFrom_A1_A2} of the algorithm is empty 
and the laws $\un\Pi_{0,T}^{\kappa_1,\kappa_2}$ and 
$\un\Pi_T^{\kappa_1,\kappa_2}$ coincide, implying that 
\nameref{alg:SBG} may be viewed as a generalisation of 
Algorithm~\ref{alg:ARA_2}. The main advantage of 
\nameref{alg:SBG} over Algorithm~\ref{alg:ARA_2} is that it 
samples $n$ increments of the Gaussian approximation over the 
interval $[L_n,T]$ using the fast Algorithm~\ref{alg:ARA}, with the 
``error term'' contribution $\un \xi_i$ being geometrically small.
%(since $\E L_n =T/2^n$) and accurate (see part~(b) of the discussion of Algorithm~\ref{alg:ARA_2} 
%in Subsection~\ref{subsubsec:Error_term} above).

The computational complexity of \nameref{alg:SBG} and 
Algorithms~\ref{alg:ARA} \&~\ref{alg:ARA_2} is simple to analyse. 
%Before proceeding, let us discuss briefly the complexity of all the 
%previous algorithms. For simplicity, 
Assume throughout that all mathematical operations (addition, multiplication, exponentiation, etc.), as well as the evaluation  of
$\ov\nu(\kappa)$ and $\ov\sigma^2_\kappa$ for all $\kappa\in(0,1]$ have constant computational cost. 
Moreover, assume that 
the simulation of any of the following random variables has 
constant expected cost: standard normal $N(0,1)$, uniform $\U(0,1)$, 
Poisson random variable (independently of its mean) 
%https://github.com/JuliaDiffEq/PoissonRandom.jl
and any jump with distribution 
$\nu|_{\R\setminus(-\kappa,\kappa)}/\ov\nu(\kappa)$ 
(independently of the cutoff level $\kappa\in(0,1]$). 
Recall that~\cite[Alg.~{\small{MAXLOCATION}}]{MR2730908} samples 
the law $\Phi_t(v,\mu)$ 
with uniformly bounded 
expected cost for all values of the parameters $\mu\in\R$, $v>0$ and $t>0$. 
The next statement follows directly from the algorithms. 

\begin{cor}
\label{cor:algs_complexities}
Under assumptions above, 
there exists a positive constant $C_1$ (resp. $C_2$; $C_3$),
independent of $\kappa_1,\kappa_2\in(0,1]$, $n\in\N$ and time 
horizon $t>0$, such that the expected computational 
complexity of Algorithm~\ref{alg:ARA} (resp. Algorithm~\ref{alg:ARA_2}; 
\nameref{alg:SBG}) is bounded by $C_1(1+\ov\nu(\kappa_2)t)$ 
(resp. $C_2(1+\ov\nu(\kappa_2)t)$; $C_3(n+\ov\nu(\kappa_2)t)$).
\end{cor}

Up to a multiplicative constant, 
Algorithms~\ref{alg:ARA} and~\ref{alg:ARA_2} have the same expected 
computational complexity. However,
Algorithm~\ref{alg:ARA_2} requires not only additional simulation of 
jump times of $X^{(\kappa_2)}$
and a sample from $\Phi_t(v,\mu)$ 
using~\cite[Alg.~{\small{MAXLOCATION}}]{MR2730908}   
between any two consecutive jumps, 
but also a sequential computation of the output (the ``for-loop'' in 
lines~\ref{alg_ling:for_start}-\ref{alg_ling:for_end}) due to the 
condition in line~\ref{alg_line:condition>} of 
Algorithm~\ref{alg:ARA_2}. This makes it hard to parallelise 
Algorithm~\ref{alg:ARA_2}. 
%in addition to the random variables that 
%need to be sampled in Algorithm~\ref{alg:ARA} (namely, the jumps of 
%$X_t^{(\kappa_2)}$ and the Gaussian increment) requires the 
%simulation of the jump-times of $X_t^{(\kappa_2)}$ and must sample 
%from $\Phi_t(v,\mu)$ for every jump. Moreover, the additions and 
%condition evaluations must be carried out sequentially, making it 
%harder to parallelise. 
\nameref{alg:SBG} avoids this issue by using the fast 
Algorithm~\ref{alg:ARA} over the stick lengths in SB 
representation~\eqref{eq:chi} and calling Algorithm~\ref{alg:ARA_2} 
only over the short time interval $[0,L_n]$, during which very few 
(if any) jumps of $X^{(\kappa_2)}$ occur. Moreover, 
\nameref{alg:SBG} consists of several conditionally 
independent evaluations of Algorithm~\ref{alg:ARA}, which is 
paralellizable, leading to additional numerical benefits 
(see Subsection~\ref{subsec:CP_example} below).

\subsection{Computational complexity of the MC/MLMC estimator based on \nameref{alg:SBG}}
\label{subsec:complexities}

This subsection gives an overview of the bounds on the 
computational complexity of the MC and MLMC estimators defined 
respectively in~\eqref{eq:MC} and~\eqref{eq:MLMC} of 
Subsection~\ref{subsec:MC_MLMC} below. 
Corollary~\ref{cor:SBG_MC} (for MC) and 
Theorem~\ref{thm:SBG_MLMC} (for MLMC) in 
Subsection~\ref{subsec:MC_MLMC} give the full analysis. 

Assume~(\nameref{asm:(O)}) 
holds with some $\delta\in(0,2]$
throughout the subsection. 
As discussed in 
Subsection~\ref{subsec:wasserstein} above, we take $\delta$ as 
large as possible. In particular, if $\sigma\neq0$ then $\delta=2$. 
Let $q\in(0,2]$ be as in~\eqref{eq:BG_bounds} and thus
$q\ge\delta$ if $\sigma=0$. We take $q$ as small as possible. 
For processes used in practice with $\sigma=0$, we may typically 
take $\delta=q=\beta$, where $\beta$ is the BG index defined 
in~\eqref{eq:I0_beta}. 
Assumption~(\nameref{asm:(Htau)}), 
required for the analysis of the class $\BT_2$ in~\eqref{def:BT2} of 
discontinuous functions of $\un \tau_T(X)$, holds with $\gamma=1$ 
as~(\nameref{asm:(O)}) is satisfied (see the discussion following 
Proposition~\ref{prop:barrier-tau} above).
When analysing the class of discontinuous functions $\BT_1$ 
in~\eqref{def:BT1}, assume~(\nameref{asm:(H)}) holds throughout 
with some $\gamma>0$. 

\subsubsection{Monte Carlo.}
An MC estimator is $L^2$-accurate at level $\epsilon>0$, 
if its bias is smaller than $\epsilon/\sqrt{2}$ and the number $N$ of
independent samples is proportional to $\epsilon^{-2}$, 
see Appendix~\ref{app:MonteCarlo}. 
The following table contains a summary of the values $\kappa$, 
as a function of $\epsilon$, such that the bias of the estimator in~\eqref{eq:MC} 
%(based on \nameref{alg:SBG})
is at most $\epsilon/\sqrt{2}$,
and the associated Monte Carlo cost $\C_\MC(\epsilon)$ (up to a constant) 
for various classes of functions of $\un\chi_T$
analysed in Subsection~\ref{subsec:bias}
%to the function of $\un\chi_T$ whose expectation we are estimating 
(see also Corollary~\ref{cor:SBG_MC} below for full details).

\begin{table}[ht]
{\scalefont{.9}
\begin{tabular}{|c|c|c|c|}
	\hline
	% ROW 1
	Family of functions $f$ & Case & $\kappa$ & $\epsilon^2\cdot \C_\MC(\epsilon)$ \\
	\hline
	% ROW 2
	$\Lip$ in $(X_T,\un{X}_T)$
	& $\sigma\ne 0$
	& $\epsilon^{1/(3-q)}|\log\epsilon|^{-1}$ 
	& $\epsilon^{-q/(3-q)}|\log\epsilon|^{q}$\\
	%\hline
	% ROW 3
	$\locLip$ in $(X_T,\un{X}_T)$
	& $\sigma\ne 0$ 
	& $\epsilon^{2/(4-q)}|\log\epsilon|^{-1/2}$ 
	& $\epsilon^{-2q/(4-q)}|\log\epsilon|^{q/2}$\\
	%\hline
	% ROW 4
	$\Lip\cup\locLip$ in $(X_T,\un{X}_T)$
	& $\sigma=0$ & $\epsilon|\log\epsilon|^{-1}$ 
	& $\epsilon^{-q}|\log\epsilon|^q$ \\
	\hline
	% ROW 5
	\multirow{2}{*}{$\BT_1$ defined in~\eqref{def:BT1}} 
	& $\sigma\ne 0$ 
	& $\max\big\{\frac{\epsilon^{3/(4-q)}}{|\log\epsilon|},
			\frac{\epsilon^{2/(3-q)}}{|\log\epsilon|^{1/2}}\big\}$ 
	& $\min\big\{\frac{|\log\epsilon|^{q}}{\epsilon^{3q/(4-q)}},
			\frac{|\log\epsilon|^{q/2}}{\epsilon^{2q/(3-q)}}\big\}$\\
	%\hline
	% ROW 6
	& $\sigma=0$ 
	& $\epsilon^{1/2+1/\gamma}|\log\epsilon|^{-1}$ 
	& $\epsilon^{-q(1/2+1/\gamma)}|\log\epsilon|^q$ \\
	\hline
	% ROW 7
	\multirow{3}{*}{$\Lip$ in $\un\tau_T(X)$} 
	& $\sigma\ne 0$ 
	& $\epsilon^{1/(3-q)}$ 
	& $\epsilon^{-q/(3-q)}$ \\
	%\hline
	% ROW 8
	& $\delta\in(0,2)\setminus\{\tfrac{2}{3}\}$ 
	& $\epsilon^{\min\{2/\delta,\max\{3/2,1/\delta\}\}}$ 
	& $\epsilon^{-q\min\{2/\delta,\max\{3/2,1/\delta\}\}}$\\
	%\hline
	% ROW 9
	& $\delta=\tfrac{2}{3}$ 
	& $\epsilon^{3/2}|\log\epsilon|^{-1}$ 
	& $\epsilon^{-3q/2}|\log\epsilon|^q$\\
	\hline
	% ROW 10
	\multirow{3}{*}{$\BT_2$ defined in~\eqref{def:BT2}} 
	& $\sigma\ne 0$ 
	& $\epsilon^{2/(3-q)}$ 
	& $\epsilon^{-2q/(3-q)}$ \\
	%\hline
	% ROW 11
	& $\delta\in(0,2)\setminus\{\tfrac{2}{3}\}$ 
	& $\epsilon^{\min\{4/\delta,\max\{3,2/\delta\}\}}$ 
	& $\epsilon^{-q\min\{4/\delta,\max\{3,2/\delta\}\}}$\\
	%\hline
	% ROW 12
	& $\delta=\tfrac{2}{3}$ 
	& $\epsilon^{3}|\log\epsilon|^{-1/2}$ 
	& $\epsilon^{-3q}|\log\epsilon|^{q/2}$\\
	\hline
\end{tabular}
}\caption{\footnotesize
Asymptotic behaviour of the level $\kappa$ and the 
complexity $\C_\MC(\epsilon)$ as $\epsilon\to 0$ for the MC 
estimator in~\eqref{eq:MC}.}
\label{tab:MC}
\end{table}

The number of sticks $n\in\N\cup\{0\}$ in \nameref{alg:SBG} 
does not affect the law of $\un\chi_T^{(\kappa)}$. It only impacts 
the MC estimator in~\eqref{eq:MC} through numerical stability and 
the reduction of simulation cost. It is hard to determine the optimal 
choice for $n$. Clearly, the choice $n=0$ 
(i.e. Algorithm~\ref{alg:ARA_2}) is not a good one as discussed in 
Subsection~\ref{subsec:SBG_sampler}  above. A balance needs to 
be struck between (i) having a vanishingly small number of jumps
in the time interval $[0,L_n]$, so that  Algorithm~\ref{alg:ARA_2} 
behaves in a numerically stable way, and (ii) not having too many 
sticks so that line~\ref{alg_line:samplingFrom_A1_A2} of 
\nameref{alg:SBG} does not execute redundant computation 
of many geometrically small increments of $X^{(\kappa)}$, 
which are not detected in the final output. A good rule of thumb is 
$n=n_0 + \cl{\log^2(1+\ov\nu(\kappa)T)}$, where 
$\cl{x}:=\inf\{j\in\Z:j\ge x\}$, $x\in\R$, and 
the initial value 
$n_0$ is chosen so that  some sticks are present if for large 
$\kappa$ the total expected number of jumps $\ov\nu(\kappa)T$ 
is small (e.g. $n_0=5$ works well in 
Subsection~\ref{subsec:CP_example} for jump diffusions with 
low activity, see Figures~\ref{fig:ARA-Speedup} 
and~\ref{fig:ARA-Speedup2}), ensuring that the expected 
number of jumps in $[0,L_n]$ vanishes as $\epsilon$ (and hence 
$\kappa$) tends to zero. 

\subsubsection{Multilevel Monte Carlo.} 
The MLMC estimator in~\eqref{eq:MLMC} is based on the coupling in 
\nameref{alg:SBG} for consecutive levels of a geometrically 
decaying sequence $(\kappa_j)_{j\in\N}$ and an increasing 
sequence of the numbers of sticks $(n_j)_{j\in\N}$. 
Table~\ref{tab:MLMC} summarises the resulting MLMC complexity 
up to logarithmic factors, with full results available in 
Theorem~\ref{thm:SBG_MLMC} below. 

There are two key ingredients in the proof of 
Theorem~\ref{thm:SBG_MLMC}:
(I)  the bounds in 
Theorem~\ref{thm:summary}
on the $L^2$-distance (i.e. the level variance, see Appendix~\ref{subsec:MLMC}) between the functions of the
marginals of the coupling 
$\un \Pi_{n_j,T}^{\kappa_j,\kappa_{j+1}}$
constructed by \nameref{alg:SBG};
(II) the bounds on the bias of various functions in Section~\ref{sec:main-theory} above.
The number of levels $m$ in the MLMC estimator in~\eqref{eq:MLMC}
is chosen to ensure that its bias, equal to the bias of 
$\un \chi_T^{(\kappa_m)}$ at 
the  top cutoff level $\kappa_m$, 
is bounded by $\epsilon/\sqrt{2}$. Thus, the value of $m$ can be expressed in terms of $\epsilon$ using Table~\ref{tab:MC} and the explicit 
formula for the cutoff $\kappa_j$, 
given in the caption of Table~\ref{tab:MLMC}. 
The formula for $\kappa_j$ at level $j$ in the MLMC estimator in~\eqref{eq:MLMC}
%at level $j$ in MLMC estimator~\eqref{eq:MLMC},
is established in the proof of Theorem~\ref{thm:SBG_MLMC} 
by minimising the multiplicative constant in the computational complexity
$\C_\ML(\epsilon)$ over all possible rates of the geometric decay of the sequence $(\kappa_j)_{j\in\N}$.

We stress that the analysis of the level variances for the various 
payoff functions of the coupling 
$\un\Pi_{n_j,T}^{\kappa_j,\kappa_{j+1}}$ 
in Theorem~\ref{thm:summary} is carried out directly for locally 
Lipschitz payoffs, see Propositions~\ref{prop:SBG_app_coupling}. 
However, in the case of the discontinuous payoffs in 
$\BT_1$ (see~\eqref{def:BT1}) and $\BT_2$ (see~\eqref{def:BT2}),
the analysis requires a certain regularity (uniformly in the cutoff 
levels) of the coupling 
$(\un\chi_T^{(\kappa_j)},\un\chi_T^{(\kappa_{j+1})})$.
This leads to  a construction of a further coupling
$(\un\chi_T^{(\kappa_j)},\un\chi_T^{(\kappa_{j+1})}, \un \chi_T)$
where the components of 
$(\un\chi_T^{(\kappa_j)},\un\chi_T^{(\kappa_{j+1})})$ 
can be compared to the limiting object $\un \chi_T$, which can be shown to possess the necessary regularity
(see Proposition~\ref{prop:SBG_app_coupling+} below for details).

%we must first study the level variances of 
%this coupling (see Appendix~\ref{app:MonteCarlo} below). Such 
%description is contained in Theorem~\ref{thm:summary} and is 
%obtained via 

%for the classes of functions $f$ considered 
%in Subsection~\ref{subsec:bias} (see Theorem~\ref{thm:SBG_MLMC} 
%below for details) under the same assumptions on $\gamma,\delta$ 
%and $q$ we made in the Monte Carlo case. 
%In this case, we must let the sequence of cutoff levels $\kappa_j$ 
%decrease to $0$ and the parameter $n_j$ in \nameref{alg:SBG} 
%grow to infinity. In Theorem~\ref{thm:SBG_MLMC}, the parameters 
%$n_j$ are taken equal to 
%$n_0 + \cl{\max\{j,\log^2(1+\ov\nu(\kappa)T)\}}$ 
%(cf. the Monte Carlo case above), the cutoff levels are taken 
%geometrically small and the rate at which they vanish is optimised 
%(see Table~\ref{tab:MLMC} below). 
%

\begin{table}[ht]
{\scalefont{.9}
\begin{tabular}{|c|c|c|c|}
	\hline
	% ROW 1
	Family of functions $f$ & Case & $a$ 
	& The power of $\epsilon^{-1}$ in 
	$\epsilon^2\cdot \C_\ML(\epsilon)$ \\
	\hline
	% ROW 2
	$\Lip$ in $(X_T,\un{X}_T)$
	& $\sigma\ne 0$
	& $2(q-1)$%\multirow{3}{*}{}
	& $2(q-1)^+/(3-q)$\\
	%\hline
	% ROW 3
	$\locLip$ in $(X_T,\un{X}_T)$
	& $\sigma\ne 0$ 
	& $2(q-1)$
	& $4(q-1)^+/(4-q)$\\
	%\hline
	% ROW 4
	$\Lip\cup\locLip$ in $(X_T,\un{X}_T)$
	& $\sigma= 0$ 
	& $2(q-1)$
	& $2(q-1)^+$\\
	\hline
	% ROW 5 
	\multirow{2}{*}{$\BT_1$ defined in~\eqref{def:BT1}} 
	& $\sigma\ne 0$ 
	& $2(2q-1)/3$%\multirow{2}{*}{} 
	& $(2q-1)^+\min\{2/(4-q),4/(9-3q)\}$\\
	%\hline
	% ROW 6
	& $\sigma=0$ 
	& $2(q(1+\gamma)-\gamma)/(2+\gamma)$
	& $(q(1+1/\gamma)-1)^+$\\
	\hline
	% ROW 7
	\multirow{3}{*}{$\Lip$ in $\un\tau_T(X)$} 
	& $\sigma\ne 0$ 
	& $\frac{5}{4}q - \frac{1}{2}$
	& $(\frac{5}{4}q - \frac{1}{2})^+$ \\
	%\hline
	% ROW 8
	& $\sigma=0$ 
	& $q-(1-\frac{q}{2})\min\{\frac{1}{2},\frac{2\delta}{2-\delta}\}$ 
	& $\dfrac{(2q-(2-q)
			\min\{1,4\delta/(2-\delta)\})^+}
		{\max\{\delta,\min\{4/3,2\delta\}\}}$\\
	\hline
	% ROW 9
	\multirow{2}{*}{$\BT_2$ defined in~\eqref{def:BT2}} 
	& $\sigma\ne 0$ 
	& $\frac{9}{8}q-\frac{1}{4}$ 
	& $(\frac{9}{4}q-\frac{1}{2})^+$ \\
	%\hline
	% ROW 10
	& $\sigma=0$ 
	& $q-(1-\frac{q}{2})\min\{\frac{1}{4},\frac{\delta}{2-\delta}\}$ 
	& $\dfrac{(2q-(2-q)
			\min\{1/2,2\delta/(2-\delta)\})^+}
		{\max\{\delta/2,\min\{2/3,\delta\}\}}$\\
	\hline
\end{tabular}
}\caption{\footnotesize
	The table presents the power of 
	$\epsilon^{-1}$ in $\epsilon^2\cdot \C_\ML(\epsilon)$ as 
	$\epsilon\to 0$, neglecting only the logarithmic factors 
	(see Theorem~\ref{thm:SBG_MLMC} below for the complete result). 
	Parameter $a$ in the table determines the decreasing sequence of 
	cutoff levels $(\kappa_j)_{j\in\N}$ as follows: 
	$\kappa_j=(1+|a|/q)^{-2(j-1)/|a|}$ if $a\neq0$ and 
	$\kappa_j=\exp(-(2/q)(j-1))$ otherwise. 
	The corresponding increasing number of sticks $n_j$ in the 
	definition of the law $\un \Pi_{n_j,T}^{\kappa_j,\kappa_{j+1}}$
	can be taken to grow asymptotically as 
	$\log^2(1+\ov\nu(\kappa_j)T)$ for large~$j$, 
	see Theorem~\ref{thm:SBG_MLMC}.}
\label{tab:MLMC}
\end{table}

\vspace{-5mm}

\section{Numerical examples}
\label{sec:numerics}

%The code used to produce the examples in this section is based on~\cite{Jorge_GitHub}.
%The implementation of \nameref{alg:SBG} above can be found in 
%the repository~\cite{Jorge_GitHub3} together with simple algorithms 
%for the simulation of the jumps of tempered stable, truncated stable 
%and Watanabe processes as well as the increments of Kou's and 
%Merton's diffusions. 

In this section we study numerically the performance  
of \nameref{alg:SBG}. All the results are based on the code available in 
repository~\cite{Jorge_GitHub3}.
In Subsection~\ref{subsec:TSW_example} we apply 
\nameref{alg:SBG} to two families of L\'evy models 
(tempered stable and Watanabe processes) 
and 
verify numerically 
the decay of the bias 
(established in Subsection~\ref{subsec:bias} above) and 
level variance (see Theorem~\ref{thm:summary} below) of the Gaussian approximations. 
%in Theorem~\ref{thm:summary} above, describing the biases and 
In Subsection~\ref{subsec:CP_example} we study numerically
the cost reduction of \nameref{alg:SBG}, when compared to 
Algorithm~\ref{alg:ARA_2}, for the simulation of the vector $\un \chi_T^{(\kappa)}$. 

\subsection{Numerical performance of \nameref{alg:SBG} for tempered stable and Watanabe processes} 
\label{subsec:TSW_example}

To illustrate numerically our results, we consider two classes of 
exponential L\'evy models $S=S_0e^{X}$. 
The first is the tempered stable class, containing 
the CGMY (or KoBoL) model, a widely used process for modeling 
risky assets in financial mathematics (see e.g.~\cite{tankov2015} 
and the references therein), 
which satisfies the regularity assumptions from 
Subsection~\ref{subsec:bias} above. 
The second is the Watanabe class, which has diffuse but singular 
transition laws~\cite[Thm~27.19]{MR3185174}, making it a good 
candidate to stress test our results. 

We numerically study the decay of the bias and level variance of 
the MLMC estimator in~\eqref{eq:MLMC} for the prices of a 
lookback put $\E[\ov S_T-S_T]$ and an up-and-out call 
$\E[(S_T-K)^+\1\{\ov S_T\le M\}]$ as well as the values of the 
ulcer index (UI) $100\E[(S_T/\ov S_T -1)^2]^{1/2}$~\cite{
	martin1989investor,Investopedia_UI} and a modified ulcer index 
(MUI) $100\E[(S_T/\ov S_T -1)^2\1\{\ov\tau_T(S)< T/2\}]^{1/2}$. 
The first three quantities are commonplace in applications, 
see~\cite{tankov2015,Investopedia_UI}. The MUI refines the UI 
by incorporating the information on the drawdown duration, 
weights trends more heavily than short-time fluctuations. 

In Subsections~\ref{subsec:TS} and~\ref{subsec:W} we use $N=10^5$ 
independent samples to estimate the means and variances of the 
variables $D^1_j$ in~\eqref{eq:MLMC} (with 
$\un\chi_T^{(\kappa_j)}$ substituted by $\ov\chi_T^{(\kappa_j)}$), 
where $\kappa_j=e^{-r(j-1)}$ and 
$n_j=\cl{\max\{j,\log^2(1+\ov\nu(\kappa_{j+1}))\}}$, $j\in\N$, 
discussed in Subsection~\ref{subsec:MC_MLMC} below.  

\subsubsection{Tempered stable.}
\label{subsec:TS}

The characteristic triplet $(\sigma^2,\nu,b)$ of the tempered stable 
L\'evy process $X$ is given by $\sigma=0$, drift $b\in\R$ and 
L\'evy measure $\nu(dx)=|x|^{-1-\alpha_{\sgn(x)}}c_{\sgn(x)}
	e^{-\lambda_{\sgn(x)}|x|}dx$, where $\alpha_\pm\in[0,2)$, 
$c_\pm\ge0$ and $\lambda_\pm>0$, cf.~\eqref{eq:Levy_Khinchin}. 
Exact simulation of increments is currently out of reach if either
$\alpha_+>1$ or $\alpha_->1$ (see e.g.~\cite{MR3969059}) and 
requires the Gaussian approximation. 

\begin{figure}[ht]
	\begin{center}
		\begin{subfigure}[c]{.49\linewidth}
			{\scalefont{.8}%add .1 and .2
				\begin{tikzpicture} 
				\begin{axis} 
				[
				title={(A) $\ov{S}_T-S_T$},
				ymin=-21,
				ymax=-6,
				xmin=-12,
				xmax=0,
				xlabel={\footnotesize $\log\kappa_j$},
				width=8cm,
				height=3.75cm,
				axis on top=true,
				axis x line=bottom, 
				axis y line=left,
				axis line style={->},
				x label style={at={(axis description cs:0.93,0.5)},anchor=north},
				legend style={at={(.76,.015)},anchor=south east}
				]
				
				% Bias
				\addplot[
				solid, mark=o, mark options={scale=.85, solid},
				color=black,
				]
				coordinates {
(-0.5,-7.153)(-1.0,-7.074)(-1.5,-7.385)(-2.0,-11.73)(-2.5,-7.411)(-3.0,-6.978)(-3.5,-6.779)(-4.0,-6.684)(-4.5,-6.682)(-5.0,-6.921)(-5.5,-7.46)(-6.0,-8.087)(-6.5,-8.646)(-7.0,-9.245)(-7.5,-9.814)(-8.0,-10.45)(-8.5,-11.14)(-9.0,-11.66)(-9.5,-12.55)(-10.0,-12.89)(-10.5,-13.27)(-11.0,-14.04)(-11.5,-14.5)(-12.0,-15.61)
				};
				
				% 'Fit' line with slope 1
				\addplot[
				loosely dashed, %mark=o, mark options={scale=.5, solid},
				color=blue,
				]
				coordinates{
(-0.5,-3.479)(-1.0,-3.927)(-1.5,-4.377)(-2.0,-4.829)(-2.5,-5.284)(-3.0,-5.74)(-3.5,-6.198)(-4.0,-6.658)(-4.5,-7.12)(-5.0,-7.583)(-5.5,-8.047)(-6.0,-8.513)(-6.5,-8.979)(-7.0,-9.447)(-7.5,-9.916)(-8.0,-10.39)(-8.5,-10.86)(-9.0,-11.33)(-9.5,-11.8)(-10.0,-12.27)(-10.5,-12.75)(-11.0,-13.22)(-11.5,-13.7)(-12.0,-14.17)
				};
				
				% Variance
				\addplot[
				solid, mark=+, mark options={scale=.85, solid},
				color=black,
				]
				coordinates {
(-0.5,-8.406)(-1.0,-8.002)(-1.5,-8.144)(-2.0,-8.314)(-2.5,-8.698)(-3.0,-9.224)(-3.5,-9.818)(-4.0,-10.39)(-4.5,-10.93)(-5.0,-11.47)(-5.5,-11.98)(-6.0,-12.53)(-6.5,-13.1)(-7.0,-13.7)(-7.5,-14.34)(-8.0,-14.99)(-8.5,-15.62)(-9.0,-16.3)(-9.5,-16.96)(-10.0,-17.64)(-10.5,-18.29)(-11.0,-18.96)(-11.5,-19.64)(-12.0,-20.29)
				};
				
				% 'Fit' line with slope 2-\alpha
				\addplot[
				loosely dashed, %mark=+, mark options={scale=.5, solid},
				color=red,
				]
				coordinates{
(-0.5,-4.98)(-1.0,-5.65)(-1.5,-6.32)(-2.0,-6.99)(-2.5,-7.66)(-3.0,-8.33)(-3.5,-9.0)(-4.0,-9.67)(-4.5,-10.34)(-5.0,-11.01)(-5.5,-11.68)(-6.0,-12.35)(-6.5,-13.02)(-7.0,-13.69)(-7.5,-14.36)(-8.0,-15.03)(-8.5,-15.7)(-9.0,-16.37)(-9.5,-17.04)(-10.0,-17.71)(-10.5,-18.38)(-11.0,-19.05)(-11.5,-19.72)(-12.0,-20.39)
				};
				
				% Legend 
%				\legend {\footnotesize $\log|\E[D_k]|$,, \footnotesize $\log\V[D_k]$};
				\end{axis}
				\end{tikzpicture}}
		\end{subfigure}
		\begin{subfigure}[c]{.49\linewidth}
			{\scalefont{.8}%add .5 and 0
				\begin{tikzpicture} 
				\begin{axis} 
				[
				title={(B) $(S_T-K)^+\1\{\ov{S}_T\le M\}$},
				ymin=-20,
				ymax=-5,
				xmin=-12,
				xmax=-.5,
				xlabel={\footnotesize $\log\kappa_j$},
				width=8cm,
				height=3.75cm,
				axis on top=true,
				axis x line=bottom, 
				axis y line=left,
				axis line style={->},
				x label style={at={(axis description cs:0.93,0.5)},anchor=north},
				legend style={at={(.76,.015)},anchor=south east}
				]
				
				% Bias
				\addplot[
				solid, mark=o, mark options={scale=.85, solid},
				color=black,
				]
				coordinates {
(-0.5,-10.15)(-1.0,-6.729)(-1.5,-5.308)(-2.0,-5.383)(-2.5,-7.961)(-3.0,-7.889)(-3.5,-8.519)(-4.0,-8.798)(-4.5,-9.022)(-5.0,-11.69)(-5.5,-10.75)(-6.0,-11.62)(-6.5,-11.2)(-7.0,-9.701)(-7.5,-11.44)(-8.0,-11.15)(-8.5,-11.3)(-9.0,-11.17)(-9.5,-14.46)(-10.0,-18.25)(-10.5,-10.61)(-11.0,-11.36)(-11.5,-11.33)(-12.0,-16.28)
				};
				
				% 'Fit' line with slope 1
				\addplot[
				loosely dashed, %mark=o, mark options={scale=.5, solid},
				color=blue,
				]
				coordinates{
(-0.5,-6.167)(-1.0,-6.447)(-1.5,-6.731)(-2.0,-7.016)(-2.5,-7.304)(-3.0,-7.594)(-3.5,-7.886)(-4.0,-8.179)(-4.5,-8.474)(-5.0,-8.77)(-5.5,-9.068)(-6.0,-9.367)(-6.5,-9.667)(-7.0,-9.968)(-7.5,-10.27)(-8.0,-10.57)(-8.5,-10.88)(-9.0,-11.18)(-9.5,-11.49)(-10.0,-11.79)(-10.5,-12.1)(-11.0,-12.41)(-11.5,-12.72)(-12.0,-13.03)
				};
				
				% Variance
				\addplot[
				solid, mark=+, mark options={scale=.85, solid},
				color=black,
				]
				coordinates {
(-0.5,-7.658)(-1.0,-6.663)(-1.5,-5.441)(-2.0,-5.416)(-2.5,-6.52)(-3.0,-6.615)(-3.5,-6.832)(-4.0,-7.155)(-4.5,-7.345)(-5.0,-7.65)(-5.5,-7.397)(-6.0,-7.855)(-6.5,-8.574)(-7.0,-8.749)(-7.5,-10.03)(-8.0,-9.543)(-8.5,-10.05)(-9.0,-10.05)(-9.5,-10.46)(-10.0,-10.46)(-10.5,-9.769)(-11.0,-11.18)(-11.5,-11.17)(-12.0,-19.86)
				};
				
				% 'Fit' line with slope 2-\alpha
				\addplot[
				loosely dashed, %mark=+, mark options={scale=.5, solid},
				color=red,
				]
				coordinates{
(-0.5,-6.08)(-1.0,-6.304)(-1.5,-6.527)(-2.0,-6.75)(-2.5,-6.974)(-3.0,-7.197)(-3.5,-7.42)(-4.0,-7.644)(-4.5,-7.867)(-5.0,-8.09)(-5.5,-8.314)(-6.0,-8.537)(-6.5,-8.76)(-7.0,-8.984)(-7.5,-9.207)(-8.0,-9.43)(-8.5,-9.654)(-9.0,-9.877)(-9.5,-10.1)(-10.0,-10.32)(-10.5,-10.55)(-11.0,-10.77)(-11.5,-10.99)(-12.0,-11.22)
				};
				
				% Legend 
				%				\legend {\footnotesize $\log|\E[D_k]|$,, \footnotesize $\log\V[D_k]$};
				\end{axis}
				\end{tikzpicture}}
		\end{subfigure}
		\begin{subfigure}[c]{.49\linewidth}
			{\scalefont{.8}
				\begin{tikzpicture} 
				\begin{axis} 
				[
				title={(C) $(S_T/\ov{S}_T-1)^2$},
				ymin=-20,
				ymax=-7,
				xmin=-12,
				xmax=0,
				xlabel={\footnotesize $\log\kappa_j$},
				width=8cm,
				height=3.75cm,
				axis on top=true,
				axis x line=bottom, 
				axis y line=left,
				axis line style={->},
				x label style={at={(axis description cs:0.7,0.5)},anchor=north},
				legend style={at={(1,.15)},anchor=south east}
				]
				
				% Bias
				\addplot[
				solid, mark=o, mark options={scale=.85, solid},
				color=black,
				]
				coordinates {
(-0.5,-12.47)(-1.0,-14.83)(-1.5,-14.78)(-2.0,-13.9)(-2.5,-11.62)(-3.0,-9.925)(-3.5,-8.949)(-4.0,-8.367)(-4.5,-8.071)(-5.0,-8.011)(-5.5,-8.178)(-6.0,-8.466)(-6.5,-8.794)(-7.0,-9.182)(-7.5,-9.587)(-8.0,-10.07)(-8.5,-10.51)(-9.0,-10.95)(-9.5,-11.39)(-10.0,-11.87)(-10.5,-12.26)(-11.0,-12.99)(-11.5,-13.35)(-12.0,-13.9)
				};
				
				% 'Fit' line with slope 1
				\addplot[
				loosely dashed, %mark=o, mark options={scale=.5, solid},
				color=blue,
				]
				coordinates{
(-0.5,-3.119)(-1.0,-3.567)(-1.5,-4.017)(-2.0,-4.469)(-2.5,-4.924)(-3.0,-5.38)(-3.5,-5.838)(-4.0,-6.298)(-4.5,-6.76)(-5.0,-7.223)(-5.5,-7.687)(-6.0,-8.153)(-6.5,-8.619)(-7.0,-9.087)(-7.5,-9.556)(-8.0,-10.03)(-8.5,-10.5)(-9.0,-10.97)(-9.5,-11.44)(-10.0,-11.91)(-10.5,-12.39)(-11.0,-12.86)(-11.5,-13.34)(-12.0,-13.81)
				};
				
				% Variance
				\addplot[
				solid, mark=+, mark options={scale=.85, solid},
				color=black,
				]
				coordinates {
(-0.5,-15.59)(-1.0,-16.56)(-1.5,-17.55)(-2.0,-17.79)(-2.5,-16.34)(-3.0,-15.43)(-3.5,-14.91)(-4.0,-14.25)(-4.5,-13.79)(-5.0,-13.57)(-5.5,-13.59)(-6.0,-13.78)(-6.5,-14.07)(-7.0,-14.45)(-7.5,-14.85)(-8.0,-15.25)(-8.5,-15.69)(-9.0,-16.14)(-9.5,-16.56)(-10.0,-17.02)(-10.5,-17.49)(-11.0,-17.96)(-11.5,-18.38)(-12.0,-18.87)
				};
				
				% 'Fit' line with slope 2-\alpha
				\addplot[
				loosely dashed, %mark=+, mark options={scale=.5, solid},
				color=red,
				]
				coordinates{
(-0.5,-8.187)(-1.0,-8.648)(-1.5,-9.109)(-2.0,-9.57)(-2.5,-10.03)(-3.0,-10.49)(-3.5,-10.95)(-4.0,-11.41)(-4.5,-11.88)(-5.0,-12.34)(-5.5,-12.8)(-6.0,-13.26)(-6.5,-13.72)(-7.0,-14.18)(-7.5,-14.64)(-8.0,-15.1)(-8.5,-15.56)(-9.0,-16.02)(-9.5,-16.48)(-10.0,-16.95)(-10.5,-17.41)(-11.0,-17.87)(-11.5,-18.33)(-12.0,-18.79)
				};
				
				% Legend 
%				\legend {\footnotesize $\log|\E[D_k]|$,, \footnotesize $\log\V[D_k]$};
				\end{axis}
				\end{tikzpicture}}
		\end{subfigure}
		\begin{subfigure}[c]{.49\linewidth}
			{\scalefont{.8}
				\begin{tikzpicture}
				\begin{axis} 
				[
				title={(D) $(S_T/\ov{S}_T-1)^2\1\{\ov\tau_T(S)<T/2\}$},
				ymin=-20,
				ymax=-7,
				xmin=-12,
				xmax=0,
				xlabel={\footnotesize $\log\kappa_j$},
				width=8cm,
				height=3.75cm,
				axis on top=true,
				axis x line=bottom, 
				axis y line=left,
				axis line style={->},
				x label style={at={(axis description cs:0.7,0.5)},anchor=north},
				legend style={at={(1,.15)},anchor=south east}
				]
				
				% Bias
				\addplot[
				solid, mark=o, mark options={scale=.85, solid},
				color=black,
				]
				coordinates {
(-0.5,-14.89)(-1.0,-13.63)(-1.5,-15.16)(-2.0,-14.76)(-2.5,-12.33)(-3.0,-10.45)(-3.5,-9.416)(-4.0,-8.723)(-4.5,-8.327)(-5.0,-8.248)(-5.5,-8.432)(-6.0,-8.741)(-6.5,-9.052)(-7.0,-9.477)(-7.5,-9.937)(-8.0,-10.38)(-8.5,-10.74)(-9.0,-11.52)(-9.5,-11.63)(-10.0,-12.14)(-10.5,-12.67)(-11.0,-13.67)(-11.5,-16.3)(-12.0,-14.33)
				};
			
				% 'Fit' line with slope max{min{\alpha,1/2},\alpha/2}
				\addplot[
				loosely dashed, %mark=o, mark options={scale=.5, solid},
				color=blue,
				]
				coordinates{
(-0.5,-9.466)(-1.0,-9.538)(-1.5,-9.613)(-2.0,-9.69)(-2.5,-9.77)(-3.0,-9.851)(-3.5,-9.935)(-4.0,-10.02)(-4.5,-10.11)(-5.0,-10.19)(-5.5,-10.28)(-6.0,-10.37)(-6.5,-10.47)(-7.0,-10.56)(-7.5,-10.65)(-8.0,-10.75)(-8.5,-10.84)(-9.0,-10.94)(-9.5,-11.04)(-10.0,-11.13)(-10.5,-11.23)(-11.0,-11.33)(-11.5,-11.43)(-12.0,-11.53)
				};
				
				% Variance
				\addplot[
				solid, mark=+, mark options={scale=.85, solid},
				color=black,
				]
				coordinates {
(-0.5,-16.75)(-1.0,-17.71)(-1.5,-18.86)(-2.0,-18.86)(-2.5,-17.1)(-3.0,-15.88)(-3.5,-15.13)(-4.0,-14.24)(-4.5,-13.67)(-5.0,-13.4)(-5.5,-13.38)(-6.0,-13.5)(-6.5,-13.74)(-7.0,-14.03)(-7.5,-14.32)(-8.0,-14.7)(-8.5,-14.85)(-9.0,-15.28)(-9.5,-15.58)(-10.0,-15.98)(-10.5,-16.16)(-11.0,-16.49)(-11.5,-16.65)(-12.0,-17.13)
				};
				
				% 'Fit' line with slope min{\alpha,(2-\alpha)/4}
				\addplot[
				loosely dashed, %mark=+, mark options={scale=.5, solid},
				color=red,
				]
				coordinates{
(-0.5,-14.31)(-1.0,-14.37)(-1.5,-14.43)(-2.0,-14.48)(-2.5,-14.54)(-3.0,-14.6)(-3.5,-14.66)(-4.0,-14.71)(-4.5,-14.77)(-5.0,-14.83)(-5.5,-14.89)(-6.0,-14.94)(-6.5,-15.0)(-7.0,-15.06)(-7.5,-15.12)(-8.0,-15.18)(-8.5,-15.23)(-9.0,-15.29)(-9.5,-15.35)(-10.0,-15.41)(-10.5,-15.46)(-11.0,-15.52)(-11.5,-15.58)(-12.0,-15.64)
				};
			
%				% 'Fit' line with slope min{2\alpha,(2-\alpha)/2}
%				\addplot[
%				loosely dashed, mark=+, mark options={scale=.5, solid},
%				color=black,
%				]
%				coordinates{
%					(-0.5,-17.845)(-12.5,-22.645)
%				};
				
				% Legend 
%				\legend {\footnotesize $\log|\E[D_k]|$,,\footnotesize $\log\V[D_k]$};
				\end{axis}
				\end{tikzpicture}}
		\end{subfigure}
	\caption{\footnotesize 
Gaussian approximation of a tempered stable process: 
log-log plot of the bias and level variance for various payoffs 
as a function of $\log\kappa_j$. 
Circle ($\mathbf{\circ}$) and plus ($\mathbf{+}$) correspond 
to $\log|\E[D_j^1]|$ and $\log\V[D_j^1]$, respectively, 
where $D_j^1$ is given in~\eqref{eq:MLMC} with 
$\kappa_j=\exp(-r(j-1))$ for $r=1/2$. 
The dashed lines in all the graphs plot the rates of the theoretical 
bounds in Subsection~\ref{subsec:bias} (blue for the bias) 
and Theorem~\ref{thm:summary} (red for level variances). 
In plots (A)--(D) the initial value of the risky asset is normalised to 
$S_0=1$ and the time horizon is set to $T=1/6$. In plot (B) we set 
$K=1$ and $M=1.2$. The model parameters are given in 
Table~\ref{tab:TS} below. 
	}\label{fig:TS}
	\end{center}
\end{figure}

\begin{table}[ht]
{\scalefont{.9}
	\begin{tabular}{|c|c|c|c|c|c|c|c|c|}
		\hline
		% ROW 1
		Parameter set
		& $b$
		& $\alpha_+$ & $\alpha_-$ 
		& $c_+$ & $c_-$ 
		& $\lambda_+$ & $\lambda_-$ 
		& Graphs in Figure~\ref{fig:TS}\\
		\hline
		% ROW 2
		1
		& 0
		& .66 & .66 
		& .1305 & .0615
		& 6.5022 & 3.0888
		& (A) and (B)\\
		\hline
		% ROW 3
		2
		& .1274
		& 1.0781 & 1.0781
		& .41077 & .41077
		& 49.663 & 59.078
		& (C) and (D)\\
		\hline
	\end{tabular}
}\caption{The parameters used for  
	Figure~\ref{fig:TS}. The first set of parameters corresponds 
	to the risk-neutral calibration to vanilla options on the USD/JPY 
	exchange rate, see~\cite[Table~3]{MR3038608}. 
	The second set is the maximum likelihood estimate based on 
	the real-world S\&P stock prices, see~\cite[Table~1]{NewTS}.
	}\label{tab:TS}
\end{table}

Figure~\ref{fig:TS} suggests that our bounds are 
close to the exhibited numerical behaviour for continuous payoff 
functions. In the discontinuous case, $\ov\chi_T^{(\kappa_j)}$ 
appears to be much closer to $\ov\chi_T$ 
(resp. $\ov\chi_T^{(\kappa_{j+1})}$), than predicted by 
Propositions~\ref{prop:barrier} \&~\ref{prop:barrier-tau} 
(resp. Theorem~\ref{thm:summary}(b) \&~(d)).  
%, perhaps with the same rate achieved in the case $\sigma\neq0$. 
%This comes in contrast with the bound in Theorem~\ref{thm:Wd-triplet}, 
%which is a multiple of $\kappa^{\max\{\delta/2,\min\{\-\delta,1/2\}\}}$. 
%Such claims are currently beyond what we can prove with the tools 
%developed here. 

\subsubsection{Watanabe model.}
\label{subsec:W}

The characteristic triplet $(\sigma,\nu,b)$ of the Watanabe 
process is given by $\sigma=0$, the L\'evy measure $\nu$ equals 
$\sum_{n\in\N}c_+\delta_{a^{-n}}+c_-\delta_{-a^{-n}}$, 
where $a\in\N\setminus\{1\}$ and $\delta_x$ is the Dirac measure 
at $x$, and the drift $b\in\R$ is arbitrary. 
The increments of the Watanabe process are diffuse but have no 
density (see~\cite[Thm~27.19]{MR3185174}). Since the process has 
very little jump activity, the bound in 
Proposition~\ref{prop:barrier-tau} (see also~\eqref{eq:mu_0}) is 
non-vanishing and the bounds in Theorem~\ref{thm:summary}(c) 
\&~(d) are not applicable, meaning that we have no theoretical 
control on the approximation of $\ov\tau_T(S)$. This is not 
surprising as such acute lack of jump activity makes the Gaussian 
approximation unsuitable (cf.~\cite[Prop.~2.2]{MR1834755}). 

\begin{figure}[ht]
	\begin{center}
		\begin{subfigure}[c]{.49\linewidth}
			{\scalefont{.8}
				\begin{tikzpicture} 
				\begin{axis} 
				[
				title={(A) $\ov{S}_T-S_T$ with drift $b=0$},
				ymin=-60,
				ymax=0,
				xmin=-30,
				xmax=0,
				xlabel={\footnotesize $\log\kappa_j$},
				width=8cm,
				height=3.75cm,
				axis on top=true,
				axis x line=bottom, 
				axis y line=left,
				axis line style={->},
				x label style={at={(axis description cs:0.93,0.45)},anchor=north},
				legend style={at={(1,.15)},anchor=south east}
				]
				
				% Bias
				\addplot[
				solid, mark=o, mark options={scale=.85, solid},
				color=black,
				]
				coordinates {
(-1.0,-1.615)(-2.0,-2.445)(-3.0,-3.262)(-4.0,-5.628)(-5.0,-6.041)(-6.0,-8.14)(-7.0,-8.665)(-8.0,-10.56)(-9.0,-11.34)(-10.0,-11.8)(-11.0,-13.61)(-12.0,-14.02)(-13.0,-15.66)(-14.0,-15.97)(-15.0,-17.96)(-16.0,-18.38)(-17.0,-20.81)(-18.0,-20.91)(-19.0,-21.39)(-20.0,-22.97)(-21.0,-23.55)(-22.0,-25.1)(-23.0,-25.72)(-24.0,-27.4)(-25.0,-27.76)(-26.0,-30.09)(-27.0,-30.15)(-28.0,-31.01)(-29.0,-32.52)(-30.0,-32.92)
				};
				
				% 'Fit' line with slope 1
				\addplot[
				loosely dashed, %mark=o, mark options={scale=.5, solid},
				color=blue,
				]
				coordinates{
(1.0,-3.04)(-30.0,-32.04)
				};
				
				% Variance
				\addplot[
				solid, mark=+, mark options={scale=.85, solid},
				color=black,
				]
				coordinates {
(-1.0,0.7442)(-2.0,-1.122)(-3.0,-2.69)(-4.0,-5.611)(-5.0,-6.439)(-6.0,-9.559)(-7.0,-11.07)(-8.0,-14.16)(-9.0,-15.69)(-10.0,-16.84)(-11.0,-19.92)(-12.0,-21.01)(-13.0,-24.1)(-14.0,-25.18)(-15.0,-28.31)(-16.0,-29.39)(-17.0,-32.42)(-18.0,-33.85)(-19.0,-34.92)(-20.0,-38.0)(-21.0,-39.08)(-22.0,-42.13)(-23.0,-43.26)(-24.0,-46.31)(-25.0,-47.42)(-26.0,-50.51)(-27.0,-51.88)(-28.0,-52.98)(-29.0,-56.05)(-30.0,-57.13)
				};
				
				% 'Fit' line with slope 2-\alpha
				\addplot[
				loosely dashed, %mark=+, mark options={scale=.5, solid},
				color=red,
				]
				coordinates{
(-1.0,1.0)(-30.0,-57.0)
				};
				
				% Legend 
%				\legend {\footnotesize $\log|\E[D_j]|$,, \footnotesize $\log\V[D_j]$};
				\end{axis}
				\end{tikzpicture}}
		\end{subfigure}
		\begin{subfigure}[c]{.49\linewidth}
			{\scalefont{.8}
				\begin{tikzpicture}
				\begin{axis} 
				[
				title={(B) $(S_T/\ov{S}_T-1)^2\1\{\ov\tau_T(S)< T/2\}$ 
					with drift $b=0$},
				ymin=-17,
				ymax=0,
				xmin=-30,
				xmax=0,
				xlabel={\footnotesize $\log\kappa_j$},
				width=8cm,
				height=3.75cm,
				axis on top=true,
				axis x line=bottom, 
				axis y line=left,
				axis line style={->},
				x label style={at={(axis description cs:0.93,0.45)},anchor=north},
				legend style={at={(.005,1)},anchor=north west}
				]
				
				% Bias
				\addplot[
				solid, mark=o, mark options={scale=.85, solid},
				color=black,
				]
				coordinates {
(-1.0,-5.232)(-2.0,-4.942)(-3.0,-6.075)(-4.0,-8.777)(-5.0,-8.525)(-6.0,-9.883)(-7.0,-10.07)(-8.0,-10.65)(-9.0,-12.54)(-10.0,-9.981)(-11.0,-11.73)(-12.0,-10.33)(-13.0,-11.17)(-14.0,-11.21)(-15.0,-12.8)(-16.0,-10.14)(-17.0,-10.16)(-18.0,-11.33)(-19.0,-10.11)(-20.0,-11.62)(-21.0,-11.12)(-22.0,-12.11)(-23.0,-11.06)(-24.0,-11.39)(-25.0,-11.12)(-26.0,-11.39)(-27.0,-11.07)(-28.0,-10.9)(-29.0,-15.25)(-30.0,-13.85)
				};
				
				% Variance
				\addplot[
				solid, mark=+, mark options={scale=.85, solid},
				color=black,
				]
				coordinates {
(-1.0,-3.72)(-2.0,-4.724)(-3.0,-5.335)(-4.0,-6.694)(-5.0,-6.909)(-6.0,-8.035)(-7.0,-7.951)(-8.0,-9.078)(-9.0,-9.288)(-10.0,-8.885)(-11.0,-9.873)(-12.0,-9.193)(-13.0,-10.21)(-14.0,-9.575)(-15.0,-10.58)(-16.0,-9.926)(-17.0,-10.38)(-18.0,-11.36)(-19.0,-10.21)(-20.0,-10.92)(-21.0,-10.33)(-22.0,-12.71)(-23.0,-11.21)(-24.0,-10.77)(-25.0,-10.94)(-26.0,-11.19)(-27.0,-10.76)(-28.0,-11.27)(-29.0,-15.56)(-30.0,-11.28)
				};
				
				% Legend 
%				\legend {\footnotesize $\log|\E[D_j]|$,\footnotesize $\log\V[D_j]$};
				\end{axis}
				\end{tikzpicture}}
		\end{subfigure}
		\begin{subfigure}[c]{.49\linewidth}
			{\scalefont{.8}
				\begin{tikzpicture} 
				\begin{axis} 
				[
				title={(C) $\ov{S}_T-S_T$ with drift $b=-.5$},
				ymin=-60,
				ymax=0,
				xmin=-30,
				xmax=0,
				xlabel={\footnotesize $\log\kappa_j$},
				width=8cm,
				height=3.75cm,
				axis on top=true,
				axis x line=bottom, 
				axis y line=left,
				axis line style={->},
				x label style={at={(axis description cs:0.93,0.45)},anchor=north},
				legend style={at={(1,.15)},anchor=south east}
				]
				
				% Bias
				\addplot[
				solid, mark=o, mark options={scale=.85, solid},
				color=black,
				]
				coordinates {
(-1.0,-2.374)(-2.0,-2.608)(-3.0,-3.63)(-4.0,-6.944)(-5.0,-7.791)(-6.0,-10.27)(-7.0,-11.57)(-8.0,-11.67)(-9.0,-18.22)(-10.0,-14.37)(-11.0,-16.72)(-12.0,-16.48)(-13.0,-20.07)(-14.0,-19.0)(-15.0,-19.91)(-16.0,-22.47)(-17.0,-24.6)(-18.0,-22.19)(-19.0,-24.73)(-20.0,-25.4)(-21.0,-26.17)(-22.0,-27.09)(-23.0,-29.03)(-24.0,-28.94)(-25.0,-29.2)(-26.0,-30.39)(-27.0,-31.75)(-28.0,-32.06)(-29.0,-33.46)(-30.0,-36.43)
				};
				
				% 'Fit' line with slope 1
				\addplot[
				loosely dashed, %mark=o, mark options={scale=.5, solid},
				color=blue,
				]
				coordinates{
(1.0,-5.238)(-30.0,-34.24)
				};
				
				% Variance
				\addplot[
				solid, mark=+, mark options={scale=.85, solid},
				color=black,
				]
				coordinates {
(-1.0,0.1569)(-2.0,-1.681)(-3.0,-3.14)(-4.0,-6.174)(-5.0,-7.234)(-6.0,-8.244)(-7.0,-11.04)(-8.0,-12.32)(-9.0,-15.77)(-10.0,-17.16)(-11.0,-20.3)(-12.0,-21.41)(-13.0,-24.46)(-14.0,-25.56)(-15.0,-28.68)(-16.0,-29.71)(-17.0,-32.82)(-18.0,-34.22)(-19.0,-35.29)(-20.0,-38.35)(-21.0,-39.45)(-22.0,-42.54)(-23.0,-43.63)(-24.0,-46.72)(-25.0,-47.79)(-26.0,-50.85)(-27.0,-52.24)(-28.0,-53.33)(-29.0,-56.4)(-30.0,-57.49)
				};
				
				% 'Fit' line with slope 2-\alpha
				\addplot[
				loosely dashed, %mark=+, mark options={scale=.5, solid},
				color=red,
				]
				coordinates{
(-1.0,1.0)(-30.0,-57.0)
				};
				
				% Legend 
%				\legend {\footnotesize $\log|\E[D_j]|$,, \footnotesize $\log\V[D_j]$};
				\end{axis}
				\end{tikzpicture}}
		\end{subfigure}
		\begin{subfigure}[c]{.49\linewidth}
			{\scalefont{.8}
				\begin{tikzpicture}
				\begin{axis} 
				[
				title={(D) $(S_T/\ov{S}_T-1)^2\1\{\ov\tau_T(S)< T/2\}$ 
				with drift $b=-.5$},
				ymin=-65,
				ymax=0,
				xmin=-30,
				xmax=0,
				xlabel={\footnotesize $\log\kappa_j$},
				width=8cm,
				height=3.75cm,
				axis on top=true,
				axis x line=bottom, 
				axis y line=left,
				axis line style={->},
				x label style={at={(axis description cs:0.93,0.45)},anchor=north},
				legend style={at={(1,.15)},anchor=south east}
				]
				
				% Bias
				\addplot[
				solid, mark=o, mark options={scale=.85, solid},
				color=black,
				]
				coordinates {
(-1.0,-4.918)(-2.0,-4.975)(-3.0,-5.414)(-4.0,-10.5)(-5.0,-8.856)(-6.0,-11.26)(-7.0,-10.49)(-8.0,-12.89)(-9.0,-12.31)(-10.0,-11.71)(-11.0,-14.29)(-12.0,-15.07)(-13.0,-21.69)(-14.0,-21.07)(-15.0,-24.02)(-16.0,-22.83)(-17.0,-23.64)(-18.0,-27.66)(-19.0,-25.5)(-20.0,-26.99)(-21.0,-27.21)(-22.0,-30.42)(-23.0,-30.81)(-24.0,-31.08)(-25.0,-31.84)(-26.0,-34.67)(-27.0,-36.01)(-28.0,-35.55)(-29.0,-36.33)(-30.0,-39.29)
				};
				
				% Variance
				\addplot[
				solid, mark=+, mark options={scale=.85, solid},
				color=black,
				]
				coordinates {
(-1.0,-3.727)(-2.0,-4.714)(-3.0,-5.497)(-4.0,-6.993)(-5.0,-7.623)(-6.0,-9.197)(-7.0,-9.688)(-8.0,-12.44)(-9.0,-13.57)(-10.0,-12.56)(-11.0,-16.4)(-12.0,-18.66)(-13.0,-28.78)(-14.0,-29.92)(-15.0,-33.0)(-16.0,-34.06)(-17.0,-37.16)(-18.0,-38.52)(-19.0,-39.65)(-20.0,-42.66)(-21.0,-43.81)(-22.0,-46.87)(-23.0,-47.98)(-24.0,-51.02)(-25.0,-52.15)(-26.0,-55.18)(-27.0,-56.57)(-28.0,-57.68)(-29.0,-60.72)(-30.0,-61.87)
				};
				
				% Legend 
%				\legend {\footnotesize $\log|\E[D_j]|$,\footnotesize $\log\V[D_j]$};
				\end{axis}
				\end{tikzpicture}}
		\end{subfigure}
		\caption{\footnotesize 
Gaussian approximation of a Watanabe process: 
log-log plot of the bias and level variance for various payoffs as a 
function of $\log\kappa_j$. 
Circle ($\mathbf{\circ}$) and plus ($\mathbf{+}$) correspond to 
$\log|\E[D_j^1]|$ and $\log\V[D_j^1]$, respectively, where $D_j^1$ 
is given in~\eqref{eq:MLMC} with $\kappa_j=\exp(-r(j-1))$ for $r=1$. 
The dashed lines in graphs (A) \&~(C) plot the rates of the theoretical 
bounds in Subsection~\ref{subsec:bias} (blue for the bias) 
and Theorem~\ref{thm:summary} (red for level variances). 
In plots (A)--(D) the initial value of the risky asset is normalised to 
$S_0=1$ and the time horizon is set to $T=1$. 
The model parameters are given by $a=2$, $c_+=c_-=1$. 
}\label{fig:Watanabe}
	\end{center}
\end{figure}

The pictures in Figure~\ref{fig:Watanabe} (A) \&~(C) suggest 
that our bounds on the bias and level variance in 
Subsection~\ref{subsec:bias} and Theorem~\ref{thm:summary} are 
robust for continuous payoff functions even if the underlying L\'evy 
process has no transition densities. There are no dashed lines in 
Figure~\ref{fig:Watanabe} (B) \&~(D) as there are no 
results for discontinuous functions of $\ov\tau_T(S)$ in this case. 
In fact, Figure~\ref{fig:Watanabe}(B) suggests that 
the decay rate of the bias and level variance for functions of 
$\ov\tau_T(S)$ can be arbitrarily slow if the process does not have 
sufficient activity. Figure~\ref{fig:Watanabe}(D), however, 
suggests that this decay is still fast if the underlying finite 
variation process $X$ has a nonzero natural drift (see also 
Remark~\ref{rem:natural-drift}). 

\subsection{The cost reduction of \nameref{alg:SBG} over Algorithm~\ref{alg:ARA_2}} 
\label{subsec:CP_example}

Recall that Algorithm~\ref{alg:ARA_2} and~\nameref{alg:SBG} 
both draw exact samples of a Gaussian approximation 
$\un\chi_T^{(\kappa)}$. However, in practice, 
\nameref{alg:SBG} may be many times faster than 
Algorithm~\ref{alg:ARA_2}: Figure~\ref{fig:ARA-Speedup2} 
plots the speedup factor in the case of a tempered stable 
process, defined in Subsection~\ref{subsec:TS} above, 
as a function of $\kappa$. In conclusion, one should use 
\nameref{alg:SBG} instead of Algorithm~\ref{alg:ARA_2} 
for the MC estimator in~\eqref{eq:MC} (recall that 
Algorithm~\ref{alg:ARA_2} is not suitable for the MLMC 
estimator, as discussed in Subsection~\ref{subsubsec:Error_term}). 

\begin{figure}[ht]
	\begin{center}
		\hspace{-3cm}
		\begin{subfigure}[l]{.49\linewidth}
			{\scalefont{.8}
				\begin{tikzpicture} 
				\begin{axis} 
				[
				title={\normalsize $\alpha_\pm=1.2$},
				ymin=0,
				ymax=65,
				ytick={10, 20, 30, 40, 50, 60},
				xmax=.1,
				xmode=log,
				xlabel={\small $\kappa$},
				width=8.5cm,
				height=4cm,
				axis on top=true,
				axis x line=middle, 
				axis y line=left,
				axis line style={->},
				x label style={at={(axis description cs:1,0.20)},anchor=north},
				legend style={at={(1.165,-.22)},anchor=north, legend columns=-1}
				]
				
				% 5 stick
				\addplot[
				loosely dashed, 
				color=black,
				]
				coordinates {
					(1.0,0.3518)(0.5,0.4424)(0.25,1.147)(0.125,0.1574)(0.0625,5.715)(0.03125,5.806)(0.01562,8.751)(0.007812,12.52)(0.003906,16.24)(0.001953,9.455)(0.0009766,15.0)(0.0004883,19.96)(0.0002441,12.23)(0.0001221,14.41)(6.104e-5,18.9)(3.052e-5,22.83)(1.526e-5,19.17)
				};
				
				% 10 stick
				\addplot[
				dashdotted, 
				color=black,
				]
				coordinates {
					(1.0,0.2264)(0.5,0.2421)(0.25,1.172)(0.125,2.292)(0.0625,4.771)(0.03125,6.277)(0.01562,8.751)(0.007812,20.08)(0.003906,32.01)(0.001953,29.22)(0.0009766,25.23)(0.0004883,34.54)(0.0002441,37.65)(0.0001221,33.98)(6.104e-5,42.25)(3.052e-5,52.45)(1.526e-5,50.3)
				};
				
				% 15 stick
				\addplot[
				densely dotted, 
				color=black,
				]
				coordinates {
					(1.0,0.1273)(0.5,0.205)(0.25,0.8192)(0.125,1.621)(0.0625,3.286)(0.03125,5.341)(0.01562,7.802)(0.007812,22.74)(0.003906,30.44)(0.001953,28.05)(0.0009766,40.4)(0.0004883,38.67)(0.0002441,36.29)(0.0001221,43.8)(6.104e-5,45.07)(3.052e-5,48.13)(1.526e-5,58.41)
				};
				
				% 20 stick
				\addplot[
				solid, 
				color=black,
				]
				coordinates {
					(1.0,0.1135)(0.5,0.1624)(0.25,0.4214)(0.125,1.453)(0.0625,2.931)(0.03125,5.39)(0.01562,7.213)(0.007812,21.5)(0.003906,31.75)(0.001953,28.28)(0.0009766,37.54)(0.0004883,43.67)(0.0002441,54.16)(0.0001221,49.05)(6.104e-5,56.8)(3.052e-5,51.64)(1.526e-5,57.76)
				};
				
				% Legend 
				\legend {\small$n=5$,
					\small$n=10$,
					\small$n=15$,
					\small$n=20$};
				\end{axis}
				\end{tikzpicture}}
		\end{subfigure}
		\hspace{-4cm}
		\begin{subfigure}[l]{.49\linewidth}
			{\scalefont{.8}
				\begin{tikzpicture} 
				\begin{axis} 
				[
				title={\normalsize $\alpha_\pm=1.4$},
				ymin=0,
				ymax=65,
				ytick={10, 20, 30, 40, 50, 60},
				xmax=.1,
				xmode = log,
				xlabel={\small $\kappa$},
				width=8.5cm,
				height=4cm,
				axis on top=true,
				axis x line=middle, 
				axis y line=left,
				axis line style={->},
				x label style={at={(axis description cs:1,0.20)},anchor=north},
				legend style={at={(-.05,-.22)},anchor=north, legend columns=-1}
				]
				
				% 5 stick
				\addplot[
				loosely dashed, 
				color=black,
				]
				coordinates {
					(1.0,0.3291)(0.5,1.063)(0.25,0.5754)(0.125,1.565)(0.0625,4.112)(0.03125,6.976)(0.01562,7.791)(0.007812,22.57)(0.003906,17.6)(0.001953,20.64)(0.0009766,14.24)(0.0004883,14.56)(0.0002441,14.93)(0.0001221,19.35)(6.104e-5,21.96)
				};
				
				% 10 stick
				\addplot[
				dashdotted, 
				color=black,
				]
				coordinates {
					(1.0,0.1229)(0.5,0.6662)(0.25,0.3726)(0.125,1.548)(0.0625,3.71)(0.03125,9.98)(0.01562,16.21)(0.007812,47.83)(0.003906,39.86)(0.001953,59.42)(0.0009766,44.99)(0.0004883,50.36)(0.0002441,49.72)(0.0001221,60.94)(6.104e-5,56.26)
				};
				
				% 15 stick
				\addplot[
				densely dotted, 
				color=black,
				]
				coordinates {
					(1.0,0.09906)(0.5,0.4035)(0.25,0.2312)(0.125,1.128)(0.0625,2.279)(0.03125,9.097)(0.01562,14.13)(0.007812,43.72)(0.003906,40.76)(0.001953,61.46)(0.0009766,49.04)(0.0004883,54.92)(0.0002441,54.82)(0.0001221,62.18)(6.104e-5,59.75)
				};
				
				% 20 stick
				\addplot[
				solid, 
				color=black,
				]
				coordinates {
					(1.0,0.112)(0.5,0.3537)(0.25,0.2976)(0.125,0.5918)(0.0625,2.635)(0.03125,7.715)(0.01562,13.69)(0.007812,41.3)(0.003906,36.59)(0.001953,57.35)(0.0009766,43.22)(0.0004883,49.1)(0.0002441,54.15)(0.0001221,55.36)(6.104e-5,61.7)
				};
				
				% Legend 
				\legend {\small$n=5\quad$,
					\small$n=10\quad$,
					\small$n=15\quad$,
					\small$n=20$};
				\end{axis}
				\end{tikzpicture}}
		\end{subfigure}
\caption{\footnotesize The pictures show the ratio of the cost of 
Algorithm~\ref{alg:ARA_2} over the cost of \nameref{alg:SBG} 
(both in seconds) for the Gaussian approximations of a tempered 
stable process as a function of the cutoff level $\kappa$. 
The parameters used are $\lambda_\pm=5$, $c_\pm=2$. 
The number of sticks $n$ in \nameref{alg:SBG} varies 
between 5 and 20. The ratio for $n=20$ is 57.8 (resp. 61.7) in the 
case $\alpha_\pm=1.2$ (resp. $\alpha_\pm=1.4$) 
for $\kappa=2^{-16}$ (resp. $\kappa=2^{-14}$).
		}\label{fig:ARA-Speedup2}
	\end{center}
\end{figure}
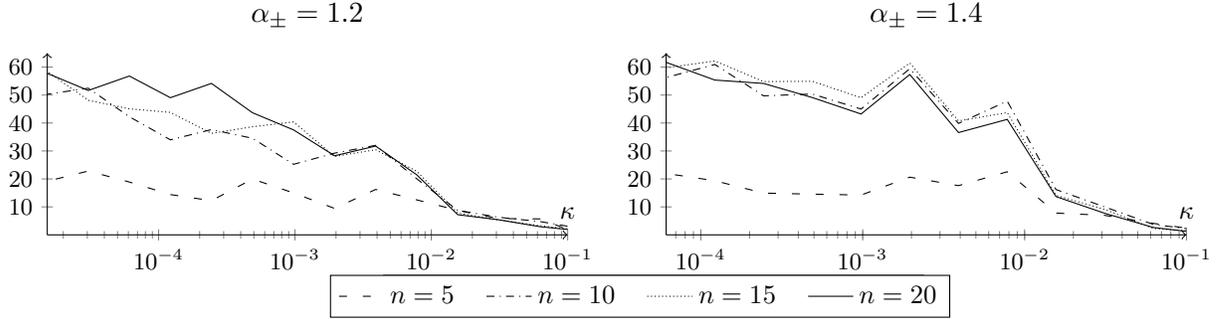

If the L\'evy process $X$ is a jump diffusion, i.e. 
$\nu(\R\setminus\{0\})<\infty$, we may apply 
Algorithms~\ref{alg:ARA} \&~\ref{alg:ARA_2} 
and~\nameref{alg:SBG} with $\kappa_1=\kappa_2=0$. 
In that case \nameref{alg:SBG} still outperforms 
Algorithm~\ref{alg:ARA_2} by a constant factor, with 
computational benefits being more pronounced when 
the total expected number of jumps 
$\lambda:=\nu(\R\setminus\{0\})T$ is large. 
%Moreover, 
%Figure~\ref{fig:ARA-Speedup} illustrates the comments made 
%in Subsection~\ref{subsec:SBG_sampler} above. Namely, we 
%show how the idea behind the SBG approximation can be 
%used to reduce the complexity of the exact simulation of 
%$\un\chi_T$ as a function of the expected number of observed 
%jumps $\lambda:=\nu(\R\setminus\{0\})T$. 
The cost reduction is most drastic when $\lambda$ is large, 
but the improvement is already significant for $\lambda=2$. 

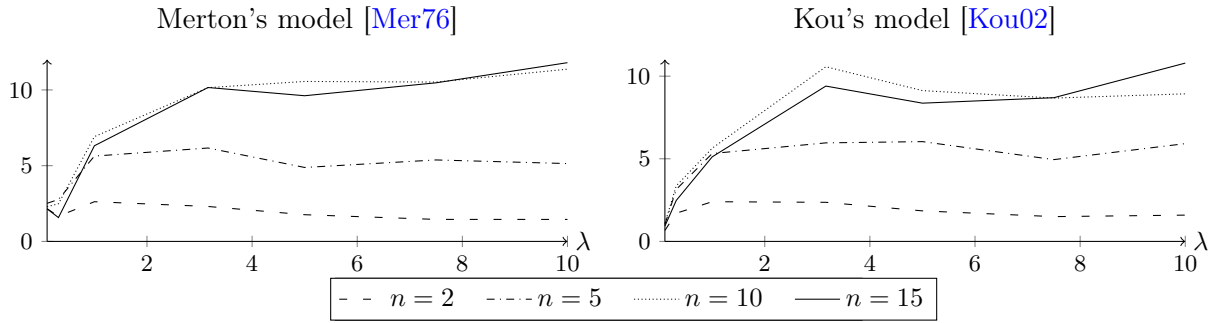
\begin{figure}[ht]
	\begin{center}
	\hspace{-3cm}
	\begin{subfigure}[l]{.49\linewidth}
		{\scalefont{.8}
		\begin{tikzpicture} 
		\begin{axis} 
		[
		title={\normalsize Merton's model~\cite{Merton76optionpricing}},
		ymin=0,
		ymax=12,
		xmin=.1,
		xmax=10,
		xlabel={\small $\lambda$},
		width=8.5cm,
		height=4cm,
		axis on top=true,
		axis x line=middle, 
		axis y line=left,
		axis line style={->},
		x label style={at={(axis description cs:1.03,0.1)},anchor=north},
		legend style={at={(1.06,-.2)},anchor=north, legend columns=-1}
		]
			
		% 2 stick
		\addplot[
		loosely dashed, 
		color=black,
		]
		coordinates {
(0.1,2.156)(0.3162,1.661)(1.0,2.621)(3.162,2.306)(5.0,1.763)(7.5,1.453)(10.0,1.452)
		};
		
		% 5 stick
		\addplot[
		dashdotted, 
		color=black,
		]
		coordinates {
(0.1,2.515)(0.3162,2.768)(1.0,5.633)(3.162,6.166)(5.0,4.882)(7.5,5.378)(10.0,5.139)
		};
	
		% 10 stick
		\addplot[
		densely dotted, 
		color=black,
		]
		coordinates {
(0.1,2.29)(0.3162,2.487)(1.0,6.918)(3.162,10.15)(5.0,10.56)(7.5,10.52)(10.0,11.37)
		};
		
		% 15 stick
		\addplot[
		solid, 
		color=black,
		]
		coordinates {
(0.1,2.15)(0.3162,1.573)(1.0,6.315)(3.162,10.16)(5.0,9.617)(7.5,10.47)(10.0,11.8)
		};
		
		% Legend 
		\legend {\small$n=2$\enskip,
			\small$n=5$\enskip,
			\small$n=10$\enskip,
			\small$n=15$};
		\end{axis}
		\end{tikzpicture}}%\subcaption{Merton's model}
	\end{subfigure}
	\hspace{-4cm}
	\begin{subfigure}[l]{.49\linewidth}
		{\scalefont{.8}
		\begin{tikzpicture} 
		\begin{axis} 
		[
		title={\normalsize Kou's model~\cite{Kou02JumpDiffusion}},
		ymin=0,
		ymax=11,
		xmin=0.1,
		xmax=10,
		xlabel={\small $\lambda$},
		width=8.5cm,
		height=4cm,
		axis on top=true,
		axis x line=middle, 
		axis y line=left,
		axis line style={->},
		x label style={at={(axis description cs:1.03,0.1)},anchor=north},
		legend style={at={(-.06,-.2)},anchor=north, legend columns=-1}
		]
		
		% 2 stick
		\addplot[
		loosely dashed, 
		color=black,
		]
		coordinates {
(0.1,0.6662)(0.3162,1.706)(1.0,2.402)(3.162,2.366)(5.0,1.848)(7.5,1.5)(10.0,1.591)
		};
		
		% 5 stick
		\addplot[
		dashdotted, 
		color=black,
		]
		coordinates {
(0.1,1.028)(0.3162,3.148)(1.0,5.319)(3.162,5.963)(5.0,6.039)(7.5,4.95)(10.0,5.917)
		};
		
		% 10 stick
		\addplot[
		densely dotted, 
		color=black,
		]
		coordinates {
(0.1,1.096)(0.3162,3.365)(1.0,5.639)(3.162,10.57)(5.0,9.126)(7.5,8.674)(10.0,8.93)
		};
		
		% 15 stick
		\addplot[
		solid, 
		color=black,
		]
		coordinates {
(0.1,0.9261)(0.3162,2.48)(1.0,5.117)(3.162,9.4)(5.0,8.366)(7.5,8.697)(10.0,10.79)
		};
		% Legend 
		\legend {\small$n=2\quad$\enskip,
			\small$n=5\quad$\enskip,
			\small$n=10\quad$\enskip,
			\small$n=15$};
		\end{axis}
		\end{tikzpicture}}%\subcaption{Kou's model}
	\end{subfigure}
	\caption{\footnotesize The pictures show, for multiple number of sticks 
	$n$, the ratio of the cost of Algorithm~\ref{alg:ARA_2} over the cost 
	of \nameref{alg:SBG} (both in seconds) for jump diffusions 
	as a function of the mean number of jumps 
	$\lambda=\nu(\R\setminus\{0\})T$. The ratio for $n=15$ is $11.8$ 
	(resp. $10.8$) in Merton's (resp. Kou's) model when $\lambda=10$.
	}\label{fig:ARA-Speedup}
	\end{center}
\end{figure}

%%% Main head for the e-companion
%\ECHead{Proofs of Statements}

\section{Proofs}
\label{sec:proofs}

%In the present section we will establish three results at the core of the 
%paper: Theorems~\ref{thm:Wd-triplet},~\ref{thm:W-marginal} 
%and~\ref{thm:K-marginal}. These proofs are split into two subsections. 

In the remainder of the paper we use the notation 
$\un\tau_t:=\un\tau_t(X)$, 
$\un\tau_t^{(\kappa)}:=\un\tau_t(X^{(\kappa)})$ for all $t>0$.

\subsection{Proof of Theorems~\ref{thm:W-marginal} and~\ref{thm:K-marginal}}
\label{subsec:marginals}

In this subsection we establish bounds on the Wasserstein and 
Kolmogorov distances between the increment $X_t$ and its Gaussian 
approximation $X^{(\kappa)}_t$ in~\eqref{eq:ARA}.

\begin{proof}[Proof of Theorem~\ref{thm:W-marginal}]
Recall the L\'evy-It\^o decomposition of $X$ at 
level $\kappa$ in~\eqref{eq:levy-ito} and the martingale 
$M^{(\kappa)} =\sigma B + J^{1,\kappa}$. Set $Z:=X-M^{(\kappa)}$ 
and note $X^{(\kappa)}= Z+\sqrt{\ov\sigma_\kappa^2+\sigma^2}W$,
	where $W$ is a standard Brownian motion in~\eqref{eq:ARA}, independent of $Z$. 
	Hence any coupling $(W_t, M^{(\kappa)}_t)$ yields a coupling of 
	$(X_t,X^{(\kappa)}_t)$ satisfying 
$\E[|X_t-X_t^{(\kappa)}|^p] =\E[|M_t^{(\kappa)}-\sqrt{\ov\sigma_\kappa^2+\sigma^2}W_t|^p]$. 
%Thus it suffices to  bound for the right-hand side in this equality.  
%	To establish such bound, it suffices to prove the bound 
%for the $L^p$-Wasserstein distance, $p\in\{1,2\}$ defined 
%in~\eqref{eq:Wasserstein}. 
Setting $W:=B$, which  amounts to the
	independence coupling $(W,J^{1,\kappa})$, and applying Jensen's inequality 
	for $p\in[1,2]$
	yields
\[
	\E[|X_t-X_t^{(\kappa)}|^p]^{2/p}\leq 
\E\big[\big|J_t^{1,\kappa}
-\big(\tsqrt{\ov\sigma_\kappa^2+\sigma^2}-\sigma\big)W_t\big|^2\big]
= \E[||J_t^{1,\kappa}|^2] + 
\big(\tsqrt{\ov\sigma_\kappa^2+\sigma^2}-\sigma\big)^2t
\le 2t\ov\sigma_\kappa^2.
\]
%	Hence the $L^p$-Wasserstein distance, defined in~\eqref{eq:Wasserstein},
%	satisfies
%$\W_p(X_t,X_t^{(\kappa)})\le \sqrt{2t}\ov\sigma_\kappa$ for $p\in[1,2]$.  

%Since
%$t>0$ is fixed, 
For any $m\in\N$ we have $M_t^{(\kappa)}\eqd \sum_{i=1}^m \xi_i$, 
where $\xi_1,\ldots,\xi_m$ are iid with 
$\xi_1\eqd M^{(\kappa)}_{t/m}$. Hence~\cite[Thm~16]{MR0388499} 
and~\cite[Thm~4.1]{MR2548505} imply the existence of universal 
constants $K_p$, $p\in[1,2]$, with $K_1=1/2$, satisfying 
\[
\W_p^p\big(M^{(\kappa)}_t,
	\tsqrt{\ov\sigma_\kappa^2+\sigma^2}W_t\big)
\le K_p^p\frac{[t(\ov\sigma_\kappa^2+\sigma^2)]^{p/2}
		\E[|\xi_1|^{p+2}]}
	{m^{p/2}\E[\xi_1^2]^{(p+2)/2}}
= K_p^p\frac{(m/t)\E[|M_{t/m}^{(\kappa)}|^{p+2}]}
	{\ov\sigma_\kappa^2+\sigma^2}
\quad\text{for all }m\in\N.
\]
According to~\cite[Thm~1.1]{MR2479503}, the limit as $m\to\infty$ 
of the right-hand side of the display above equals 
$K_p^p\int_{(-\kappa,\kappa)}|x|^{p+2}\nu(dx)
/(\ov\sigma_\kappa^2+\sigma^2)
\le K_p^p\kappa^p\varphi_\kappa^2$, implying the claim in the 
theorem.  
\end{proof}

\begin{proof}[Proof of Theorem~\ref{thm:K-marginal}]
(a) Define %We proceed as in the previous proof. It suffices to show that 
$d_\kappa:=\sup_{x\in\R}|\p(M_t^{(\kappa)}\le x)
	-\p(\tsqrt{\ov\sigma_\kappa^2 +\sigma^2}W_t\le x)|$ and note  
%$d_\kappa\le C_{\BE}t^{-1/2}\kappa\ov\sigma_\kappa^2/
%	(\ov\sigma_\kappa^2+\sigma^2)^{3/2}$, since 
\[
|\p(X_t\le x)-\p(X_t^{(\kappa)}\le x)|
=\big|\E\big[\p(M_t^{(\kappa)}\le x-Z_t|Z_t)
-\p\big(\tsqrt{\ov\sigma_\kappa^2 +\sigma^2} W_t
	\le x-Z_t\big|Z_t\big)\big]\big|
\le d_\kappa,
\]
where the processes $Z$ and $M^{(\kappa)}$ are as in the proof 
of Theorem~\ref{thm:W-marginal}. Since $M^{(\kappa)}$ is a L\'evy 
process, for any $m\in\N$ we have 
$M_t^{(\kappa)}\eqd \sum_{i=1}^m \xi_i$, where 
$\xi_1,\ldots,\xi_m$ are iid with $\xi_1\eqd M^{(\kappa)}_{t/m}$. 
By the Berry-Esseen inequality~\cite[Thm~1]{MR2929524}, there 
exists a constant $C_{\BE}\in(0,\tfrac{1}{2})$ such that 
\[
d_\kappa
\le \frac{C_{\BE}\E[|\xi_1|^3]}{\sqrt{m}\E[\xi_1^2]^{3/2}}
= \frac{C_{\BE}t}{m\sqrt{m}}\cdot\frac{(m/t)\E[|M^{(\kappa)}_{t/m}|^3]}
	{(t/m)^{3/2}(\ov\sigma_\kappa^2+\sigma^2)^{3/2}}
= C_{\BE}\frac{(m/t)\E[|M^{(\kappa)}_{t/m}|^3]}
	{\sqrt{t}(\ov\sigma_\kappa^2+\sigma^2)^{3/2}} 
\quad\text{ for all $m\in\N$.}
\]
According to~\cite[Thm~1.1]{MR2479503}, the limit as $m\to\infty$ 
of the right-hand side of the display above equals 
$C_{\BE}\int_{(-\kappa,\kappa)}|x|^3\nu(dx)
	/(\sqrt{t}(\ov\sigma_\kappa^2+\sigma^2)^{3/2})
\le C_{\BE}(\kappa/\ov\sigma_\kappa)\varphi_\kappa^3/\sqrt{t}$, 
implying~(a). 

(b) By~\cite[Thm~3.1(a)]{MR1449834}, $X_t$ has a smooth density 
$f_t$ and, given $T>0$, the constant 
$C'=\sup_{(t,x)\in(0,T]\times\R}t^{1/\delta}f_t(x)$ is finite. 
Applying~\eqref{eq:Lp-marginal} and~\eqref{eq:Lp-to-barrier} in 
Lemma~\ref{lem:Lp-to-barrier} with $p=2$ 
gives~\eqref{eq:K-marginal-3}. 
\end{proof}

\subsection{Proof of Theorem~\ref{thm:Wd-triplet}}
\label{subsec:proofThm1}

We recall an elementary result for stick-breaking processes. 

\begin{lem}
\label{lem:mom_ell}
Let $(\varpi_n)_{n\in\N}$ be a stick-breaking process on $[0,1]$ 
based on the law $\U(0,1)$. 
%and set 
%$\widetilde\varpi_n:=\sum_{k=n+1}^\infty\varpi_k$. 
For any measurable function $\phi\ge0$, we have  
\[
\sum_{n\in\N} \E[\phi(\varpi_n)]
%=\sum_{n=1}^\infty \E[\phi(\widetilde\varpi_n)]
=\int_0^1\frac{\phi(x)}{x}dx.
\]
In particular, for any $a_1,a_2>0$ and $b_1<b_2$ with $b_2>0$, we have
\[
\sum_{n\in\N}
\E[\min\{a_1\varpi_n^{b_1},a_2\varpi_n^{b_2}\}]
=\begin{cases}
\tfrac{a_2}{b_2}
\min\big\{1,\tfrac{a_1}{a_2}\big\}^{b_2/(b_2-b_1)}
+\tfrac{a_1}{b_1}\big(1
-\min\big\{1,\tfrac{a_1}{a_2}\big\}^{b_1/(b_2-b_1)}\big),
&b_1\neq0,\\
b_2^{-1}\min\{a_2,a_1\}\big(1+\log^+\big(\tfrac{a_2}{a_1}\big)\big),
&b_1=0.\\
\end{cases}
\]
\end{lem}

\begin{proof}
%Since $\varpi_n\eqd\widetilde\varpi_n$ is the product of $n$ 
%independent $\U(0,1)$ random variables, 
The law of $-\log\varpi_n$ 
is gamma with shape $n$ and scale $1$. Applying Fubini's theorem, 
implies 
\[
\sum_{n\in\N}\E[\phi(\varpi_n)]
=\sum_{n\in\N}
	\int_0^\infty\frac{x^{n-1}}{(n-1)!}e^{-x}\phi(e^{-x})dx
=\int_0^\infty\phi(e^{-x})dx
=\int_0^1 \frac{\phi(x)}{x}dx.
\]
The formula for $\phi(x):=\min\{a_1x^{b_1},a_2x^{b_2}\}$ follows 
by a direct calculation.
\end{proof}

The $L^p$-Wasserstein distance, defined in~\eqref{eq:Wasserstein}, 
satisfies
$\W_p^p(\xi,\xi_*)=\int_0^1 |F^{-1}(u)-F_*^{-1}(u)|^p du$, where $F^{-1}$ 
(resp. $F_*^{-1}$) is the right inverse of the distribution function $F$ (resp. 
$F_\ast$) of the real-valued random variable $\xi$ (resp. $\xi_*$) 
(see~\cite[Thm~2.10]{MR4028181}). 
Thus the \textit{comonotonic} (or \textit{minimal transport}) coupling, defined by
\begin{equation}\label{eq:comonotonic}
(\xi,\xi_*):=(F^{-1}(U),F_*^{-1}(U)) \qquad
\text{for some }U\sim\U(0,1),
\end{equation}
attains the infimum in definition~\eqref{eq:Wasserstein}.

\begin{lem}\label{lem:comonotonic_ind}
If  real-valued random variables $\xi$ and $\xi_*$ are comonotonically coupled, then 
\[
\E[|\1{\{\xi\le x\}}-\1{\{\xi_*\le x\}}|]
=|\E[\1{\{\xi\le x\}}-\1{\{\xi_*\le x\}}]|
\qquad \text{for any}\enskip x\in\R.
\]
\end{lem}

\begin{proof}
Suppose $(\xi,\xi_*)=(F^{-1}(U),F_*^{-1}(U))$ for some $U\sim\U(0,1)$, 
where $F$ and $F_*$ are the distribution functions of $\xi$ and $\xi_*$. 
Suppose $y:=F(x)\le F_*(x)=:y_*$. Since $F^{-1}$ and $F_*^{-1}$ are 
monotonic functions, it follows that $\1{\{\xi\le x\}}-\1{\{\xi_*\le x\}}\le 0$ 
a.s. since this difference equals $-1$ or $0$ according to $U\in(y,y_*]$ or 
$U\in(0,1)\setminus(y,y_*]$, respectively. If $y\ge y_*$, we have 
$\1{\{\xi\le x\}}-\1{\{\xi_*\le x\}}\ge 0$ a.s. In either case, the result follows. 
\end{proof}

For any $t>0$, let $G_t^\kappa$ denote the joint law of the
comonotonic coupling 
of $X_t$ and 
$X_t^{(\kappa)}$ %are comonotonically coupled, 
defined in~\eqref{eq:comonotonic}. Note that a coupling
$(X_t,X_t^{(\kappa)})$
with law 
$G_t^\kappa$
satisfies the inequality in Theorem~\ref{thm:W-marginal}.
The following lemma is crucial in 
the proof of Theorem~\ref{thm:Wd-triplet}. 

\begin{lem}\label{lem:bound_sum_xi}
Let $\ell=(\ell_n)_{n\in\N}$ be a stick-breaking process on $[0,t]$ 
and $(\xi_n,\xi_n^{(\kappa)})$, $n\in\N$, a sequence of 
random vectors that, conditional on $\ell$, are independent 
and satisfy $(\xi_n,\xi_n^{(\kappa)})\sim G_{\ell_n}^{\kappa}$ 
for all $n\in\N$. Then for any $p\in[1,2]$ and $x\in\R$ we have 
\begin{equation}\label{eq:bound_sum_xi}
\E\bigg[\bigg(\sum_{n=1}^\infty |\xi_n-\xi_n^{(\kappa)}|\bigg)^p\bigg]^{1/p}
\le \mu_p(\kappa,t)
\quad\text{and}\quad
\E\bigg[\sum_{n=1}^\infty \ell_n
	\big|\1\{\xi_n\le x\}-\1\{\xi_n^{(\kappa)}\le x\}\big|\bigg]
\le \mu_0^\tau(\kappa,t),
\end{equation}
where $\mu_p$ and $\mu_0^\tau$ are defined in~\eqref{def:mu} 
and~\eqref{eq:mu_0}, respectively. Moreover, 
if Assumption~(\nameref{asm:(O)}) holds,
	%$\inf_{u\in(0,1]}u^{\delta-2}(\ov\sigma^2_{u}+\sigma^2)>0$ for some 
%$\delta\in(0,2]$, 
	then for every $T>0$ there exists a constant $C>0$, dependent only on $(T,\delta,\sigma,\nu)$, such that for all $t\in[0,T]$, 
$\kappa\in(0,1]$ and $x\in\R$ we have 
\begin{equation}\label{eq:bound_sum_xi2}
\E\bigg[\sum_{n=1}^\infty \ell_n
\big|\1\{\xi_n\le x\}-\1\{\xi_n^{(\kappa)}\le x\}\big|\bigg]
\le \mu_\delta^\tau(\kappa,t),
\end{equation}
where $\mu_\delta^\tau$ is defined in~\eqref{def:mu_tau_delta}.
\end{lem}

\begin{proof}
Note that $\mu_p(\kappa,t)=\mu_2(\kappa,t)$ for all $p\in(1,2]$. 
Hence, by Jensen's inequality, in~\eqref{eq:bound_sum_xi} we need 
only consider $p\in\{1,2\}$. 
Pick $n\in\N$ 
and set 
$\kappa_p:=K_p^p\kappa^p\varphi_\kappa^2$, $p\in\{1,2\}$, 
where $K_p$ and $\varphi_\kappa$ are as in the statement of Theorem~\ref{thm:W-marginal}.
	Condition on $\ell_n$ and apply the bound 
in~\eqref{eq:Lp-marginal} to obtain 
\begin{equation}\label{eq:coupling_bound}
\E[|\xi^{(\kappa)}_n-\xi_n|^p|\ell_n]
\leq \min\big\{2^{p/2}\ov\sigma_\kappa^p\ell_n^{p/2},
	\kappa_p\big\},
\qquad p\in\{1,2\}.
\end{equation}
An application of~\eqref{eq:coupling_bound} %for $p=1$ 
and Lemma~\ref{lem:mom_ell} yield the first inequality 
in~\eqref{eq:bound_sum_xi} for $p=1$:
\[
\sum_{n=1}^\infty \E\big[\big|\xi_n-\xi^{(\kappa)}_n\big|\big]
\le\sum_{n=1}^\infty
	\E[\min\big\{\sqrt{2\ell_n}\ov\sigma_\kappa,\kappa_1\big\}]
=2\min\big\{\sqrt{2t}\ov\sigma_\kappa,
	\kappa_1\big\}\big(
	1+\log^+\big(\sqrt{2t}\ov\sigma_\kappa/\kappa_1\big)\big).
\]

Consider the case $p=2$. A simple expansion yields 
\[
\E\bigg[\bigg(\sum_{n=1}^\infty
	\big|\xi_n-\xi^{(\kappa)}_n\big|\bigg)^2\bigg]
=\sum_{n=1}^\infty \E\big[\big(\xi_n-\xi^{(\kappa)}_n\big)^2\big]
	+2\sum_{n=1}^\infty \sum_{m=n+1}^\infty 
		\E\big[\big|\xi_n-\xi^{(\kappa)}_n\big|
			\big|\xi_m-\xi^{(\kappa)}_m\big|\big].
\]
We proceed to bound the two sums. The inequality 
in~\eqref{eq:coupling_bound} for $p=2$ and 
Lemma~\ref{lem:mom_ell} imply 
\[
\sum_{n=1}^\infty \E\big[\big(\xi_n-\xi^{(\kappa)}_n\big)^2\big]
\le\sum_{n=1}^\infty
\E\big[\min\big\{2\ov\sigma_\kappa^2\ell_n,
	\kappa_2\big\}\big]
=\min\big\{2t\ov\sigma_\kappa^2,\kappa_2\big\}\big(
		1+2\log^+\big(\sqrt{2t}\ov\sigma_\kappa/\sqrt{\kappa_2}
	\big)\big).
\]
Define the $\sigma$-algebra $\mathcal{F}_n
:=\sigma(\ell_1,\ldots,\ell_n)$ and use the conditional 
independence to obtain 
\[
\E\big[\big|\xi_n-\xi^{(\kappa)}_n\big|
	\big|\xi_m-\xi^{(\kappa)}_m\big|
\big|\mathcal{F}_m\big]
\le\min\{\sqrt{2\ell_n}\ov\sigma_\kappa,\kappa_1\}
	\min\{\sqrt{2\ell_m}\ov\sigma_\kappa,\kappa_1\},
\qquad n<m.
\]
Note that $(\ell_m/L_n)_{m=n+1}^\infty$ is a stick-breaking process 
on $[0,1]$ independent of $\mathcal{F}_n$. Use the tower property 
and apply~\eqref{eq:Lp-marginal} and Lemma~\ref{lem:mom_ell} to 
get 
\begin{align*}
\sum_{m=n+1}^\infty
\E\big[\big|\xi_n-\xi^{(\kappa)}_n\big|
\big|\xi_m-\xi^{(\kappa)}_m\big|\big|\mathcal{F}_n\big]
%&\le\sum_{m=n+1}^\infty
%\E\big[\min\big\{\sqrt{2\ell_n}\ov\sigma_\kappa,\kappa_1\big\}
%	\min\big\{\sqrt{2\ell_m}\ov\sigma_\kappa,\kappa_1\big\}
%		\big|\mathcal{F}_n\big]\\
&\leq \min\{\sqrt{2\ell_n}\ov\sigma_\kappa,\kappa_1\}
	\sum_{m=n+1}^\infty
		\E\big[\min\big\{\sqrt{2\ell_m}\ov\sigma_\kappa,\kappa_1\big\}
			\big|\mathcal{F}_n\big]\\
&=2\min\big\{\sqrt{2\ell_n}\ov\sigma_\kappa,\kappa_1\big\}
\min\big\{\sqrt{2L_n}\ov\sigma_\kappa,\kappa_1\big\}\Big(
1+\log^+\Big(\tfrac{\sqrt{2L_n}\ov\sigma_\kappa}{\kappa_1}\Big)\Big)\\
	& \leq 2 \min\big\{2L_{n-1}\ov\sigma_\kappa^2,\kappa_1^2\big\}
	\big(1 + \log^+\big(\sqrt{2t}\ov\sigma_\kappa/\kappa_1\big)\big),
\end{align*}
where  
$\max\{L_n,\ell_n\}\le L_{n-1}\le t$
is used in the last step.
Since $\ell_n\eqd L_n$, $n\in\N$,
Lemma~\ref{lem:mom_ell} yields 
\begin{align*}
2\sum_{n=1}^\infty\sum_{m=n+1}^\infty
\E\big[\big|\xi_n-\xi^{(\kappa)}_n\big|
\big|\xi_m-\xi^{(\kappa)}_m\big|\big]
&\le 4\sum_{n=1}^\infty
	\E\left[	\min\big\{2L_{n-1}\ov\sigma_\kappa^2,\kappa_1^2\big\}\right]
	\big(1
		+ \log^+\big(\sqrt{2t}\ov\sigma_\kappa/\kappa_1\big)\big)\\
%&\le 2\min\big\{2t\ov\sigma_\kappa^2,\kappa_1^2\big\}
%	\big(2 
%	+ \log^+\big(\sqrt{2t}\ov\sigma_\kappa/\kappa_1\big)\big)^2
	& =2 \mu_1(\kappa,t)^2.
\end{align*}
Putting everything together yields the first inequality 
in~\eqref{eq:bound_sum_xi} for $p=2$. 

Next we prove the second inequality in~\eqref{eq:bound_sum_xi}. 
By Lemma~\ref{lem:comonotonic_ind}, we have 
\begin{equation}
\label{eq:ell_xi}
\E[|\1{\{\xi_n\le x\}}-\1{\{\xi_n^{(\kappa)}\le x\}}||\ell_n]
%=|\E[\1{\{\eta_s<0\}}-\1{\{\eta'_s<0\}}]|
=\big|\p(X_{\ell_n}\le x|\ell_n)
	-\p\big(X^{(\kappa)}_{\ell_n}\le x\big|\ell_n\big)\big|.
\end{equation}
Applying~\eqref{eq:K-marginal-1} in 
Theorem~\ref{thm:K-marginal}(a) implies 
$\ell_n\big|\p(X_{\ell_n}\le x|\ell_n)
	-\p\big(X^{(\kappa)}_{\ell_n}\le x\big|\ell_n\big)\big|\le
\frac{1}{2}(\kappa/\ov\sigma_\kappa)\varphi_\kappa^3\ell_n^{1/2}$. 
By Fubini's theorem, conditioning in each summand on $\ell_n$, 
equality~\eqref{eq:ell_xi} and Lemma~\ref{lem:mom_ell}, we have 
\[
\E\bigg[\sum_{n\in\N}\ell_n
	\big|\1\{\xi_n\le x\}-\1\{\xi_n^{(\kappa)}\le x\}\big|\bigg] 
\le\frac{1}{2}\sqrt{t} (\kappa/\ov\sigma_\kappa)\varphi_\kappa^3
	\sum_{n\in\N}\E\big[(\ell_n/t)^{1/2}\big]
= \mu_0^\tau(\kappa,t).
\]

Next let $\delta\in(0,2]$ satisfy 
$\inf_{u\in(0,1]}u^{\delta-2}(\ov\sigma_u^2+\sigma^2)>0$. 
By~\eqref{eq:K-marginal-3} in Theorem~\ref{thm:K-marginal}(b), 
we see that $\ell_n\big|\p(X_{\ell_n}\le x|\ell_n)
	-\p\big(X^{(\kappa)}_{\ell_n}\le x\big|\ell_n\big)\big|
\le\psi_\kappa^{2/3}\ell_n^{1-2/(3\delta)}$, where 
$\psi_\kappa = C\kappa\varphi_\kappa$ as defined 
in~\eqref{eq:mu_tau}. Moreover, we have 
$\ell_n\big|\p(X_{\ell_n}\le x|\ell_n)
-\p\big(X^{(\kappa)}_{\ell_n}\le x\big|\ell_n\big)\big|\le \ell_n$.
Hence by~\eqref{eq:ell_xi} and Lemma~\ref{lem:mom_ell}, 
we obtain 
\begin{align*}
&\hspace{-16mm}
\sum_{n=1}^\infty\E\big[\ell_n\big|\1{\{\xi_n\le x\}}
	-\1{\{\xi^{(\kappa)}_n\le x\}}\big|\big]
\le \sum_{n=1}^\infty\E\big[\min\big\{\ell_n,
	\psi_\kappa^{2/3}\ell_n^{1-2/(3\delta)}\big\}\big]\\
&=\begin{cases}
\min\{t,\psi_\kappa^\delta\}
	+\frac{3\delta}{3\delta-2}\psi_\kappa^{2/3}t^{1-\tfrac{2}{3\delta}}
	\big(1-\min\big\{1,t^{-1/\delta}\psi_\kappa\big\}^{\delta-2/3}\big),
&\delta\in(0,2]\setminus\{\tfrac{2}{3}\},\\
\min\{t,\psi_\kappa^{2/3}\}
	(1+\log^+(t\psi_\kappa^{-2/3})),
&\delta=\tfrac{2}{3}.
\end{cases}\qedhere
\end{align*}
\end{proof}

\begin{proof}[Proof of Theorem~\ref{thm:Wd-triplet}]
Let $\ell=(\ell_n)_{n\in\N}$ and  $(\xi_n,\xi_n^{(\kappa)})$, 
$n\in\N$, be as in Lemma~\ref{lem:bound_sum_xi}. 
%	Recall that the Skorokhod space $\D[0,t]$ of 
%	\cadlag~functions on $[0,T]$ (see~\cite[p.~109]{MR1700749}) is 
%	Polish~\cite[p.~112]{MR1700749} and thus 
%	Borel~\cite[Thm~A1.2]{MR1876169}. 
Define the vector 
\begin{align*}
\big(\zeta_1,\zeta_2,\zeta_3,
	\zeta_1^{(\kappa)},\zeta_2^{(\kappa)},\zeta_3^{(\kappa)}\big)
&:=\sum_{n=1}^\infty 
	\big(\xi_n,\min\{\xi_n,0\},\ell_n\cdot\1{\{\xi_n\le 0\}},
		\xi^{(\kappa)}_n,\min\{\xi_n^{(\kappa)},0\},\ell_n\cdot\1{\{\xi^{(\kappa)}_n\le 0\}}\big).
\end{align*}
By~\eqref{eq:chi_infty} and~\eqref{eq:comonotonic}, 
we have $(\zeta_1,\zeta_2,\zeta_3)\eqd\un\chi_t$ 
and $(\zeta^{(\kappa)}_1,\zeta^{(\kappa)}_2,\zeta^{(\kappa)}_3)
	\eqd\un\chi_t^{(\kappa)}$. Hence, 
it suffices to show that these vectors satisfy~\eqref{eq:Lp-chi}, 
\eqref{eq:mu_0} and~\eqref{eq:mu_tau}. 
Since $x\mapsto\min\{x,0\}$ is in $\Lip_1(\R)$, the inequalities
\[
\max\big\{\big|\zeta_1-\zeta_1^{(\kappa)}\big|,
	\big|\zeta_2-\zeta_2^{(\kappa)}\big|\big\}
\le\sum_{n=1}^\infty \big|\xi_n-\xi^{(\kappa)}_n\big|
\quad\text{and}\quad
\big|\zeta_3-\zeta^{(\kappa)}_3\big|
\le\sum_{n=1}^\infty
	\ell_n\big|\1{\{\xi_n\le 0\}}-\1{\{\xi^{(\kappa)}_n\le 0\}}\big|
\]
follow from the triangle inequality.
The theorem follows from Lemma~\ref{lem:bound_sum_xi}. 
\end{proof}

\begin{rem}
\label{rem:Wd-triplet}
%If $\delta\in(0,2)$ satisfies $\varsigma^2=\inf_{u\in(0,1]}u^{\delta-2}
%\ov\sigma^2_u>0$, it is possible to define $\mu^\tau_\delta$ in a way that 
%features the function $t\ov\sigma_\kappa^2$. However, under our 
%assumption, we have $\ov\sigma_\kappa\ge\varsigma\kappa^{1-\delta/2}
%\ge\varsigma\kappa$ for any $\kappa\in(0,1]$. Hence the bound above, 
%up to constants, is the best we can obtain with our methods.
Let $C_t$ and $C_t^{(\kappa)}$ denote the convex minorants of 
$X$ and $X^{(\kappa)}$ on $[0,t]$, respectively. 
Couple $X$ and $X^{(\kappa)}$ in such a 
way that the stick-breaking processes describing the lengths of the 
faces of their convex minorants (see~\cite[Thm~1]{MR2978134}  
and~\cite[Sec.~4.1]{LevySupSim}) coincide. (The Skorokhod space 
$\D[0,t]$ and the space of sequences on $\R$ are both Borel 
spaces by~\cite[Thms~A1.1, A1.2 \&~A2.2]{MR1876169}, so the 
existence of such a coupling is guaranteed 
by~\cite[Thm~6.10]{MR1876169}.) Geometric arguments, similar to 
the ones in~\cite{EpsStrongSim}, show that the sequences of heights 
of the faces of the convex minorants, denoted by 
$(\xi_n)_{n\in\N}$ and $(\xi^{(\kappa)}_n)_{n\in\N}$, satisfy 
\[
\sup_{s\in[0,t]}\big|C_t(s)-C_t^{(\kappa)}(s)\big|
\le \sum_{n=1}^\infty\big|\xi_n-\xi_n^{(\kappa)}\big|
\quad\text{and}\quad
\big|\un\tau_t-\un\tau_t^{(\kappa)}\big|
\le \sum_{n=1}^\infty 
	\ell_n\big|\1\{\xi_n\le 0\}-\1\{\xi_n^{(\kappa)}\le 0\}\big|.
\]
Hence, if $(\xi_n,\xi^{(\kappa)}_n)$, $n\in\N$, are coupled as in 
Lemma~\ref{lem:bound_sum_xi}, the inequalities 
in~\eqref{eq:bound_sum_xi} and~\eqref{eq:bound_sum_xi2}
yield the same bounds as in Theorem~\ref{thm:Wd-triplet} but in a 
stronger metric (namely, the distance between the convex minorants 
in the supremum norm), while retaining the control on the 
time of the infimum. 
\end{rem}

\subsection{The proofs of Propositions~\ref{prop:logLip},~\ref{prop:barrier},~\ref{prop:barrier-tau} and~\ref{prop:SimpAsmH}}
\label{subsec:proofProps}

%Before the proof of Proposition~\ref{prop:logLip}, note that the 
The L\'evy-Khintchine formula for $X_t$ in~\eqref{eq:Levy_Khinchin},
the definition of $X^{(\kappa)}_t$ in~\eqref{eq:ARA}
and the inequality $e^z\geq 1+z$ (for all $z\in\R$) imply
\begin{equation}\label{eq:exp_bound}
\begin{split}
t^{-1}\log\E\big[e^{uX^{(\kappa)}_t}\big]
&=bu+(\sigma^2+\ov\sigma^2_{\kappa})u^2/2
+\int_{\R\setminus(-\kappa,\kappa)}(e^{ux}-1-ux\cdot\1_{(-1,1)}(x))\nu(dx)\\
&\leq\ov\sigma^2_{\kappa}u^2/2+t^{-1}\log\E\big[e^{uX_t}\big]
\quad	\text{for any $u\in\R$, $t>0$ and $\kappa\in(0,1]$.}
\end{split}
\end{equation}
Thus
$\E[\exp(uX^{(\kappa)}_t)]\leq\E[\exp(uX_t)]
\exp(\ov\sigma^2_{\kappa}u^2t/2)$ and,  in particular, the Gaussian approximation 
$X^{(\kappa)}$ has as many exponential moments as the L\'evy process $X$. 

\begin{proof}[Proof of Proposition~\ref{prop:logLip}]
By~\cite[Thm~6.16]{MR2459454}, there exists a coupling between 
$(\xi,\zeta)\eqd(X_T,\un{X}_T)$ and 
$(\xi',\zeta')\eqd (X^{(\kappa)}_T,\un{X}^{(\kappa)}_T)$, such that 
$\E[(|\xi-\xi'|+|\zeta-\zeta'|)^2]^{1/2}
=\W_2((X_T,\un{X}_T),(X^{(\kappa)}_T,\un{X}^{(\kappa)}_T))$. 
The identity $e^b-e^a=\int_a^b e^zdz$ implies that, for $x\ge y$ and 
$x'\ge y'$,  we have
\begin{equation}
\label{eq:loclip}
|f(x,y)-f(x',y')|\leq K(|e^x-e^{x'}|+|e^y-e^{y'}|)
\leq K(|x-x'|+|y-y'|)e^{\max\{x,x'\}}.
\end{equation}
Apply this inequality, the Cauchy-Schwartz inequality, the elementary 
inequalities, which hold for all $a,b\geq0$, $(a+b)^2\le 2(a^2+b^2)$ and 
$(a+b)^{1/2}\le a^{1/2}+b^{1/2}$ 
and the bound in~\eqref{eq:exp_bound} to obtain  
\begin{align*}
\E|f(\xi,\zeta) - f(\xi',\zeta')|
&\le K\E[(|\xi-\xi'|+|\zeta-\zeta'|)^2]^{1/2}\E[(e^{\xi}+e^{\xi'})^2]^{1/2}\\
&\le 2^{1/2}K\W_2((X_T,\un{X}_T),(X^{(\kappa)}_T,\un{X}^{(\kappa)}_T))
	\E[e^{2\xi}+e^{2\xi'}]^{1/2}\\
&\le 2K\W_2((X_T,\un{X}_T),(X^{(\kappa)}_T,\un{X}^{(\kappa)}_T))
	\E[e^{2X_T}]^{1/2}(1 + e^{\ov\sigma^2_{\kappa}T}).
\end{align*}
Applying Corollary~\ref{cor:Wd-chi} gives the desired inequality, 
concluding the proof of the proposition.
\end{proof}

We now introduce a tool that uses the $L^p$-distance 
$\E[|\zeta-\zeta'|^p]^{1/p}$ between random variables $\zeta$ 
and $\zeta'$ to bound the $L^1$-distance between the indicators 
$\E|\1_{[y,\infty)}(\zeta)-\1_{[y,\infty)}(\zeta')|$. 

\begin{lem}\label{lem:Lp-to-barrier}
Let $(\xi,\zeta)$ and $(\xi',\zeta')$ be random vectors in 
$\R^n\times\R$. Fix $y\in\R$ and let $h\in\Lip_K(\R^n)$ satisfy  
$0\le h\le M$ for some constants $K,M\geq0$. Then for any 
$p,r>0$, we have 
\begin{align}
\label{eq:lem-barrier}
\E\big|h(\xi)\1_{[y,\infty)}(\zeta)-h(\xi')\1_{[y,\infty)}(\zeta')\big|
&\le K\E\|\xi-\xi'\| 
+ M\p(|\zeta-y|\le r) + Mr^{-p}\E[|\zeta-\zeta'|^p].
\end{align}
In particular, if $|\p(\zeta\le y)-\p(\zeta\le y+r)|\le C|r|^\gamma$ 
for some $C,\gamma>0$ and all $r\in\R$, then 
\begin{equation}
\label{eq:Lp-to-barrier}	
\begin{split}
\E\big|h(\xi)\1_{[y,\infty)}(\zeta)
	- h(\xi')\1_{[y,\infty)}(\zeta')\big|
&\le K\E\|\xi-\xi'\|
	+ M(2C\gamma/p)^{\frac{p}{p+\gamma}}(1+p/\gamma)
	\E[|\zeta-\zeta'|^p]^{\frac{\gamma}{p+\gamma}}.
\end{split}\end{equation}
\end{lem}

\begin{rem}
%\label{rem:Lp-to-barrier}
An analogous bound to the one in~\eqref{eq:lem-barrier} holds for 
the indicator $\1_{(-\infty,y]}$. Moreover, it follows from the 
proof below that the boundedness of the function $h$ assumed in 
Lemma~\ref{lem:Lp-to-barrier} may be replaced with a moment 
assumption $\xi,\xi'\in L^q$ for some $q>1$. In such a case, 
H\"older's inequality could be invoked to obtain an analogue 
to~\eqref{eq:barrier-dec-bound} below. Similar arguments 
may be used to simultaneously handle multiple indicators. 
\end{rem}

\begin{proof}[Proof of Lemma~\ref{lem:Lp-to-barrier}]
Applying the local $\gamma$-H\"older continuous property of the 
distribution function of $\zeta$ to~\eqref{eq:lem-barrier} and 
optimising over $r>0$ yields~\eqref{eq:Lp-to-barrier}. Thus, it 
remains to establish~\eqref{eq:lem-barrier}.
	
Elementary set manipulation yields 
\begin{align*}
|\1_{\{y\le\zeta\}}- \1_{\{y\le\zeta'\}}|
&= |\1_{\{\zeta'<y\le\zeta\}}-\1_{\{\zeta<y\le\zeta'\}}|\\
&\le \1_{\{|\zeta-\zeta'|>r,\zeta'<y\le\zeta\}} 
	+ \1_{\{|\zeta-\zeta'|\le r,\zeta'< y\le\zeta\}} 
	+ \1_{\{|\zeta-\zeta'|>r,\zeta<y\le\zeta'\}}  
	+ \1_{\{|\zeta-\zeta'|\le r,\zeta< y\le\zeta'\}}\\
&\le \1_{\{|\zeta-\zeta'|>r\}} + \1_{\{|\zeta-y|\le r\}}.
\end{align*}
Hence, the triangle inequality and the Lipschitz property gives 
\begin{equation}
\label{eq:barrier-dec-bound}
\begin{split}
|h(\xi)\1_{[y,\infty)}(\zeta)-h(\xi')\1_{[y,\infty)}(\zeta')|
&\le|h(\xi)| |\1_{[y,\infty)}(\zeta)-\1_{[y,\infty)}(\zeta')|
	+|h(\xi)-h(\xi')| \1_{[y,\infty)}(\zeta')\\
&\le M(\1_{\{|\zeta-y|\le r\}} + \1_{\{|\zeta-\zeta'|>r\}})
	+ K\|\xi-\xi'\|.
\end{split}
\end{equation}
Taking expectations and using Markov's inequality 
$\p(|\zeta-\zeta'|>r)\le r^{-p}\E[|\zeta-\zeta'|^p]$ 
yields~\eqref{eq:lem-barrier}.
\end{proof}

\begin{proof}[Proof of Proposition~\ref{prop:barrier}]
Theorem~\ref{thm:Wd-triplet} and~\eqref{eq:Lp-to-barrier} in 
Lemma~\ref{lem:Lp-to-barrier} (with $C$ and $\gamma$ given in 
Assumption~(\nameref{asm:(H)}) and $p=2$) applied to 
$(X_T,\un{X}_T)$ and $(X^{(\kappa)}_T,\un{X}^{(\kappa)}_T)$ 
under the SBG coupling give the claim. 
%yield the following inequalities
%\begin{align*}
%\E\big|f(X_T,\un{X}_T)-
%	f\big(X^{(\kappa)}_T,\un{X}^{(\kappa)}_T\big)\big|
%&\le 4MC(2r)^\gamma+K\E|X_T-X^{(\kappa)}_T|
%+3(M/r)\E|\un{X}_T-\un{X}^{(\kappa)}_T|,\\
%&\le 8MCr^\gamma
%	+(K+8M/r)\mu_1(\kappa,T),\quad\text{for any $r>0$.}
%\end{align*}
%	Minimising over $r\in(0,\infty)$ gives 
%$r=(\mu_1(\kappa,T)/(C\gamma))^{1/(1+\gamma)}$ 
%and the desired result follows.
\end{proof}

\begin{proof}[Proof of Proposition~\ref{prop:barrier-tau}]
Analogous to the proof of Proposition~\ref{prop:barrier}, 
applying Theorem~\ref{thm:Wd-triplet} and~\eqref{eq:Lp-to-barrier} 
in Lemma~\ref{lem:Lp-to-barrier} (with $C$ and $\gamma$ given 
in Assumption~(\nameref{asm:(Htau)}) and $p=1$), gives the result.
\end{proof}

\begin{lem}
\label{lem:density-tau}
Suppose $X$ is not a compound Poisson process. 
Then the law of $\un\tau_T$ is absolutely continuous on $(0,T)$ 
and its density is locally bounded on $(0,T)$. 
\end{lem}

\begin{proof}
If $X$ or $-X$ is a subordinator then $\un\tau_T$ is a.s. $0$ or $T$, 
respectively. In either case, the result follows immediately. 
Suppose now that neither $X$ nor $-X$ is a subordinator. Denote by 
$\ov{n}(\zeta>\cdot)$ (resp. $\un{n}(\zeta>\cdot)$) the intensity 
measures of the lengths $\zeta$ of the excursions away from $0$ of 
the Markov process $\ov{X}-X$ (resp. $X-\un{X}$). Then, 
by Theorem~5 in~\cite{MR3098676} with $F\equiv K\equiv 1$, 
the law of $\un\tau_T$ can only have atoms at $0$ or $T$, is 
absolutely continuous on $(0,T)$ and its density is given by 
$s\mapsto\un{n}(\zeta>s)\ov{n}(\zeta>T-s)$, $s\in(0,T)$. The maps 
$s\mapsto\un{n}(\zeta>s)$ and $s\mapsto\ov{n}(\zeta>s)$ are 
non-increasing, so the density is bounded on any compact subset of 
$(0,T)$, completing the proof.
\end{proof}

In preparation for the next result, we introduce the following 
assumption. 

\begin{asm*}[S-$\alpha$]\label{asm:(S)}
There exists some function $a:(0,\infty)\to(0,\infty)$ such that 
$X_t/a(t)$ converges in distribution to an $\alpha$-stable law as 
$t\to 0$.
\end{asm*}

\begin{prop}
\label{prop:AsmH}
Let Assumption~(\nameref{asm:(S)}) hold for some $\alpha\in(0,2]$.\\
\nf{(a)} If $\alpha>1$, then Assumption~(\nameref{asm:(H)}) holds 
uniformly on compact subsets of $(-\infty,0)$ with $\gamma=1$.\\
\nf{(b)} Suppose $\rho:=\lim_{t\to 0}\p(X_t>0)\in(0,1)$. Then for any 
$\gamma\in(0,\min\{\rho,1-\rho\})$, there exists some constant 
$C>0$ such that Assumption~(\nameref{asm:(Htau)}) holds for all 
$s\in[0,T]$. 
\end{prop}

Note that $\rho$ is well defined under 
Assumption~(\nameref{asm:(S)}) and that $X_t/a(t)$ can only have a 
nonzero weak limit as $t\to 0$ if the limit is $\alpha$-stable.  
Moreover, in that case, $a$ is necessarily regularly varying at $0$ 
with index $1/\alpha$ and $\alpha$ is given in terms of the L\'evy 
triplet $(\sigma^2,\nu,b)$ of $X$:
\[
\alpha
:=\begin{cases}
2, &\sigma\ne 0,\\
1, &\beta\in(0,1)\text{ and }b\ne\int_{(-1,1)}x\nu(dx),\\
\beta, &\text{otherwise}, % provided $\beta>0$.}
\end{cases}
\]
where $\beta$ is the BG index introduced in~\eqref{eq:I0_beta}. 
In fact, the assumptions of Proposition~\ref{prop:SimpAsmH} imply 
Assumption~(\nameref{asm:(S)}) by~\cite[Prop.~2.3]{ZoomIn}, 
so Proposition~\ref{prop:AsmH} generalises 
Proposition~\ref{prop:SimpAsmH}. We refer the reader 
to~\cite[Sec.~3~\&~4]{MR3784492} for conditions that are 
equivalent to Assumption~(\nameref{asm:(S)}). 

Assumption~(\nameref{asm:(S)}) allows for the cases $\rho=0$ or 
$\rho=1$ when $\alpha\le 1$, correspond to the stable limit being 
a.s. negative or a.s. positive, respectively. In these cases, the 
distribution of $\ov\tau_T(X)$ may have an atom at $0$ or $T$, 
while the law of $\ov\tau_T(X^{(\kappa)})$ is absolutely continuous, 
making the convergence in Kolmogorov distance impossible. 
This is the reason for excluding $\rho\in\{0,1\}$
in Proposition~\ref{prop:AsmH}.

\begin{proof}[Proof of Proposition~\ref{prop:AsmH}]
By~\cite[Lem.~5.7]{ZoomIn}, under the assumptions in part~(a) of 
the proposition, $\un X_T$ has a continuous density on $(-\infty,0)$, 
implying the conclusion in~(a). 

Since $\rho=\lim_{t\to0}\p(X_t>0)\in(0,1)$, $0$ is regular for both 
half-lines by Rogozin's criterion~\cite[Thm~47.2]{MR3185174}. 
\cite[Thm~6]{MR3098676} then asserts that the law of $\un\tau_T$ 
is absolutely continuous with density given by 
$s\mapsto\un{n}(\zeta>s)\ov{n}(\zeta>T-s)$, $s\in(0,T)$. 
The maps $s\mapsto\un{n}(\zeta>s)$ and $s\mapsto\ov{n}(\zeta>s)$ 
are non-increasing and, by~\cite[Prop.~3.5]{ZoomIn}, regularly 
varying with indices $\rho-1$ and $-\rho$, respectively. Thus for 
any $\gamma\in(0,\min\{\rho,1-\rho\})$ there exists some 
$C>0$ such that $\un{n}(\zeta>s)\ov{n}(\zeta>T-s)
\le Cs^{\gamma-1}(T-s)^{\gamma-1}$ for all $s\in(0,T)$. Thus, for any 
$s,t\in[0,T/2]$ with $t\ge s$, we have 
\begin{align*}
\p(\un\tau_T\le t)-\p(\un\tau_T\le s)
&\le\int_s^t Cu^{\gamma-1}(T-u)^{\gamma-1}du
\le C\int_s^t u^{\gamma-1}(T/2)^{\gamma-1}du\\
&\le C\gamma^{-1}(T/2)^{\gamma-1}(t^\gamma-s^\gamma)
\le C\gamma^{-1}(T/2)^{\gamma-1}(t-s)^\gamma.
\end{align*}
since the map $x\mapsto x^\gamma$ is concave. 
A similar bound holds for $s,t\in[T/2,T]$. Moreover, when 
$s\in[0,T/2]$ and $t\in[T/2,T]$ we have 
\begin{align*}
\p(\un\tau_T\le t)-\p(\un\tau_T\le s)
&\le \p(\un\tau_T\le t)-\p(\un\tau_T\le T/2)
	+ \p(\un\tau_T\le T/2)-\p(\un\tau_T\le s)\\
&\le C\gamma^{-1}(T/2)^{\gamma-1}[(T/2-s)^\gamma
	+ (t-T/2)^\gamma]
\le C\gamma^{-1}(T/2)^{2\gamma-2}(t-s)^\gamma.
\end{align*}
This gives part~(b) of the proposition. 
\end{proof}

%\begin{proof}[Proof of Corollary~\ref{cor:K-sup-tau}]
%In part (a) it suffices to prove the inequality for each argument 
%in the minimum. Applying~\eqref{eq:Lp-to-barrier} in 
%Lemma~\ref{lem:Lp-to-barrier} with $p\in\{1,2\}$ and the 
%bound~\eqref{eq:Lp-chi} in Theorem~\ref{thm:Wd-triplet} gives 
%these inequalities. Part (b) follows from~\eqref{eq:Lp-to-barrier} 
%in Lemma~\ref{lem:Lp-to-barrier} with $p=1$ and~\eqref{eq:mu_0} 
%\&~\eqref{eq:mu_tau} in Theorem~\ref{thm:Wd-triplet}. 
%\end{proof}

\subsection{Level variances under \nameref{alg:SBG}}
\label{subsec:couplings}

In the present subsection we establish bounds on the level variances 
under the coupling $\un\Pi_{n,T}^{\kappa_1,\kappa_2}$ 
(constructed in \nameref{alg:SBG}) for Lipschitz, locally 
Lipschitz and discontinuous payoff functions (see $\BT_1$ 
in~\eqref{def:BT1} and $\BT_2$ in~\eqref{def:BT2}) of $\un\chi_T$.

\begin{thm}
\label{thm:summary}
Fix $T>0$, $n\in\N$ and $1\geq\kappa_1>\kappa_2>0$. 
Denote $(Z_{n,T}^{(\kappa_i)},\un Z_{n,T}^{(\kappa_i)},
\un\tau_{n,T}^{(\kappa_i)})=\un\chi_{n,T}^{(\kappa_i)}$, 
$i\in\{1,2\}$, where the vector 
$\big(\un\chi_{n,T}^{(\kappa_1)},\un\chi_{n,T}^{(\kappa_2)}\big)$, 
constructed in \nameref{alg:SBG}, follows the law 
$\un\Pi_{n,T}^{\kappa_1,\kappa_2}$. 
\item[\nf{(a)}] For any Lipschitz function $f\in\Lip_K(\R^2)$, $K>0$, 
we have 
\begin{equation}
\label{eq:summary_2}
\E\big[\big(f\big(Z_{n,T}^{(\kappa_2)},
\un Z_{n,T}^{(\kappa_2)}\big)
-f\big(Z_{n,T}^{(\kappa_1)},
\un Z_{n,T}^{(\kappa_1)}\big)\big)^2\big] 
\le K^2T\big(27\sigma^2 2^{-n}+40\ov\sigma^2_{\kappa_1}\big).
\end{equation}
For $f\in\locLip_K(\R^2)$, defined in Subsection~\ref{subsec:bias} 
above, if $\int_{[1,\infty)}e^{4x}\nu(dx)<\infty$ then there exists a 
constant $C>0$ independent of $(n,\kappa_1,\kappa_2)$ such that  
\begin{align}
	\label{eq:summary_3}
	\E\big[\big(f\big(Z_{n,T}^{(\kappa_2)},
	\un Z_{n,T}^{(\kappa_2)}\big)
	-f\big(Z_{n,T}^{(\kappa_1)},
	\un Z_{n,T}^{(\kappa_1)}\big)\big)^2\big]
	&\leq C\big((2/3)^{n/2}\cdot\1_{\R\setminus\{0\}}(\sigma) 
	+ \ov\sigma^2_{\kappa_1} 
	+ \ov\sigma_{\kappa_1}\kappa_1\big).
\end{align}
\item[\nf{(b)}] Suppose Assumption~(\nameref{asm:(H)}) is satisfied 
by some $y<0$ and $C,\gamma>0$. Then for any $f\in\BT_1(y,K,M)$, 
$K,M\ge 0$, there exists some $K'>0$ independent of 
$(n,\kappa_1,\kappa_2)$ such that 
\begin{equation}
	\label{eq:summary_4}
	\begin{split}
		\E\big[\big(f\big(Z_{n,T}^{(\kappa_2)},
		\un Z_{n,T}^{(\kappa_2)}\big)
		-f\big(Z_{n,T}^{(\kappa_1)},
		\un Z_{n,T}^{(\kappa_1)}\big)\big)^2\big]
		&\le K'\big(\sigma^2 2^{-n}
		+ \ov\sigma^2_{\kappa_1}\big)^{\frac{\gamma}{2+\gamma}}.
	\end{split}
\end{equation}
\item[\nf{(c)}] If $\delta\in(0,2]$ satisfies 
Assumption~(\nameref{asm:(O)}), then there exists some $C>0$ 
such that for any $K>0$, $f\in\Lip_K(\R)$, $n\in\N$, 
$\kappa_1>\kappa_2$ and $p\in\{1,2\}$, we have 
\begin{equation}\label{eq:summary_5}
	\E\big[\big|f\big(\un\tau_{n,T}^{(\kappa_1)}\big)
	-f\big(\un\tau_{n,T}^{(\kappa_2)}\big)\big|^p\big]
	\le 2K^pT^p\big[2^{-n} 
	+ C\ov\sigma_{\kappa_1}^{\min\{\frac{2\delta}{2-\delta},
		\frac{1}{2}\}}
	\big(1+|\log\kappa_1|\cdot\1_{\{2/5\}}(\delta)\big)\big].
\end{equation}
\item[\nf{(d)}] Fix $s\in(0,T)$ and let 
Assumption~(\nameref{asm:(O)}) hold for some $\delta\in(0,2]$, 
then for any $f\in\BT_2(s,K,M)$, $K,M\ge 0$, there exists a constant 
$C>0$ such that for any $n\in\N$, $p\in\{1,2\}$ and 
$\kappa_1>\kappa_2$, we have 
\begin{equation}
	\label{eq:summary_6}
	\begin{split}
		\E\big[\big|f\big(\un\chi_{n,T}^{(\kappa_1)}\big)
		-f\big(\un\chi_{n,T}^{(\kappa_2)}\big)\big|^p\big]
		&\le C\big[2^{-n/2} 
		+ \ov\sigma_{\kappa_1}^{\min\{\frac{\delta}{2-\delta},
			\frac{1}{4}\}}
		\big(1+\tsqrt{|\log\kappa_1|}
		\cdot\1_{\{2/5\}}(\delta)\big)\big].
	\end{split}
\end{equation}
\end{thm}

The synchronous coupling of the large jumps of the Gaussian 
approximations, implicit in \nameref{alg:SBG}, ensures that 
no moment assumption on the large jumps of $X$ is necessary 
for~\eqref{eq:summary_2} to hold. For locally Lipschitz payoffs, 
however, the function may magnify the distance when a large 
jump occurs. This leads to the moment assumption 
$\int_{[1,\infty)}e^{4x}\nu(dx)<\infty$ for $f\in\locLip_K(\R^2)$. 

The proof of Theorem~\ref{thm:summary} requires bounds on 
certain moments of the differences of the components of the output 
of Algorithms~\ref{alg:ARA} \&~\ref{alg:ARA_2} and~\nameref{alg:SBG}, 
given in Proposition~\ref{prop:SBG_app_coupling}. 
%and~\ref{prop:SBG_app_coupling+} below.

\begin{prop}\label{prop:SBG_app_coupling}
%and Lipschitz function $f:\R^2\to\R$ with Lipschitz constant $K$. 
For any $1\geq\kappa_1>\kappa_2>0$, $t>0$ and 
$n\in\N$, the following statements hold.\\
\nf{(a)} The pair $\big(Z_t^{(\kappa_1)},Z_t^{(\kappa_2)}\big)
\sim\Pi_t^{\kappa_1,\kappa_2}$, constructed in 
Algorithm~\ref{alg:ARA}, satisfies the following inequalities 
\[
\E\big[\big(Z_t^{(\kappa_1)}-Z_t^{(\kappa_2)}\big)^2\big]
\leq 2(\ov\sigma^2_{\kappa_1}-\ov\sigma^2_{\kappa_2})t,
\qquad
\E\big[\big(Z_t^{(\kappa_1)}-Z_t^{(\kappa_2)}\big)^4\big]
\leq 12(\ov\sigma^2_{\kappa_1}-\ov\sigma^2_{\kappa_2})^2t^2
+ (\ov\sigma^2_{\kappa_1}-\ov\sigma^2_{\kappa_2})\kappa_1^2t.
\]
Moreover, we have 
$\E[(\un Z_t^{(\kappa_1)}-\un Z_t^{(\kappa_2)})^{2p}]
\le 4\E[(Z_t^{(\kappa_1)}-Z_t^{(\kappa_2)})^{2p}]$, 
for any $p\in\{1,2\}$.\\
\nf{(b)} The vector 
$\big(Z^{(\kappa_1)}_t,\un{Z}^{(\kappa_1)}_t,\un\tau_t^{(\kappa_1)},
Z_t^{(\kappa_2)},\un Z_t^{(\kappa_2)},\un\tau_t^{(\kappa_2)}\big)
\sim\un\Pi_t^{\kappa_1,\kappa_2}$, constructed in 
Algorithm~\ref{alg:ARA_2}, %in $\un\Pi_t^{\kappa_1,\kappa_2}$ 
satisfies the following inequalities 
\[
\E\big[\big(Z_t^{(\kappa_1)}-Z_t^{(\kappa_2)}\big)^2\big]
= 2(\sigma^2+\ov\sigma^2_{\kappa_1})t,
\qquad
\E\big[\big(Z_t^{(\kappa_1)}-Z_t^{(\kappa_2)}\big)^4\big]
\le 12(\sigma^2+\ov\sigma^2_{\kappa_1})^2t^2
+ (\ov\sigma^2_{\kappa_1}-\ov\sigma^2_{\kappa_2})\kappa_1^2t.
\]
Moreover, we have 
$\E[(\un{Z}_t^{(\kappa_1)}-\un Z_t^{(\kappa_2)})^{2p}]
\le 4\E[(Z_t^{(\kappa_1)}-Z_t^{(\kappa_2)})^{2p}]$, 
for any $p\in\{1,2\}$.\\
\nf{(c)} The coupling 
$\big(\un\chi^{(\kappa_1)}_{n,t},\un\chi^{(\kappa_2)}_{n,t}\big)
\sim\un\Pi_{n,t}^{\kappa_1,\kappa_2}$, constructed in 
\nameref{alg:SBG}, with components 
$\un\chi_{n,t}^{(\kappa_i)}=(Z_{n,t}^{(\kappa_i)},
\un Z_{n,t}^{(\kappa_i)},\un\tau_{n,t}^{(\kappa_i)})$, $i\in\{1,2\}$, 
satisfies the following inequalities: 
\begin{align}
\label{eq:SBG_1}
\E\big[\big(Z^{(\kappa_1)}_{n,t}
	-Z^{(\kappa_2)}_{n,t}\big)^2\big]
&\leq 2(\sigma^2 2^{-n} + \ov\sigma^2_{\kappa_1})t,\\
\label{eq:SBG_3}
\E\big[\big(Z^{(\kappa_1)}_{n,t}
	-Z^{(\kappa_2)}_{n,t}\big)^4\big]
&\leq (25\ov\sigma^4_{\kappa_1} 
	+ 24\sigma^4 3^{-n})t^2 
+ \ov\sigma^2_{\kappa_1}\kappa_1^2 t,\\
\label{eq:SBG_2}
\E\big[\big(\un Z^{(\kappa_1)}_{n,t}
	-\un Z^{(\kappa_2)}_{n,t}\big)^2\big]
&\leq (2+3\pi)(\sigma^2+\ov\sigma_{\kappa_1}^2) 2^{-n}t 
	+ (2+5\pi)\ov\sigma^2_{\kappa_1}t,\\
\label{eq:SBG_4}
\E\big[\big(\un Z^{(\kappa_1)}_{n,t}
	-\un Z^{(\kappa_2)}_{n,t}\big)^4\big]
&\leq 2\cdot 10^3\big[(\sigma^2+\ov\sigma_{\kappa_1})^2 3^{-n} 
	+\ov\sigma^4_{\kappa_1}\big]t^2
	+ 2\pi\ov\sigma^{5/2}_{\kappa_1}\kappa_1^{3/2}t^{5/4} 
	+ 4\ov\sigma^2_{\kappa_1}\kappa_1^2 t.
\end{align}
\end{prop}

\begin{rem}\label{rem:MLMC_CM}
(i) By Proposition~\ref{prop:SBG_app_coupling}, the $L^2$-norms 
of the differences $Z^{(\kappa_1)}_{n,t}-Z^{(\kappa_2)}_{n,t}$ and 
$\un Z^{(\kappa_1)}_{n,t}-\un Z^{(\kappa_2)}_{n,t}$ of the 
components of 
$(\un\chi_{n,t}^{(\kappa_1)},\un\chi_{n,t}^{(\kappa_2)})$, 
constructed in \nameref{alg:SBG}, decay at the same rate as 
the $L^2$-norm of $Z_t^{(\kappa_1)}-Z_t^{(\kappa_2)}$, constructed 
in Algorithm~\ref{alg:ARA}. 
Indeed, assume that $\kappa_1=c\kappa_2$ for some $c>1$, 
$\kappa_2\to0$ and, for some $c',r>0$ and all $x>0$, 
we have $\ov\nu(x)=\nu(\R\setminus(-x,x))\ge c'x^{-r}$. 
Then, for $n=\cl{\log^2(1+\ov\nu(\kappa_2))}$ we have 
$2^{-n}\le\ov\sigma_{\kappa_1}^2$ for all sufficiently small 
$\kappa_1$, implying the claim by 
Proposition~\ref{prop:SBG_app_coupling}(a) \&~(c). Moreover, 
by Corollary~\ref{cor:algs_complexities}, the corresponding 
expected computational complexities of Algorithm~\ref{alg:ARA} 
and~\nameref{alg:SBG} are proportional as $\kappa_2\to 0$. 
Furthermore, since the decay of the bias of \nameref{alg:SBG} 
is, by Theorem~\ref{thm:Wd-triplet}, at most a logarithmic factor 
away from that of Algorithm~\ref{alg:ARA}, the MLMC estimator 
based on Algorithm~\ref{alg:ARA} for $\E f(X_t)$ has the same 
computational complexity (up to logarithmic factors) as the 
MLMC estimator for $\E f(X_t,\un X_t)$ based on 
\nameref{alg:SBG} (see Table~\ref{tab:MLMC} above for the 
complexity of the latter). 

\noindent (ii) The proof of Proposition~\ref{prop:SBG_app_coupling} implies 
that an improvement in Algorithm~\ref{alg:ARA} (i.e. a simulation 
procedure for a coupling with a smaller $L^2$-norm of 
$Z_t^{(\kappa_1)}-Z_t^{(\kappa_2)}$) would result in an 
improvement in \nameref{alg:SBG} for the simulation of a 
coupling $(\un\chi_t^{(\kappa_1)},\un\chi_t^{(\kappa_2)})$. 
Interestingly, this holds in spite of the fact that 
\nameref{alg:SBG} calls Algorithm~\ref{alg:ARA_2} whose 
coupling $\un\Pi_t^{\kappa_1,\kappa_2}$ is inefficient in terms 
of the $L^2$-distance but is applied over the short interval $[0,L_n]$. 

\noindent (iii)  A nontrivial 
bound on the moments of the difference 
$\un\tau_t^{(\kappa_1)}-\un\tau_t^{(\kappa_2)}$ under the 
coupling of Algorithm~\ref{alg:ARA_2}, which would 
complete the statement in Proposition~\ref{prop:SBG_app_coupling}(b),
appears to be out of reach.
By the SB representation in~\eqref{eq:chi}, such a bound 
is not necessary for our purposes. 
The corresponding bound 
on the moments of the difference $\un\tau_{n,t}^{(\kappa_1)}-\un\tau_{n,t}^{(\kappa_2)}$,
constructed in \nameref{alg:SBG},
follows from Proposition~\ref{prop:SBG_app_coupling+} below, see~\eqref{eq:SBG_tau}.

\noindent (iv) The bounds on the fourth moments 
in~\eqref{eq:SBG_3} and~\eqref{eq:SBG_4} are required to 
control the level variances of the MLMC estimator in the case 
of locally Lipschitz payoff functions and are applied in the proof of 
Theorem~\ref{thm:summary}(a). 
\end{rem}

\begin{proof}[Proof of Proposition~\ref{prop:SBG_app_coupling}]
(a) The difference $Z^{(\kappa_1)}_t-Z^{(\kappa_2)}_t$ (constructed  by Algorithm~\ref{alg:ARA}) equals  
by~\eqref{eq:ARA} a sum of two independent martingales: 
$((\ov\sigma_{\kappa_1}^2+\sigma^2)^{1/2}
	-(\ov\sigma_{\kappa_2}^2+\sigma^2)^{1/2})W_t$ and 
$J_t^{2,\kappa_1}-J_t^{2,\kappa_2} + (b_{\kappa_1}-b_{\kappa_2})t$. 
Thus, we obtain the identity 
\[
\E\big[\big(Z_t^{(\kappa_1)}-Z_t^{(\kappa_2)}\big)^2\big]
=\Big[\Big(\sqrt{\sigma^2+\ov\sigma^2_{\kappa_1}} 
-\sqrt{\sigma^2+\ov\sigma^2_{\kappa_2}}\Big)^2
+\ov\sigma^2_{\kappa_1}-\ov\sigma^2_{\kappa_2}\Big]t.
\]
The first inequality follows since 
$0<(\sigma^2+\ov\sigma^2_{\kappa_1})^{1/2}
	-(\sigma^2+\ov\sigma^2_{\kappa_2})^{1/2}
\le(\ov\sigma^2_{\kappa_1}-\ov\sigma^2_{\kappa_2})^{1/2}$. 
Since $Z^{(\kappa_1)}_t-Z^{(\kappa_2)}_t$ is a L\'evy process,
differentiating its L\'evy-Khintchine formula in~\eqref{eq:Levy_Khinchin} 
yields the identity
\begin{align*}
\E\big[\big(Z_t^{(\kappa_1)}-Z_t^{(\kappa_2)}\big)^4\big]
&=3\Big[\Big(\sqrt{\sigma^2+\ov\sigma^2_{\kappa_1}} 
-\sqrt{\sigma^2+\ov\sigma^2_{\kappa_2}}\Big)^2
+\ov\sigma^2_{\kappa_1}-\ov\sigma^2_{\kappa_2}\Big]^2t^2
+t\int_{(-\kappa_1,\kappa_1)\setminus(-\kappa_2,\kappa_2)}x^4\nu(dx),
\end{align*}
which implies the second inequality. 
Since $|\un Z_t^{(\kappa_1)}-\un Z_t^{(\kappa_2)}|
\le\sup_{s\in[0,t]}|Z_s^{(\kappa_1)}-Z_s^{(\kappa_2)}|$, 
Doob's maximal martingale 
inequality~\cite[Prop.~7.16]{MR1876169} applied to the martingale 
$(Z_s^{(\kappa_1)}-Z_s^{(\kappa_2)})_{s\in[0,t]}$ yields 
\[
\E\big[\big|\un Z_t^{(\kappa_1)}-\un Z_t^{(\kappa_2)}\big|^p\big]
\le \big(1-1/p\big)^{-p}
\E\big[\big| Z_t^{(\kappa_1)}- Z_t^{(\kappa_2)}\big|^p\big],
	\qquad p>1.
\]
The corresponding inequalities follow because $(p/(p-1))^p\le 4$ 
for $p\in\{2,4\}$. 

(b) Analogous to part (a), the difference 
$Z_t^{(\kappa_1)}-Z_t^{(\kappa_2)}$ constructed in 
Algorithm~\ref{alg:ARA_2} is a sum of two independent martingales: 
$(\ov\sigma_{\kappa_1}^2+\sigma^2)^{1/2}B_t
	-(\ov\sigma_{\kappa_2}^2+\sigma^2)^{1/2}W_t$ and 
$J_t^{2,\kappa_1}-J_t^{2,\kappa_2} + (b_{\kappa_1}-b_{\kappa_2})t$, 
where $B$ and $W$ are independent standard Brownian motions. 
Thus the statements follow as in part (a). 

(c) %Let $\big(\xi_{k,1},\xi_{k,2}\big)\sim \Pi_{\ell_k}^{\kappa_1,\kappa_2}$ for $k\in\{1,\ldots,n\}$ and $\big(\un\xi_1,\un\xi_2) \sim\un\Pi_{L_n}^{\kappa_1,\kappa_2}$ 
	Let $(\xi_{1,k},\xi_{2,k})\sim \Pi_{\ell_k}^{\kappa_1,\kappa_2}$, $k\in\{1,\ldots,n\}$, and 
$(\un\zeta_1,\un\zeta_2)\sim\un\Pi_{L_n}^{\kappa_1,\kappa_2}$ 
	be independent draws as in line~\ref{alg_line:samplingFrom_A1_A2} of
	\nameref{alg:SBG} above. 
Denote by $(\xi_{i,n+1},\un\xi_{i,n+1})$ the first two coordinates 
of $\un\zeta_i$, $i\in\{1,2\}$. Since the variables 
$\{\xi_{1,k}-\xi_{2,k}\}_{k=1}^{n+1}$ 
	%in the construction of 
%$(\un\chi_{n,t}^{(\kappa_1)},\un\chi_{n,t}^{(\kappa_2)})$ 
have zero mean and are uncorrelated, by conditioning on 
$\{\ell_k\}_{k=1}^{n}$ and $L_n$ and applying parts~(a) and~(b) 
we obtain 
\begin{align*}
\E\big[\big(Z_{n,t}^{(\kappa_1)}-Z_{n,t}^{(\kappa_2)}\big)^2\big]
&=\V\big[Z_{n,t}^{(\kappa_1)}-Z_{n,t}^{(\kappa_2)}\big]
	=\V\big[\xi_{1,n+1}-\xi_{2,n+1}\big]
	+\sum_{k=1}^{n}\V\big[\xi_{1,k}-\xi_{2,k}\big]\\
&\leq 2(\sigma^2+\ov\sigma^2_{\kappa_1})\E[L_{n}]
	+2\ov\sigma^2_{\kappa_1}\sum_{k=1}^{n}\E[\ell_k]=
2(\sigma^2+\ov\sigma^2_{\kappa_1})2^{-n}t
	+2\ov\sigma^2_{\kappa_1}(1-2^{-n}) t.
\end{align*}
implying~\eqref{eq:SBG_1}.
%We thus deduce the following inequality, 
%which itself implies~\eqref{eq:SBG_1}: 
%\[
%\E\big[\big(Z_{n,t}^{(\kappa_1)}
%	-Z_{n,t}^{(\kappa_2)}\big)^2\big]
%\le 
%2(\sigma^2+\ov\sigma^2_{\kappa_1})2^{-n}t
%	+2\ov\sigma^2_{\kappa_1}(1-2^{-n}) t.
%\]
Similarly, by conditioning on $\{\ell_k\}_{k=1}^{n}$ and $L_{n}$, 
we deduce that the expectations of 
\[
(\xi_{1,k_1}-\xi_{2,k_1})^3(\xi_{1,k_2}-\xi_{2,k_2}),
\qquad
(\xi_{1,k_1}-\xi_{2,k_1})^2\prod_{i=2}^3 (\xi_{1,k_i}-\xi_{2,k_i}),
\qquad\text{and}\qquad
\prod_{i=1}^4 (\xi_{1,k_i}-\xi_{2,k_i}),
\]
vanish for any distinct $k_1,k_2,k_3,k_4\in\{1,\ldots,n+1\}$. 
Thus, by expanding, we obtain 
\begin{align*}
\E\big[\big(Z_{n,t}^{(\kappa_1)}-Z_{n,t}^{(\kappa_2)}\big)^4\big]
&=\sum_{k=1}^{n+1} \E\big[\big(\xi_{1,k}-\xi_{2,k}\big)^4\big]
	+6\sum_{m=1}^{n}\sum_{k=m+1}^{n+1} 
		\E\big[\big(\xi_{1,m}-\xi_{2,m}\big)^2
		\big(\xi_{1,k}-\xi_{2,k}\big)^2\big].
\end{align*}
The summands in the first sum are easily bounded by parts (a) and 
(b). To bound the summands of the second sum, condition on 
$\{\ell_k\}_{k=1}^n$ and $L_n$ and apply parts (a) and (b): 
\[
\E\big[\big(\xi_{1,k}-\xi_{2,k}\big)^2
	\big(\xi_{1,m}-\xi_{2,m}\big)^2\big]
\le\begin{cases}
4\ov\sigma^4_{\kappa_1}\E[\ell_m\ell_k], 
& m<k\le n,\\ 
4(\sigma^2+\ov\sigma^2_{\kappa_1})
	\ov\sigma^2_{\kappa_1}\E[\ell_mL_n], 
& m<k=n+1.
\end{cases}
\]
Inequality~\eqref{eq:SBG_3} follows since 
$\E[\ell_m\ell_k]=3^{-m}2^{m-k-1}t^2$, 
$\E[\ell_kL_{n}]=3^{-k}2^{k-n-1}t^2$ for $m<k\le n$ and 
$\sigma^22^{-n}\ov\sigma_\kappa^2
\le \sigma^23^{-n/2}\ov\sigma_\kappa^2
\le (\sigma^4 3^{-n}+\ov\sigma_\kappa^4)/2$.

The representation in line~\ref{alg_line:min_in_sum} of \nameref{alg:SBG}
and the elementary inequality $|\min\{a,0\}-\min\{b,0\}|\leq|a-b|$ 
(for all $a,b\in\R$) imply 
\begin{equation}
\label{eq:Z_nt}
\begin{split}
\E\big[\big(\un Z_{n,t}^{(\kappa_1)}
	-\un Z_{n,t}^{(\kappa_2)}\big)^2\big]
& \le\E\bigg[\big(\un\xi_{1,n+1}-\un\xi_{2,n+1}\big)^2
	+ \sum_{k=1}^n\big(\xi_{1,k}-\xi_{2,k}\big)^2\bigg]\\
&\hspace{-15mm}+ 2\E\bigg[\sum_{k=1}^n
	\big|\un\xi_{1,n+1}-\un\xi_{2,n+1}
		\big|\big|\xi_{1,k}-\xi_{2,k}\big| 
	+ \sum_{m=1}^{n-1}\sum_{k=m+1}^n
		\big|\xi_{1,m}-\xi_{2,m}\big|
			\big|\xi_{1,k}-\xi_{2,k}\big|\bigg].
\end{split}
\end{equation}
The first term on the right-hand side of this inequality is easily 
bounded via the inequalities in parts (a) and (b). To bound 
the second term, condition on $\{\ell_k\}_{k=1}^n$ and 
$L_n$, apply the Cauchy-Schwarz inequality, denote 
$\upsilon:=\sqrt{\sigma^2+\ov\sigma_{\kappa_1}^2}$ and 
observe that for $m<k\le n$ we get 
\begin{align*}
\E\big[\big|\un\xi_{1,n+1}-\un\xi_{2,n+1}\big|
	\big|\xi_{1,k}-\xi_{2,k}\big|\big]
&\leq\E\Big[\tsqrt{16(\sigma^2+\ov\sigma^2_{\kappa_1})
	\ov\sigma^2_{\kappa_1}\ell_k L_n}\Big]
=\pi \upsilon\ov\sigma_{\kappa_1}(2/3)^{n}(3/4)^kt,\\
	\E\big[\big|\xi_{1,m}-\xi_{2,m}\big|
		\big|\xi_{1,k}-\xi_{2,k}\big|\big]
&\le\E\Big[\tsqrt{4\ov\sigma^4_{\kappa_1}\ell_m \ell_k}\Big]
=\pi\ov\sigma^2_{\kappa_1}(1/2)^{m+1}(2/3)^{k-m}t,
\end{align*}
where the equalities follow from the definition of the stick-breaking 
process (see Subsection~\ref{subsec:SB-rep}). 
By~\eqref{eq:Z_nt} we have 
\begin{align*}
\E\big[\big(\un Z_{n,t}^{(\kappa_1)}
-\un Z_{n,t}^{(\kappa_2)}\big)^2\big]
& \le \upsilon^22^{1-n}t
	+ 2\ov\sigma_{\kappa_1}^2 t\sum_{k=1}^\infty 2^{-k}
	+ 2\pi\upsilon\ov\sigma_{\kappa_1}\big(\tfrac{2}{3}\big)^n t
		\sum_{k=1}^\infty\big(\tfrac{3}{4}\big)^k
	+ \pi\ov\sigma_{\kappa_1}^2 t
		\sum_{m=1}^{\infty}\sum_{k=1}^{\infty}
			2^{-m}\big(\tfrac{2}{3}\big)^k,
\end{align*}
so~\eqref{eq:SBG_2} follows from the inequalities 
$v(2/3)^{n}\ov\sigma_\kappa
\le \upsilon2^{-n/2}\ov\sigma_\kappa
\le (\upsilon^2 2^{-n}+\ov\sigma_\kappa^2)/2$.

As before, 
$|\min\{a,0\}-\min\{b,0\}|\leq|a-b|$ for $a,b\in\R$, 
%the inequality $|a^--b^-|\leq|a-b|$ for $a,b\in\R$, 
yields the inequality 
\begin{equation}\label{eq:4th-inf}
\E\big[\big(\un Z_{n,t}^{(\kappa_1)}
	-\un Z_{n,t}^{(\kappa_2)}\big)^4\big]
\le \E\bigg[\bigg(\big|\un\xi_{1,n+1}-\un\xi_{2,n+1}\big|
	+\sum_{k=1}^n\big|\xi_{1,k}-\xi_{2,k}\big|\bigg)^4\bigg].
\end{equation}
By Jensen's inequality, $\E[|\vartheta|^3]\le\E[\vartheta^4]^{3/4}$ 
and $\E[\vartheta]\le\sqrt{\E[\vartheta^2]}$ for any random 
variable $\vartheta$. Hence, we may bound the first and third 
conditional moments of $|\xi_{1,k}-\xi_{2,k}|$ and 
$|\un\xi_{1,n+1}-\un\xi_{2,n+1}|$ given $\{\ell_k\}_{k=1}^n$ and 
$L_n$. Thus, by expanding~\eqref{eq:4th-inf}, conditioning on 
$\{\ell_k\}_{k=1}^n$ and $L_n$, and using elementary estimates as 
in all the previously developed bounds, we obtain~\eqref{eq:SBG_4}. 
\end{proof}

In order to control the level variances of the MLMC estimator 
in~\eqref{eq:MLMC} for discontinuous payoffs of $\un\chi_t$ and 
functions of $\un \tau_t$, we would need to apply 
Lemma~\ref{lem:Lp-to-barrier} to the components of 
$\big(\un\chi^{(\kappa_1)}_{n,t},\un\chi^{(\kappa_2)}_{n,t}\big)$
constructed in \nameref{alg:SBG}. In particular, 
the assumption in Lemma~\ref{lem:Lp-to-barrier} requires a control 
on the constants in the locally Lipschitz property of the distribution 
functions of the various components of 
$\big(\un\chi^{(\kappa_1)}_{n,t},\un\chi^{(\kappa_2)}_{n,t}\big)$
in terms of the cutoff levels $\kappa_1$ and $\kappa_2$. As such 
a uniform bound in the cutoff level appears to be out of reach,
we establish Proposition~\ref{prop:SBG_app_coupling+},
which allows us to compare the sampled quantities 
$\un\chi^{(\kappa_1)}_{n,t}$ and $\un\chi^{(\kappa_2)}_{n,t}$
with their limit $\un\chi_t$ (as $\kappa_1,\kappa_2\to0$). Since, 
under mild assumptions, the distribution functions of the 
components of the limit $\un\chi_t$ possess the necessary 
regularity and do not depend on the cutoff level, the application 
of Lemma~\ref{lem:Lp-to-barrier} in the proof of 
Theorem~\ref{thm:summary} becomes feasible using 
Proposition~\ref{prop:SBG_app_coupling+}.

%The control of the fourth moments is 
%required to describe the level variances of locally Lipschitz functions. 
%Proposition~\ref{prop:SBG_app_coupling} is not enough to describe 
%the level variances of functions of $\un\tau_t$ or discontinuous 
%functions. Indeed, in the application of 
%Lemma~\ref{lem:Lp-to-barrier}, it is necessary to have sufficient 
%control on the distribution function of one of the variables 
%(which we have obtained in other contexts via 
%Assumption~(\nameref{asm:(H)}) and~(\nameref{asm:(Htau)})). 
%However, it is not clear how such regularity (say, the constants in 
%the locally Lipschitz property of a distribution function) depends on 
%the cutoff levels $\kappa_1$ and $\kappa_2$. To circumvent this 
%issue, we require the establish the next result, which jointly couples 
%both approximations output by \nameref{alg:SBG} with the 
%corresponding limit $\un\chi_t$. 

\begin{prop}
\label{prop:SBG_app_coupling+}
There exists a coupling between $\un\chi_t=(X_t,\un X_t,\un\tau_t)$ 
and $\big(\un\chi^{(\kappa_1)}_{n,t},\un\chi^{(\kappa_2)}_{n,t}\big)
\sim\un\Pi_{n,t}^{\kappa_1,\kappa_2}$ such that for any 
$i\in\{1,2\}$ and $p\ge 1$, the vector $(Z_{n,t}^{(\kappa_i)},
\un Z_{n,t}^{(\kappa_i)},\un\tau_{n,t}^{(\kappa_i)})
=\un\chi_{n,t}^{(\kappa_i)}$ satisfies 
\begin{align}
\label{eq:SBG+_1}
\E\big[\big(X_t - Z^{(\kappa_i)}_{n,t}\big)^2\big]
&\le (4\sigma^2 2^{-n} \cdot\1_{\{1\}}(i) 
	+ 2\ov\sigma^2_{\kappa_i})t,\\
\label{eq:SBG+_2}
\E\big[\big(\un X_t - \un Z^{(\kappa_i)}_{n,t}\big)^2\big]
&\le (48\sigma^2 2^{-n} \cdot\1_{\{1\}}(i) 
	+ 42\ov\sigma^2_{\kappa_i})t.
\end{align}
Moreover, if $\delta\in(0,2]$ satisfies 
Assumption~(\nameref{asm:(O)}), we have 
\begin{equation}
\label{eq:SBG+_3}
\E\big[\big|\un\tau_t-\un\tau^{(\kappa_i)}_{n,t}\big|^p\big]
\le 2^{-n}t^p + t^{p-1}\theta(t,\kappa_i),
\end{equation}
where, given $T\ge t$, there exists a constant $C>0$ dependent only 
on $(T,\sigma^2,\nu,b)$ such that for all $\kappa\in(0,1]$, 
the function $\theta(t,\kappa)$ is defined as 
\[
\theta(t,\kappa)
:=\begin{cases}
	\min\{1,\sqrt{C\ov\sigma_{\kappa}}\}t,
	&\delta=2,\\
	%\begin{array}r
	\min\{t,(C\ov\sigma_{\kappa})^{\frac{2\delta}{2-\delta}}\}
	+\tfrac{4\delta }{5\delta-2}\sqrt{C\ov\sigma_{\kappa}}
	\big(t^{\frac{5\delta-2}{4\delta}}
	-\min\{t,(C\ov\sigma_{\kappa})^{\frac{4\delta}{4-2\delta}}\}
	^{\frac{5\delta-2}{4\delta}}\big),%\\
	%\end{array}
	&\delta\in(0,2)\setminus\{\tfrac{2}{5}\},\\
	\min\{t,\sqrt{C\ov\sigma_{\kappa}}\}+\sqrt{C\ov\sigma_{\kappa}}
	\log^+\big(t/\sqrt{C\ov\sigma_{\kappa}}\big),
	&\delta=\tfrac{2}{5}.
\end{cases}
\]
\end{prop}

As a simple consequence of~\eqref{eq:SBG+_3} (with $p=1$) in 
Proposition~\ref{prop:SBG_app_coupling+} and the elementary 
inequality $|\un\tau^{(\kappa_1)}_{n,t}
	-\un\tau^{(\kappa_2)}_{n,t}|\le t$, we deduce that the coupling in 
\nameref{alg:SBG} satisfies 
\begin{equation}
\label{eq:SBG_tau}
\E\big[\big|\un\tau^{(\kappa_1)}_{n,t}
-\un\tau^{(\kappa_2)}_{n,t}\big|^p\big]
\le 2^{1-n}t^p + 2t^{p-1}\theta(t,\kappa_1),
\quad\text{for any }p\ge1.
\end{equation}
The bounds in~\eqref{eq:SBG+_1} and~\eqref{eq:SBG+_2} of Proposition~\ref{prop:SBG_app_coupling+} imply the inequalities  
in~\eqref{eq:SBG_1} and~\eqref{eq:SBG_2} 
of Proposition~\ref{prop:SBG_app_coupling}(c) with slightly worse 
constants.

\begin{proof}
The proof and construction of the random variables is analogous to 
that of Proposition~\ref{prop:SBG_app_coupling}(c), where, for 
$i\in\{1,2\}$, we compare the increment $Z_s^{(\kappa_i)}$ defined 
in Algorithm~\ref{alg:ARA} with the L\'evy-It\^o decomposition 
$X_s = bs+\sigma W_s + J^{1,\kappa_i}_s + J^{2,\kappa_i}_s$
($W$ is as in Algorithm~\ref{alg:ARA}, independent of 
$J^{1,\kappa_i}$ and $J^{2,\kappa_i}$) over the time horizons 
$s\in\{\ell_1,\ldots,\ell_{n-1}\}$. Similarly, we compare the pair of 
vectors $(\un\chi_s^{(\kappa_1)},\un\chi_s^{(\kappa_2)})$ output 
by Algorithm~\ref{alg:ARA_2} with $\un \chi_s$ for $s=L_n$, where 
we assume that the (standardised) Brownian component of $X$ 
equals that of $\un\chi_s^{(\kappa_2)}$ (and is thus independent 
of the one in $\un\chi_s^{(\kappa_1)}$) and all jumps in 
$J^{2,\kappa_2}$ are synchronously coupled. 

Denote the first and fourth components of the vector 
$(\un\chi_s^{(\kappa_1)},\un\chi_s^{(\kappa_2)})$ by 
$Z_s^{(\kappa_1)}$ and $Z_s^{(\kappa_2)}$, respectively.
Hence, it is enough to obtain the analogous bounds and identities 
to those presented in parts (a) and (b) for the expectations 
$\E[(X_t-Z_t^{(\kappa_i)})^2]$, $i\in\{1,2\}$ under both 
couplings: $\Pi_t^{\kappa_1,\kappa_2}$ and 
$\un\Pi_t^{\kappa_1,\kappa_2}$. Such bounds may be obtained 
using the proofs of parts (a) and (b), resulting in the following: 
for $i\in\{1,2\}$, we have 
\begin{align}
\label{eq:Xt-Zt}
\E\big[\big(X_t-Z_t^{(\kappa_i)}\big)^2\big]
&=\Big[\Big(\sqrt{\sigma^2+\ov\sigma^2_{\kappa_i}}-\sigma\Big)^2
	+\ov\sigma^2_{\kappa_i}\Big]t\leq 2\ov\sigma^2_{\kappa_i}t,
	&\text{under }\Pi_t^{\kappa_1,\kappa_2},\\
\nonumber
\E\big[\big(X_t-Z_t^{(\kappa_i)}\big)^2\big]
&=2(\sigma^2\cdot\1_{\{1\}}(i)+\ov\sigma^2_{\kappa_1})t,
	&\text{under }\un\Pi_t^{\kappa_1,\kappa_2}.
\end{align}
Thus Doob's martingale inequality and elementary inequalities 
give~\eqref{eq:SBG+_1} and~\eqref{eq:SBG+_2}. 

By the construction of the law $\un\Pi_{n,t}^{\kappa_1,\kappa_2}$ 
in \nameref{alg:SBG}, there exist random 
variables $(\xi'_k)_{k=1}^{n}$ such that for $k\in\{1,\ldots,n\}$, 
conditional on $\ell_k=s$ and independently of $\{\ell_j\}_{j\ne k}$, 
the distributional equality $(\xi'_k,\xi_{1,k},\xi_{2,k})
\eqd (X_s,Z_s^{(\kappa_1)},Z_s^{(\kappa_2)})$ holds, where 
$(Z_t^{(\kappa_1)},Z_t^{(\kappa_2)})\sim\Pi^{\kappa_1,\kappa_2}_t$ 
and~$W$ in Algorithm~\ref{alg:ARA} equals the Brownian component 
of~$X$ in~\eqref{eq:levy-ito}. Note that by~\eqref{eq:chi} we have 
\begin{equation}
\label{eq:dif_tau_kappa}
\big|\un\tau_t-\un\tau^{(\kappa_i)}_{n,t}\big|
\le L_n+\sum_{k=1}^n
\ell_k\big|\1\{\xi_k'<0\}-\1\{\xi_{i,k}<0\}\big|,
\quad\text{for }i\in\{1,2\}.
\end{equation}
%In particular, we may replace $n$ with any $m\in\{0,\ldots,n\}$ 
%on the right-hand side of~\eqref{eq:dif_tau_kappa}. 

Let $\delta\in(0,2]$ be as in the statement of the proposition. 
By~\cite[Thm~3.1(a)]{MR1449834}, as in the proof of 
Theorem~\ref{thm:Wd-triplet}, we know that the density $f_t$ of 
$X_t$ exists, is smooth and, given $T>0$, the constant 
$C':=2^{3/2}\sup_{(s,x)\in(0,T]\times\R}s^{1/\delta}f_s(x)$ is 
finite. Thus,~\eqref{eq:Lp-to-barrier} in Lemma~\ref{lem:Lp-to-barrier}
(with constants $\gamma=1$ \& $C=2^{-3/2}\ell_k^{-1/\delta}C'$ 
and $M=1$, $K=0$ \& $p=1$) gives 
\begin{align*}
\E\big[\big|\1{\{\xi'_k<0\}}-\1{\{\xi_{i,k}<0\}}\big|\big|\ell_k\big]
&\le\min\big\{1,2^{-1/4}\sqrt{C'}\ell_k^{-\frac{1}{2\delta}}
	\E\big[|\xi'_k-\xi_{i,k}|\big|\ell_k\big]^{1/2}\big\}\\
&\le \min\big\{1,2^{-1/4}\sqrt{C'}\ell_k^{-\frac{1}{2\delta}}
	(2\ov\sigma_{\kappa_i}^2\ell_k)^{1/4}\big\},
\end{align*}
for any $i\in\{1,2\}$ and $k\in\{1,\ldots,n\}$, where the second 
inequality follows from Jensen's inequality and~\eqref{eq:Xt-Zt}. 
Hence, elementary inequalities,~\eqref{eq:dif_tau_kappa} and 
Lemma~\ref{lem:mom_ell} imply the following: for $i\in\{1,2\}$, 
\begin{align*}
\E|\un\tau_t-\un\tau^{(\kappa_i)}_{n,t}|
&\le \E L_{n} + \sum_{k=1}^n
	\E\big[\ell_k \big|\1\{\xi'_k<0\} - \1\{\xi_{i,k}<0\}\big|\big]\\
&\le 2^{-n}t + \sum_{k=1}^\infty 
	\E\big[\min\big\{\sqrt{C'\ov\sigma_{\kappa_i}}
		\ell_k^{\frac{5}{4}-\frac{1}{2\delta}},\ell_k\big\}\big]
\le 2^{-n}t + \theta(t,\kappa_i).
\end{align*}
For $p>1$, the result follows from the case $p=1$ and 
the inequality $|\un\tau_t-\un\tau_{n,t}^{(\kappa_i)}|^p
\le t^{p-1}|\un\tau_t-\un\tau_{n,t}^{(\kappa_i)}|$. 
\end{proof}

\begin{proof}[Proof of Theorem~\ref{thm:summary}]
\label{proof:Thm_bias-cost-var}
(a) Proposition~\ref{prop:SBG_app_coupling}(c) and elementary 
inequalities yield~\eqref{eq:summary_2}, so it remains to 
consider the case $f\in\locLip_K(\R^2)$. As in the proof of 
Proposition~\ref{prop:logLip}, by the inequality in~\eqref{eq:loclip} 
and the Cauchy-Schwarz inequality, we have 
\begin{align*}
\E\big[\big(f(Z_{n,T}^{(\kappa_1)},\un Z_{n,T}^{(\kappa_1)})
	-f(Z_{n,T}^{(\kappa_2)},\un Z_{n,T}^{(\kappa_2)})\big)^2\big]^2
&\le K^4K'\E\big[(|Z_{n,T}^{(\kappa_1)}-Z_{n,T}^{(\kappa_2)}|
	+ |\un Z_{n,T}^{(\kappa_1)}-\un Z_{n,T}^{(\kappa_2)}|)^4\big],
\end{align*}
where 
$K':=\E[(\exp(Z_{n,T}^{(\kappa_1)})+\exp(Z_{n,T}^{(\kappa_2)}))^4]
\le 8\E[\exp(4X_T^{(\kappa_1)})+\exp(4X_T^{(\kappa_2)})\big]$. 
Applying~\eqref{eq:exp_bound}, we get $\E[\exp(4X_T^{(\kappa_i)})]
	\le\E[\exp(4X_T)]\exp(4T\ov\sigma^2_{\kappa_i})$ and 
$\ov\sigma^2_{\kappa_i}\le \ov\sigma^2_1$, $i\in\{1,2\}$, 
where $\E[\exp(4X_T)]$ is finite since 
$\int_{[1,\infty)}e^{4x}\nu(dx)<\infty$.  
The concavity of $x\mapsto\sqrt{x}$ and 
Inequalities~\eqref{eq:SBG_3} \&~\eqref{eq:SBG_4} 
Proposition~\ref{prop:SBG_app_coupling}(c) imply the existence 
of a constant $C>0$ satisfying 
\[
\tsqrt{\E\big[(|Z_{n,T}^{(\kappa_1)}-Z_{n,T}^{(\kappa_2)}|
	+ |\un Z_{n,T}^{(\kappa_1)}-\un Z_{n,T}^{(\kappa_2)}|)^4\big]}
\leq C(2/3)^{n/2} + 11T\ov\sigma^2_{\kappa_1} 
	+ \sqrt{2\pi}T^{5/8}\ov\sigma^{5/4}_{\kappa_1}\kappa_1^{3/4}
	+ \sqrt{5T}\ov\sigma_{\kappa_1}\kappa_1.
\]
Inequality~\eqref{eq:summary_3} then follows from the fact that 
$\ov\sigma_{\kappa_1}^{1/4}\kappa_1^{3/4}
\le\max\{\ov\sigma_{\kappa_1},\kappa_1\}
\le\ov\sigma_{\kappa_1}+\kappa_1$.

(b) Let 
$(\un\chi_T,\un\chi_{n,T}^{(\kappa_1)},\un\chi_{n,T}^{(\kappa_2)})$ 
be coupled as in Proposition~\ref{prop:SBG_app_coupling+}, where 
$\un\chi_T=(X_T,\un X_T,\un\tau_T)$ and  
$\un\chi_{n,T}^{(\kappa_i)}=(Z_{n,T}^{(\kappa_i)},
\un Z_{n,T}^{(\kappa_i)},\un\tau_{n,T}^{(\kappa_i)})$, $i\in\{1,2\}$. 
The triangle inequality and the inequalities $0\leq f\leq M$ give 
\begin{align*}
\E\big[\big(f\big(Z_{n,T}^{(\kappa_1)},
		\un Z_{n,T}^{(\kappa_1)}\big)
	-f\big(Z_{n,T}^{(\kappa_2)},
		\un Z_{n,T}^{(\kappa_2)}\big)\big)^2\big]
&\le M\E\big|f\big(Z_{n,T}^{(\kappa_1)},
		\un Z_{n,T}^{(\kappa_1)}\big)
	-f\big(Z_{n,T}^{(\kappa_2)},
		\un Z_{n,T}^{(\kappa_2)}\big)\big|\\
&\leq M\sum_{i=1}^2
	\E\big|f\big(Z_{n,T}^{(\kappa_i)},
		\un Z_{n,T}^{(\kappa_i)}\big)
	-f\big(X_T,\un X_T\big)\big|.
\end{align*}
Apply~\eqref{eq:Lp-to-barrier} in Lemma~\ref{lem:Lp-to-barrier} 
with $C$ and $\gamma$ from Assumption~(\nameref{asm:(H)}) to 
$(X_T,\un X_T)$ and 
$\big(Z^{(\kappa_i)}_{n,T},\un Z^{(\kappa_i)}_{n,T}\big)$ 
to~get 
\begin{align*}
\E\big|f(X_T,\un X_T)
	-f\big(Z^{(\kappa_i)}_{n,T},\un Z^{(\kappa_i)}_{n,T}\big)\big|
&\le K\E\big[\big|Z^{(\kappa_i)}_{n,T}-X_T\big|\big]
	+ M(1+2/\gamma)
	(C^2\gamma^2\E\big[\big|\un Z_{n,T}^{(\kappa_1)}
		-\un X_T\big|^2\big]^\gamma)^{\frac{1}{2+\gamma}}\\
&\le K\sqrt{T(4\sigma^2 2^{-n}\cdot\1_{\{1\}}(i) 
		+ 2\ov\sigma^2_{\kappa_i})}
	+ K''(\sigma^2 2^{-n}\cdot\1_{\{1\}}(i) 
		+ \ov\sigma^2_{\kappa_i})^{\frac{\gamma}{2+\gamma}},
\end{align*} 
for any $i\in\{1,2\}$, where $K'':=M(1+2/\gamma)
(48C^2\gamma^2T^\gamma)^{1/(2+\gamma)}$. In the second 
inequality we used the bounds~\eqref{eq:SBG+_1} 
\&~\eqref{eq:SBG+_2} in Proposition~\ref{prop:SBG_app_coupling+}. 
Since $\ov\sigma_{\kappa_1}\ge \ov\sigma_{\kappa_2}$, 
the result follows.

(c) Recall that the inequality in~\eqref{eq:SBG_tau} follows 
from~\eqref{eq:SBG+_3} 
of Proposition~\ref{prop:SBG_app_coupling+}. The inequality 
in~\eqref{eq:summary_5} in the proposition is a direct consequence 
of the Lipschitz property and~\eqref{eq:SBG_tau}.

(d) The proof follows along the same lines as in part (b): we 
apply~\eqref{eq:Lp-to-barrier} in Lemma~\ref{lem:Lp-to-barrier} 
with $C$ and $\gamma$ from Assumption~(\nameref{asm:(Htau)}) 
and bounds~\eqref{eq:SBG+_1}--\eqref{eq:SBG+_3} in 
Proposition~\ref{prop:SBG_app_coupling+}.
%The inequality~\eqref{eq:summary_6} follows by applying the elementary 
%inequality $\sqrt{a+b}\le\sqrt{a}+\sqrt{b}$ and Lemma~\ref{lem:Lp-to-barrier} 
%for every $s$ for which the density of $\tau_T$ exists. Since $0$ is regular for 
%both half-lines, the density exists for almost every $s$ by 
%assumption,~\cite[Thm~3.1]{MR1449834} and~\cite[Thm~3]{MR3098676}. 
%This concludes the proof.
\end{proof}

\subsection{MC and MLMC estimators}
\label{subsec:MC_MLMC}

In the present subsection we address the application of our previous 
results to estimate the expectation $\E[f(\un\chi_T)]$ for various 
real-valued functions $f$ satisfying $\E[f(\un\chi_T)^2]<\infty$. 
By definition, an estimator 
$\Upsilon$ of $\E[f(\un\chi_T)]$ has $L^2$-accuracy of level 
$\epsilon>0$ if $\E[(\Upsilon-\E f(\un\chi_T))^2]<\epsilon^2$. 
We assume in this subsection that $X$ has jumps of infinite activity, i.e. $\nu(\R\setminus\{0\})=\infty$.
If the jumps of $X$ are finite activity, 
% $\nu(\R\setminus\{0\})<\infty$, 
both Algorithm~\ref{alg:ARA_2} and~\nameref{alg:SBG} are exact with the 
latter outperforming the former in practice by a constant factor, 
which is a function of the total number of jumps 
$T\nu(\R\setminus\{0\})<\infty$, see 
Subsection~\ref{subsec:CP_example} for a numerical example.

%We begin by introducing the MC estimator. 
\subsubsection{MC estimator} 
Pick %$N\in\N$ and 
$\kappa\in(0,1]$ and let the sequence $\un\chi_T^{\kappa,i}$, 
$i\in\N$, be iid with the same distribution as 
$\un\chi_T^{(\kappa)}$ simulated by \nameref{alg:SBG} 
with $n\in\N\cup\{0\}$ sticks. Note that the choice of $n$ 
does not affect the asymptotic behaviour as $\epsilon\searrow 0$ 
of the computational complexity $\C_\MC(\epsilon)$. The MC 
estimator based on $N\in\N$ independent samples is given by 
\begin{equation}\label{eq:MC}
\Upsilon_\MC:=\frac{1}{N}\sum_{i=1}^N f\big(\un\chi_T^{\kappa,i}\big).
\end{equation}
The requirements on the bias and variance of the estimator
$\Upsilon_\MC$ (see Appendix~\ref{app:MonteCarlo}), together with 
Theorem~\ref{thm:Wd-triplet} and the  bounds in~\eqref{eq:summary_1} 
as well as 
Propositions~\ref{prop:logLip}, \ref{prop:barrier} \&~\ref{prop:barrier-tau},  imply 
Corollary~\ref{cor:SBG_MC}. 
By expressing $\kappa$ in terms of 
$\epsilon$ via Corollary~\ref{cor:SBG_MC} and~\eqref{def:mu},~\eqref{def:mu_tau_delta}--\eqref{eq:bound_mu_tau}, 
%Corollary~\ref{cor:algs_complexities} implies 
%the asymptotic behaviour 
%of $\kappa$ and 
the formulae for the expected computational complexity $\C_\MC(\epsilon)$ in Table~\ref{tab:MC} 
(of Subsection~\ref{subsec:complexities} above) follow. 

\begin{cor}
\label{cor:SBG_MC}
For any $\epsilon\in(0,1)$, define $\kappa$  as in (a)--(d) below 
and set
$N:=\cl{2\epsilon^{-2}\V\big[f\big(\un\chi_T^{(\kappa)}\big)\big]}$ 
as in Appendix~\ref{app:MonteCarlo}.
Then the MC estimator $\Upsilon_\MC$ 
of $\E[f(\un\chi_T)]$ has $L^2$-accuracy of level $\epsilon$ and 
expected cost $\C_\MC(\epsilon)$ bounded by a constant multiple 
of $(1+\ov\nu(\kappa)T)N$.\\ 
\nf{(a)} For any $K>0$, $g\in\Lip_K(\R^2)$ 
(resp. $g\in\locLip_K(\R^2)$) and $f:(x,z,t)\mapsto g(x,z)$, set
\begin{align*}
%\label{eq:SBG_MC_1}
	\kappa
	&:=\sup\{\kappa'\in(0,1]:2\mu_1(\kappa',T)<\epsilon/\sqrt{2}\}\\
%\label{eq:SBG_MC_2}
(\text{resp. }
\kappa
	&:=\sup\{\kappa'\in(0,1]:8K^2\mu_2(\kappa',T)
	(1+\exp(2T\ov\sigma^2_{\kappa'}))\E[\exp(2X_T)]<\epsilon^2/2\}).
\end{align*}
\nf{(b)} Pick $y<0$ and let Assumption~(\nameref{asm:(H)}) hold for 
some $C,\gamma>0$. Suppose $f:\R^3\to\R$ is given by 
$f(x,z,t)=h(x)\1_{[y,\infty)}(z)$ where $h\in\Lip_K(\R)$ and 
$0\le h\le M$ for some $K,M>0$. Then 
\begin{equation*}%\label{eq:SBG_MC_3}
\kappa := \sup\{\kappa'\in(0,1]:
	M(C\gamma)^{2/(2+\gamma)}(1+2/\gamma)
		\mu_2(\kappa',T)^{2\gamma/(2+\gamma)}
	+ K\mu_1(\kappa',T)<\epsilon/\sqrt{2}\}.
\end{equation*}
\nf{(c)} Let $\delta\in(0,2]$ satisfy Assumption~(\nameref{asm:(O)}). 
Let $f:(x,z,t)\mapsto g(t)$, where $g\in\Lip_K(\R)$, $K>0$, then 
	\begin{equation*}%\label{eq:SBG_MC_4}
\kappa := \sup\{\kappa'\in(0,1]:K\mu^\tau_*(\kappa',T)<\epsilon/\sqrt{2}\}.
\end{equation*}
\nf{(d)} Fix $s\in(0,T)$ and let $\delta\in(0,2]$ satisfy 
Assumption~(\nameref{asm:(O)}). Then there exists a constant $C>0$ 
such that for $f\in\BT_2(s,K,M)$, $K,M>0$, we have 
\begin{equation*}
%\label{eq:SBG_MC_5}
\kappa := \sup\{\kappa'\in(0,1]:
	C\sqrt{K\mu^\tau_*(\kappa',T)}<\epsilon/\sqrt{2}\}.
\end{equation*}
\end{cor}

\subsubsection{MLMC estimator} 
Let $(\kappa_j)_{j\in\N}$ (resp. $(n_j)_{j\in\N\cup\{0\}}$) be a 
decreasing (resp. increasing) sequence in $(0,1]$ (resp. $\N$) such 
that $\lim_{j\to\infty}\kappa_j=0$. 
Let $\un\chi^{0,i}\eqd\un\chi_T^{(\kappa_1)}$ and 
$(\un\chi^{j,i}_1,\un\chi^{j,i}_2)\sim \un\Pi_{n_j,T}^{\kappa_j,\kappa_{j+1}}$, $i,j\in\N$, be independent 
draws constructed by \nameref{alg:SBG}. 
Then, for the parameters $m,N_0,\ldots,N_m\in\N$, 
the MLMC estimator takes the form 
\begin{equation}\label{eq:MLMC}
\Upsilon_\ML
:=\sum_{j=0}^m\frac{1}{N_j}\sum_{i=1}^{N_j}D_j^i,
\quad\text{where}\quad 
D_j^i:=\begin{cases}
f\big(\un\chi^{j,i}_2\big)-f\big(\un\chi^{j,i}_1\big),& j\geq 1,\\
f\big(\un\chi^{0,i}\big),& j=0.
\end{cases}
\end{equation}
The bias of the MLMC estimator is equal to that of the MC estimator 
in~\eqref{eq:MC} with $\kappa=\kappa_m$. Given the sequences 
$(n_j)_{j\in\N\cup\{0\}}$ and $(\kappa_j)_{j\in\N}$, which 
determine the simulation algorithms used in 
estimator~\eqref{eq:MLMC}, Appendix~\ref{subsec:MLMC} derives 
the asymptotically optimal (as $\epsilon\searrow0$) values for the 
integers $m$ and $(N_j)_{j=0}^m$ minimising the expected 
computational complexity of~\eqref{eq:MLMC} under the  constraint 
that the $L^2$-accuracy  of $\Upsilon_\ML$ is of level $\epsilon$.
The key quantities are the bounds  $B(j)$, $V(j)$ and $C(j)$ on the bias, level variance and the computational complexity 
of \nameref{alg:SBG} at level $j$ (i.e. run with parameters 
$\kappa_j$ and $n_j$). The number of levels $m$ 
in~\eqref{eq:MLMC} is determined by the bound on the bias $B(j)$,  
while the number of samples $N_j$ used at level $j$ is given by the 
bounds on the complexity and level variances, see the formulae 
in~\eqref{eq:bias_m_value}--\eqref{eq:MLMC_const}. 
Proposition~\ref{prop:optimal_SBG_MLMC}, which is a consequence 
of Theorem~\ref{thm:Wd-triplet} and 
Propositions~\ref{prop:logLip},~\ref{prop:barrier} 
\&~\ref{prop:barrier-tau} (for bias), Theorem~\ref{thm:summary} 
(for level variance) and Corollary~\ref{cor:algs_complexities} 
(for complexity), summarises the relevant bounds $B(j)$, $V(j)$ and 
$C(j)$ established in this paper (suppressing the unknown 
constants as we are only interested in the asymptotic behaviour 
as $\epsilon\searrow0$).
%of the total computational complexity).
%is asymptotic as $\epsilon\searrow0$).
%The complexity analysis carried out in Appendix~\ref{subsec:MLMC} 
%the bounds in Theorems~\ref{thm:Wd-triplet}~\&~\ref{thm:summary}, 
%Proposition~\ref{prop:barrier} and~\eqref{eq:summary_1} imply the 
%following result. (We remove most constants for clarity and since 
%other constants involved are typically unavailable.)

\begin{prop}
\label{prop:optimal_SBG_MLMC}
Given sequences $(\kappa_j)_{j\in\N}$ and $(n_j)_{j\in\N\cup\{0\}}$ 
as above, set $C(j):=n_j+\ov\nu(\kappa_{j+1})T$. The following 
choices of functions $B$ and $V$ ensure that, for any $\epsilon>0$, 
the MLMC estimator $\Upsilon_\ML$, with integers $m$ and 
$\{N_j\}_{j=0}^m$ given 
by~\eqref{eq:bias_m_value}-\eqref{eq:MLMC_const}, has 
$L^2$-accuracy of level $\epsilon$ with complexity asymptotically 
proportional to $\C_\ML(\epsilon)
=2\epsilon^{-2}\big(\sum_{j=0}^m\sqrt{C(j)V(j)}\big)^2$.\\ 
\nf{(a)} If $K>0$, $g\in\Lip_K(\R^2)$ (resp. $g\in\locLip_K(\R^2)$) 
and $f:(x,z,t)\mapsto g(x,z)$, then for any $j\in\N$, 
\[\begin{split}
B(j) := \mu_1(\kappa_j,T)\quad&\text{and}\quad
V(j) := \sigma^2 2^{-n_j}+\ov\sigma^2_{\kappa_j},\\
(\text{resp. }
B(j) := \mu_2(\kappa_j,T)\quad&\text{and}\quad
V(j) := (2/3)^{n_j/2}\cdot\1_{\R\setminus\{0\}}(\sigma) 
	+ \ov\sigma^2_{\kappa_j} 
	+ \ov\sigma_{\kappa_j}\kappa_j).
\end{split}\]
\nf{(b)} Pick $y<0$ and let Assumption~(\nameref{asm:(H)}) hold for 
some $C,\gamma>0$. If $f\in\BT_1(y,K,M)$, $K,M>0$, then for any 
$j\in\N$,
\[
B(j) := \min\{\mu_1(\kappa_j,T)^{\gamma/(1+\gamma)},
		\mu_2(\kappa_j,T)^{2\gamma/(2+\gamma)}\}
\quad\text{and}\quad
V(j) := \sigma^{2\gamma/(2+\gamma)} 2^{-n_j\gamma/(2+\gamma)}
	+ \ov\sigma^{2\gamma/(2+\gamma)}_{\kappa_j}.
\]
\nf{(c)} Let Assumption~(\nameref{asm:(O)}) hold for some 
$\delta\in(0,2]$ and $f:(x,z,t)\mapsto g(t)$ 
for some $g\in\Lip_K(\R)$, $K>0$, then for any $j\in\N$, 
\[
B(j) := \mu^\tau_*(\kappa_j,T)
\quad\text{and}\quad
V(j) := 2^{-n_j} 
	+ \ov\sigma_{\kappa_j}^{\min\{1/2,2\delta/(2-\delta)\}}
	(1+|\log\kappa_j|\cdot\1_{\{2/5\}}(\delta)).
\]
\nf{(d)} Let $f\in\BT_2(s,K,M)$ for some $s\in(0,T)$ and $K,M\ge 0$. 
If $\delta\in(0,2]$ satisfies Assumption~(\nameref{asm:(O)}), 
then for any $j\in\N$,
\[
B(j) := \tsqrt{\mu^\tau_*(\kappa_j,T)}
\quad\text{and}\quad
V(j) := 2^{-n_j/2} 
	+ \ov\sigma_{\kappa_j}^{\min\{1/4,\delta/(2-\delta)\}}
	\big(1+\tsqrt{|\log\kappa_j|}
		\cdot\1_{\{2/5\}}(\delta)\big).
\]
\end{prop}

\begin{rem}
\label{rem:optimal_SBG_MLMC}
By~\eqref{def:mu} and~\eqref{eq:MLMC_const} we note that 
$\kappa_m$ in Proposition~\ref{prop:optimal_SBG_MLMC}(a) is 
bounded by (and typically proportional to) 
$C_0 \epsilon/|\log\epsilon|$. Moreover, if 
$\kappa_m=e^{-r(m-1)}$ for some $r>0$, then the constant $C_0$ 
does not depend on the rate $r$. A similar statement holds for (b), 
(c) and (d), see Table~\ref{tab:MC} above. 
\end{rem}

It remains to choose the parameters 
$(n_j)_{j\in\N\cup\{0\}}$ and $(\kappa_j)_{j\in\N}$
for the estimator in~\eqref{eq:MLMC}.
Since we require the bias to vanish 
geometrically fast,  we set $\kappa_j=e^{-r(j-1)}$ 
for $j\in\N$ and some $r>0$. The value of the rate $r$ 
in Theorem~\ref{thm:SBG_MLMC} below is obtained by minimising 
the multiplicative constant in the complexity $\C_\ML(\epsilon)$.
Note that $n_j$ does not affect the bias (nor the bound $B(j)$) of 
$\Upsilon_{\ML}$. By Proposition~\ref{prop:optimal_SBG_MLMC}, 
$n_j$ may be as small as a multiple of 
$\log(1/\ov\sigma_{\kappa_j}^2)$ without affecting the asymptotic 
behaviour of the level variances $V(j)$ and as large as 
$\ov\nu(\kappa_{j+1})$ without increasing the asymptotic 
behaviour of the cost of each level $C(j)$. Moreover, to ensure that 
the term $\sigma^22^{-n_j}$ in the level variances 
(see Theorem~\ref{thm:summary} above) decays geometrically, 
it suffices to let $n_j$ grow at least linearly in $j$. In short, there is 
large interval within which we may choose $n_j$ without it having any 
effect on the asymptotic performance of the MLMC estimation 
(see Theorem~\ref{thm:SBG_MLMC} below). 
The choice $n_j=n_0+\cl{\max\{j,\log^2(1+\ov\nu(\kappa_{j+1})T)\}}$, 
for $j\in\N$, 
in the numerical examples of Section~\ref{sec:numerics}
fall within this interval
(recall $\cl{x}=\inf\{j\in\Z:j\ge x\}$ for $x\in\R$). 
%where we put $n_0=5$ (this initial value is motivated 
%by the numerical results in Section~\ref{sec:numerics} above). 

%With Proposition~\ref{prop:optimal_SBG_MLMC} in hand, we may 
%finally analyse the computational complexity of the MC and MLMC 
%estimations, defined in Subsection~\ref{subsec:MC_MLMC} below. 
%The complexity analysis of the MC and MLMC estimations is 
%provided in Corollary~\ref{cor:SBG_MC} below and the next result, 
%respectively. 

\begin{thm}\label{thm:SBG_MLMC}
Suppose $q\in(0,2]$ and $c>0$ satisfy $\ov\nu(\kappa)
\le c\kappa^{-q}$ and $\ov\sigma^2_\kappa\le c\kappa^{2-q}$ 
for all $\kappa\in(0,1]$. Pick $r>0$, set $\kappa_j:=e^{-r(j-1)}$ and 
assume that $\max\{j,\log_{2/3}(\ov\sigma_{\kappa_j}^4)\}\le n_j 
\le C\ov\nu(\kappa_{j+1})$ for some $C>0$ and all sufficiently large 
$j\in\N$. Then, in cases (a)--(d) below, there exists  a constant $C_r>0$ such 
that, for all $\epsilon\in(0,1)$, the MLMC estimator $\Upsilon_\ML$ 
defined in~\eqref{eq:MLMC}, with parameters given 
by~\eqref{eq:bias_m_value}-\eqref{eq:MLMC_const}, is $L^2$-accurate at level $\epsilon$ 
with the stated expected computational complexity $\C_\ML(\epsilon)$. 
Moreover, $C_r$ is minimal for $r := (2/|a|)\log(1+|a|/q)
\cdot\1_{\R\setminus\{0\}}(a) + (2/q)\cdot \1_{\{0\}}(a)$,
	with $a\in\R$ given explicitly in each case (a)--(d).\\ 
\item[\nf{(a)}] Let $g\in\Lip_K(\R^2)\cup\locLip_K(\R^2)$ for $K>0$ 
and $f:(x,z,t)\mapsto g(x,z)$. Define $a:=2(q-1)$ and 
$b:=\1\{\sigma=0\} 
+ \1\{\sigma\ne 0\}\cdot (\1_{\{g\in\Lip_K(\R^2)\}}\cdot1/(3-q) 
\1_{\{g\notin\Lip_K(\R^2)\}}\cdot 2/(4-q))$, then 
\begin{equation}\label{eq:SBG_MLMC_1}
	\C_\ML(\epsilon)
	\le \frac{C_r}{\epsilon^{2+a^+b}}\big(1%\1_{(0,1)}(q) 
	+ \log^2\epsilon\cdot\1_{\{1\}}(q)
	+ |\log\epsilon|^{(a/2)(1+\1\{g\in\Lip_K(\R^2)\})}\cdot\1_{(1,2]}(q)\big).
\end{equation}
\item[\nf{(b)}] Let $f:(x,z,t)\mapsto g(x,z)$ where $g\in\BT_1(y,K,M)$ 
for some $y<0$ and $K,M\geq0$, such that 
Assumption~(\nameref{asm:(H)}) is satisfied by $y$ and some 
$C,\gamma>0$. 
Define $a:=2\frac{q(1+\gamma)-\gamma}{2+\gamma}
	\in(-\tfrac{2\gamma}{2+\gamma},2]$ and 
$b:=(1/2+1/\gamma)(\1\{\sigma=0\}
	+\1\{\sigma\ne 0,q<1\}\cdot 4/(9-3q)
	+\1\{\sigma\ne 0,q\ge 1\}\cdot 2/(4-q))$, then 
\begin{equation}\label{eq:SBG_MLMC_2}
	\C_\ML(\epsilon)
	\le \frac{C_r}{\epsilon^{2+a^+b}}\big(
	1
	+ \log^2\epsilon
	\cdot\1_{\{\gamma/(1+\gamma)\}}(q)
	+ |\log\epsilon|^a
	\cdot\1_{(\gamma/(1+\gamma),1)}(q)
	+ |\log\epsilon|^{a/2}
	\cdot\1_{[1,2]}(q)\big),
\end{equation}
\item[\nf{(c)}] Let $f:(x,z,t)\mapsto g(t)$ where $g\in\Lip_K(\R)$, 
$K>0$, and let Assumption~(\nameref{asm:(O)}) hold for some 
$\delta\in(0,2]$. Define 
$a:=q-(1-\tfrac{q}{2})\min\{\tfrac{1}{2},\tfrac{2\delta}{2-\delta}\}$ 
and $b:=\min\{2/\delta,\max\{3/2,1/\delta\}\}$, then 
\begin{equation}\label{eq:SBG_MLMC_3}
	\C_\ML(\epsilon)
	\le \frac{C_r}{\epsilon^{2+a^+b}}
	\begin{cases}
		1%\1_{(0,2/5)}(q)
		+\log^2\epsilon\cdot\big(
		\1_{(0,2/5)}(\delta)\1_{\{\delta\}}(q)
		+\1_{\{2\}}(\delta)\1_{\{2/5\}}(q)\big),
		%+\epsilon^{-2a}\cdot\1_{(2/5,2]}(q),
		&\delta\in(0,2]\setminus\{\tfrac{2}{5},\tfrac{2}{3}\},\\
		|\log\epsilon|\cdot\1_{(2/5,2]}(q)%\1_{(0,2]\setminus\{2/5,1/2\}}(q)
		+|\log\epsilon|^3\cdot\1_{\{2/5\}}(q),
		&\delta=2/5,\\
		|\log\epsilon|^a,
		&\delta=2/3.
	\end{cases}
\end{equation}
\nf{(d)} Fix $s\in(0,T)$ and let $\delta\in(0,2]$ satisfy 
Assumption~(\nameref{asm:(O)}). Define 
$a:=q-(1-\tfrac{q}{2})\min\{\tfrac{1}{4},\tfrac{\delta}{2-\delta}\}$ and 
$b:=\min\{4/\delta,\max\{3,2/\delta\}\}$, then for any $K,M\ge 0$ 
and $f\in\BT_2(s,K,M)$, we have 
\begin{equation}\label{eq:SBG_MLMC_4}
\C_\ML(\epsilon)
\le\frac{C_r}{\epsilon^{2+a^+b}}
\begin{cases}
	1%\1_{(0,2/9)}(q)
	+\log^2\epsilon\cdot\1_{\{2/9\}}(q),
	%+\epsilon^{-2a}\cdot\1_{(2/9,2]}(q),
	&\delta=2,\\
	1%\1_{(0,2]\setminus\{2/9,1/2\}}(q)
	+\sqrt{|\log\epsilon|}\cdot\1_{\{2/5\}}(\delta)
	+|\log\epsilon|^{a/2}\cdot\1_{\{2/3\}}(\delta),
	&\delta\in(0,2).
\end{cases}
\end{equation}
\end{thm}

\begin{rem}
\label{rem:SBG_MLMC_BG}
For most models either $\beta=\delta$ or $\sigma>0$, 
implying $a^+b\in[0,2]$ in parts (a) and (c), 
$a^+b\in[0,2(1/2+1/\gamma)]$ in part (b) (with $\gamma$ typically 
equal to $1$) and $a^+b\in[0,4]$ in part (d). 
\end{rem}

\begin{proof}[Proof of Theorem~\ref{thm:SBG_MLMC}]
\label{proof:SBG_MLMC}
Note that $\kappa_1=1$ by definition %of $(\kappa_j)_{j\in\N}$ 
independently or $r>0$, thus making both the variance $\V[D_0^i]$ 
and the cost of sampling of $D_0^i$ independent of $r$. We may 
thus ignore the $0$-th term in the bound 
$\epsilon^{-2} (\sum_{j=0}^m\sqrt{V(j)C(j)})^2$ on the complexity 
$\C_\ML(\epsilon)$  derived in Appendix~\ref{subsec:MLMC}. 
Since $m$ is given by~\eqref{eq:bias_m_value}, 
by Table~\ref{tab:MC} and Remark~\ref{rem:optimal_SBG_MLMC}, 
the function $\mu:(0,1)\mapsto(0,\infty)$ given by 
\begin{equation}
\label{def:m_bar}
\ov{m}(\epsilon)
:=\begin{cases}
(b|\log\epsilon|+c\log|\log\epsilon|)/r,
&\text{in parts (a) \& (b) and, if $\delta=\frac{2}{3}$, in parts (c) \& (d)},\\
b|\log\epsilon|/r,
&\text{in parts (c) \& (d) if $\delta\ne\frac{2}{3}$},
\end{cases}
\end{equation}
\[
\text{where}\qquad
c=\begin{cases}
1,
&\text{in parts (a) \& (c)},\\
1/2,
&\text{in parts (b) \& (d)},
\end{cases}
\]
satisfies $m\le\ov{m}(\epsilon)+C'/r$ for all 
$\epsilon\in(0,1)$ and $r>0$, where the constant $C'>0$ is 
independent of $r>0$. Thus, we need only study the growth rate of 
\[
\phi(\epsilon)
:=\sum_{j=1}^{\cl{\ov{m}(\epsilon)}}\sqrt{C(j)V(j)}
=\sum_{j=1}^{\cl{\ov{m}(\epsilon)}}
	\sqrt{(n_j+\ov\nu(\kappa_{j+1})T)V(j)},
\quad\text{as }\epsilon\to0, 
\]
because $\C_\ML(\epsilon)$ is bounded by a constant multiple of $\epsilon^{-2}\phi(\epsilon)^2$. 
In the cases where $V(j)$ contains a term of the form $e^{-sn_j}$ 
for some $s>0$ (only possible if $\sigma\neq0$), the 
product $n_je^{-sn_j}\le e^{-sn_j/2}$ vanishes geometrically fast 
since $n_j\ge j$ for all large $j$. Thus, the corresponding 
component in $\phi(\epsilon)$ is bounded as $\epsilon\to 0$ and 
may thus be ignored. 
By Proposition~\ref{prop:optimal_SBG_MLMC}, 
in all cases we may assume that 
$V(j)$ is bounded by a multiple of a power of 
$\ov\sigma_{\kappa_j}^2$
and $C(j)$
is dominated by a multiple of  $\ov\nu(\kappa_{j+1})$.
%	Under the other two restrictions on $n_j$, 
%we see that $V(j)$ and $C(j)$ are bounded by multiples of 
%$\ov\sigma_{\kappa_j}^2$ and $\ov\nu(\kappa_{j+1})$, respectively.  

Since $\ov\nu(\kappa)\le c\kappa^{-q}$ and 
$\ov\sigma^2_{\kappa}\le c\kappa^{2-q}$ for $\kappa\in(0,1]$, 
Proposition~\ref{prop:optimal_SBG_MLMC} implies 
\begin{equation*}
\phi(\epsilon)
\le K_*\begin{cases}
\sum_{j=1}^{\cl{\ov{m}(\epsilon)}}
	\sqrt{\kappa_{j+1}^{-q}\kappa_j^{2-q}}, 
& \text{in part (a)},\\
\sum_{j=1}^{\cl{\ov{m}(\epsilon)}}
	\sqrt{\kappa_{j+1}^{-q}	
	\kappa_j^{(2-q)\gamma/(2+\gamma)}}, 
& \text{in part (b)},\\
\sum_{j=1}^{\cl{\ov{m}(\epsilon)}}
\sqrt{\kappa_{j+1}^{-q}
	\kappa_j^{(2-q)\min\{1/2,2\delta/(2-\delta)\}}
	(1+|\log\kappa_j|\1_{\{2/5\}}(\delta))}, 
& \text{in part (c)},\\
\sum_{j=1}^{\cl{\ov{m}(\epsilon)}}
\sqrt{\kappa_{j+1}^{-q}	
	\kappa_j^{(2-q)\min\{1/4,\delta/(2-\delta)\}}
	(1+\sqrt{|\log\kappa_j|}\1_{\{2/5\}}(\delta))}, 
& \text{in part (d)},
\end{cases}
\end{equation*}
for some constant $K_*>0$ independent of $r$ and all 
$\epsilon\in(0,1)$, where in part (a) we used the fact that 
$\ov\sigma_{\kappa}\kappa\le\sqrt{c}\kappa^{2-q/2}$ for all 
$\kappa\in(0,1]$. 

(a) Recall that $\kappa_j=e^{-r(j-1)}$ and 
$\kappa_{j+1}=e^{-r(j-1)-r}$, implying 
\begin{equation}
\label{eq:sigma_times_nu}
\kappa_{j+1}^{-q}\kappa_j^{2-q}
=e^{rq} e^{ar(j-1)},
	\quad\text{ for all $j\in\N$, where $a=2(q-1)$},
\end{equation}

Suppose $a<0$, implying $q\in(0,1)$. 
By~\eqref{eq:sigma_times_nu}, the sequence 
$(\kappa_{j+1}^{-q}\kappa_j^{2-q})_{j\in\N}$ decays geometrically 
fast. This implies that 
$\lim_{\epsilon\downarrow0}\phi(\epsilon)<\infty$ and gives 
the desired result. Moreover, the leading constant $C_r$, 
as a function of $r$, is proportional to 
$e^{rq}/(1-e^{ar/2})^2$ as $\epsilon\downarrow0$. 
Since $a\ne 0$ for $q\in(0,1)$, the minimal value of $C_r$ is 
attained when $r=(2/|a|)\log(1+|a|/q)$. 

Suppose $a=0$, implying $q=1$. 
By~\eqref{eq:sigma_times_nu} and~\eqref{def:m_bar}, 
$\phi(\epsilon)\le K_*e^{r/2}
(b|\log\epsilon|+\log|\log\epsilon|)/r$, giving the desired result. 
As before, the leading constant $C_r$, as a function of $r$ is 
proportional to $e^r/r^2$ as $\epsilon\to 0$, attaining its minimum 
at $r=2$. 

Suppose $a>0$, implying $q\in(1,2]$. By~\eqref{eq:sigma_times_nu}
and~\eqref{def:m_bar}, it similarly follows that 
\[
\phi(\epsilon)^2
\le\frac{K_*^2e^{rq}}{(e^{ar/2}-1)^2}e^{a
		(b|\log\epsilon|+\log|\log\epsilon|)}
=\frac{K_*^2e^{rq}}{(e^{ar/2}-1)^2}\epsilon^{-ab}|\log\epsilon|^a.
\]
The corresponding result follows easily, where the leading constant 
$C_r$, as a function of $r$, is proportional to 
$e^{rq}/(e^{ar/2}-1)^2$ as $\epsilon\downarrow0$ and attains 
its minimum at $r=(2/a)\log(1+a/q)$, concluding the proof of (a). 

(b) As before, we have 
\begin{equation}
\label{eq:sigma_times_nu_2}
\kappa_{j+1}^{-q}
	\kappa_j^{(2-q)\gamma/(2+\gamma)}
=e^{rq} e^{ar(j-1)},
\quad\text{ for all $j\in\N$, 
	where $a=2\frac{q(1+\gamma)-\gamma}{2+\gamma}$}.
\end{equation} 

Suppose $a<0$, implying $q<\gamma/(1+\gamma)$. 
Then $\lim_{\epsilon\downarrow0}\phi(\epsilon)<\infty$ 
by~\eqref{eq:sigma_times_nu_2}, implying the claim. 
Moreover, $C_r$ is minimal for $r=(2/|a|)\log(1+|a|/q)$ 
as in part (a). 

Suppose $a=0$, implying $q=\gamma/(1+\gamma)$. 
Then $\phi(\epsilon)^2
\le K_*^2r^{-2}e^{rq}(b|\log\epsilon|+\log|\log\epsilon|/2)^2$, and 
the leading constant is minimised when $r=2/q=2+2/\gamma$. 

Suppose $a>0$, implying $q>\gamma/(1+\gamma)$. 
By~\eqref{eq:sigma_times_nu_2}, we have 
\[
\phi(\epsilon)^2
\le\frac{K_*^2e^{rq}}{(e^{ar/2}-1)^2}
	e^{a(b|\log\epsilon|+\log|\log\epsilon|/2)}
=\frac{K_*^2e^{rq}}{(e^{ar/2}-1)^2}\epsilon^{-ab}
	|\log\epsilon|^{a/2},
\]
and the leading constant is minimal for $r=(2/a)\log(1+a/q)$. 

In parts (c) and (d), note that $a<0$ if and only if $\delta=2$ 
(i.e. $\sigma\ne 0$). Analogous arguments as in (a) and (b), 
complete the proof of the theorem.
 \end{proof}

\appendix

\section{MC and MLMC estimators}
\subsection{Monte Carlo  estimator}
\label{app:MonteCarlo}

Consider square integrable random variables $P,P_1,P_2,\ldots$.
Let $\{P_k^i\}_{k,i\in\N}$ be independent with $P_k^i\eqd P_k$ 
for $k,i\in\N$. Suppose $|\E P-\E P_k|\leq B(k)$ for all $k\in\N$ 
and assume $C(n)$ bounds the expected computational cost of 
simulating a single value of $P_n$. Pick arbitrary $\epsilon>0$ and 
define $m:=\inf\{k\in\N:B(k)<\epsilon/\sqrt{2}\}$,
%	\quad\text{and}\quad 
	$N:=\cl{2\V[P_m]/\epsilon^2}$.
Then the Monte Carlo estimator  
\[
\hat{P}:=\frac{1}{N}\sum_{i=1}^N P_m^i\quad
\text{of $\E P$ is $L^2$-accurate at level $\epsilon$,}\quad\text{i.e.}\>\>
\E \big[(\hat{P} - \E P)^2\big]^{1/2}<\epsilon,
\]
since $\E[(\hat{P}-\E P)^2]=\V[\hat{P}]+(\E P_m-\E P)^2$ and the 
variance satisfies $\V[\hat{P}]<\epsilon^2/2$ (by the definition 
of $N$), while $(\E P_m-\E P)^2<\epsilon^2/2$ (by the definition of 
$m$). Thus, if the bound $B(m)$ on the bias is asymptotically sharp, 
the formulae for $m,N\in\N$ above result in the computational 
complexity given by 
$\C_\MC(\epsilon)=NC(m)=\cl{2\V[P_m]/\epsilon^2}C(m)$. 
Although in practice one does not have access to the variance 
$\V[P_m]$, it is typically close to $\V[P]$ (which often has an 
\textit{a priori} bound) or can be estimated via simulation.

\subsection{Multilevel Monte Carlo  estimator}
\label{subsec:MLMC}
This section is based on~\cite{Heinrich_MLMC,MR2436856}.  
Let $P,P_1,P_2,\ldots$ be square integrable random variables 
and set $P_0:=0$. Let $\{D_k^i\}_{k\in\N\cup\{0\},i\in\N}$ be independent random variables satisfying
$D_k^i\eqd D_k^1$ and $\E[D_k^i]=\E[P_{k+1}-P_k]$ 
for  any
$k\in\N\cup\{0\}$ and $i\in\N$. For $k\in\N\cup\{0\}$, 
assume that the bias and level variance satisfy $B(k)\ge|\E P-\E P_k|$ and $V(k)\ge\V[D_k^1]$ for some functions $k\mapsto B(k)$ and $k\mapsto V(k)$, respectively,  and let 
$C(k)$ bound the expected computational complexity of simulating 
a single value of $D_k^1$. For $m\in\N\cup\{0\}$ and any 
$N_0,\ldots,N_m\in\N$, the MLMC estimator 
\begin{equation*}%\label{eq:MLMC_estim}
\hat{P} 
:= \sum_{k=0}^m \frac{1}{N_k}
\sum_{i=1}^{N_k}D_k^i
%	\quad \text{is $L^2$-accurate at level $\epsilon$,}
%	\quad\E \big[(\hat{P} - \E P)^2\big]<\epsilon^2,
\end{equation*}
satisfies 
$\E[(\hat{P}-\E P)^2]=\V[\hat{P}]+(\E P_m-\E P)^2$, since 
$\E\hat{P}=\E P_m$. Thus, for any $\epsilon>0$, the inequality 
$\E\big[(\hat{P} - \E P)^2\big]<\epsilon^2$ holds if the number of 
levels in $\hat P$ equals 
\begin{equation}
\label{eq:bias_m_value}
m := \inf\{k\in\N\cup\{0\}:B(k)<\epsilon/\sqrt{2}\} 
\end{equation}
and the variance is bounded by 
$\V[\hat{P}] = \sum_{k=0}^m\V[D_k^1]/N_k
\le \sum_{k=0}^mV(k)/N_k\leq\epsilon^2/2$.
Since the computational complexity of $\hat P$, 
$\C_\ML(\epsilon)=\sum_{k=0}^m C(k)N_k$, is linear in the number 
of samples $N_k$ on each level $k$, we only require that the 
variance $\V[\hat{P}]$ be of the same order as  
$\epsilon^2/2=\sum_{k=0}^mV(k)/N_k$. Then, by the 
Cauchy-Schwartz inequality, we have
\[
\C_\ML(\epsilon) \epsilon^2/2
=\bigg(\sum_{k=1}^m C(k) N_k\bigg)
	\bigg(\sum_{k=0}^m\frac{V(k)}{N_k}\bigg)
\geq\bigg(\sum_{k=0}^m\sqrt{C(k)V(k)}\bigg)^2,
\]
where the lower bound does not depend on  $N_0,\ldots,N_m$ 
and is attained if and only if 
%Then, for every $\epsilon>0$, the constants 
%$n,N_1,\ldots,N_n\in\N$ given by 
\begin{equation}\label{eq:MLMC_const}
N_k := \cl{2\epsilon^{-2}
	\sqrt{\frac{V(k)}{C(k)}}\sum_{j=0}^m\sqrt{C(j)V(j)}}
\text{ for }k\in\{0,\ldots,n\},
%\quad%\enskip\text{and}\enskip
%n = \inf\Big\{k\in\N\cup\{0\}:B(k)<\frac{\epsilon}{\sqrt{2}}\Big\}.
\end{equation}
ensuring that the expected cost is a multiple of 
\begin{equation}
\label{eq:ML_cost}
\C_\ML(\epsilon)
=2\epsilon^{-2}\big(\sum_{k=0}^m \sqrt{C(k)V(k)}\big)^2.
\end{equation} 
Moreover, if $B$, $V$ and $C$ are asymptotically sharp, the 
formulae in~\eqref{eq:MLMC_const}, up to constants, minimise the 
expected computational complexity. Consequently, the 
computational complexity analysis of the MLMC estimator is 
reduced to the analysis of the behaviour of 
$\sum_{j=0}^m\sqrt{C(j)V(j)}$ as $\epsilon\downarrow0$.

\bibliographystyle{amsalpha}
\bibliography{References}

\thanks{
	\noindent JGC and AM are supported by The Alan Turing Institute under the EPSRC grant EP/N510129/1; 
	AM supported by EPSRC grant EP/P003818/1 and the Turing Fellowship funded by the Programme on Data-Centric Engineering of Lloyd's Register Foundation;
	JGC supported by CoNaCyT scholarship 2018-000009-01EXTF-00624 CVU 699336. 
}

\end{document}